\documentclass[11pt, reqno]{amsart}
\usepackage{amsfonts,amssymb,amscd,amstext,mathrsfs,color}
\usepackage[dvipsnames]{xcolor}
\usepackage{graphicx}
\usepackage{tikz}
\usepackage{hyperref}
\usepackage{mathtools}

\numberwithin{equation}{section}

\newtheorem{theorem}{Theorem}[section]
\newtheorem{corollary}[theorem]{Corollary}
\newtheorem{lemma}[theorem]{Lemma}
\newtheorem{proposition}[theorem]{Proposition}

\theoremstyle{definition}
\newtheorem{remark}[theorem]{Remark}

\newtheorem{definition}[theorem]{Definition}
\newtheorem{example}[theorem]{Example}

\DeclareMathOperator{\arcsinh}{arcsinh}

\newcommand{\C}{\mathbb{C}}
\newcommand{\D}{\mathbb{D}}
\newcommand{\B}{\mathbb{B}}
\newcommand{\N}{\mathbb{N}}

\newcommand{\R}{\mathbb{R}}

\newcommand{\id}{\mathrm{id}}
\renewcommand{\H}{\mathbb{H}}
\renewcommand{\Im}{{\operatorname{Im}\,}}
\renewcommand{\Re}{{\operatorname{Re}\,}}

\begin{document}
	
	\title{The Julia-Wolff-Carath\'eodory theorem in convex finite type domains}
	
	\author{Leandro Arosio$^1$ and Matteo Fiacchi$^2$}
	
	\address{Leandro Arosio,
		Dipartimento Di Matematica, Universit\`a di Roma \lq\lq Tor Vergata\rq\rq, Via Della Ricerca Scientifica 1, 00133 Roma, Italy}
	\email{arosio@mat.uniroma2.it}
\address{Matteo Fiacchi,
		Dipartimento Di Matematica, Universit\`a di Roma \lq\lq Tor Vergata\rq\rq, Via Della Ricerca Scientifica 1, 00133 Roma, Italy}
	\email{fiacchi@mat.uniroma2.it}
	
	
	
	\thanks{${}^1$  Partially supported by   INdAM, by  PRIN {\sl  Real and Complex Manifolds: Geometry and Holomorphic Dynamics} n. 2022AP8HZ9, and by the MUR Excellence Department Project MatMod@TOV
		CUP:E83C23000330006}
	
	\thanks{${}^2$ Partially supported by the European Union (ERC Advanced grant HPDR, 101053085 to Franc Forstneri\v c), the research program P1-0291 from ARIS, Republic of Slovenia, and by the MUR Excellence Department Project MatMod@TOV CUP:E83C23000330006}
	\begin{abstract}
Rudin's version of the classical   Julia-Wolff-Carath\'eodory theorem is a cornerstone of holomorphic function theory in the unit ball of $\C^d$. 
In this paper we obtain a complete generalization of Rudin's  theorem  for a holomorphic map $f\colon D\to D'$ between   convex domains of finite type. In particular, given a point   $\xi\in \partial D$ with finite dilation we show that 
the $K$-limit  of $f$ at $\xi$ exists and is a point $\eta\in \partial D'$,
and we obtain asymptotic estimates for all entries of the Jacobian matrix of the differential $df_z$ in terms of the multitypes at the points  $\xi$ and at $\eta$. We  introduce a generalization  of Bracci-Patrizio-Trapani's pluricomplex Poisson kernel which, together with the dilation at $\xi$, gives a formula for the restricted $K$-limit of the normal component of the normal derivative
$\langle df_z(n_\xi),n_\eta\rangle$.
Our principal tools are methods from Gromov hyperbolicity theory,  a scaling in the normal direction, and the strong asymptoticity of complex geodesics.
To obtain our main result we prove a conjecture by Abate on the Kobayashi type of a vector $v$, proving that it is equal to the reciprocal of the line type of $v$, and we give new extrinsic characterizations of both  $K$-convergence and restricted  convergence to a point $\xi\in \partial D$ in terms of the multitype at $\xi$.
	\end{abstract}
	
	\maketitle
	\tableofcontents

	\section{Introduction}
	
	The Julia--Wolff--Carath\'eodory theorem    is a classical result in the theory of one complex variable collecting  results obtained  by the three authors in the decade 1920-1930 \cite{Julia,Wolff,Carath}.   The starting point is what is nowadays known as the Julia lemma, which is a boundary version of the Schwarz lemma. Let $\D\subset \C$ be the unit disc, let $\xi\in \partial \D$, and let $f\colon \D\to\D$ be a holomorphic map.    Assume that, as $z\to \xi$,  the image $f(z)$ goes to  the boundary $\partial \D$ at least as fast as  $z$, meaning that the {\sl dilation}
	$$\lambda_\xi:=\liminf_{z\to\xi}\frac{1-|f(z)|}{1-|z|}$$ is finite\footnote{Notice that the dilation is never 0.}.  The Julia lemma says that every horosphere of radius $R>0$ centered at $\xi$ (which is the open disc of Euclidean radius $R/(R+1)$ internally tangent to $\xi$) is mapped by $f$ in a horosphere of radius $\lambda_\xi R$ centered at a point $\eta\in \partial \D$, and as an immediate consequence $f$ has non-tangential limit 
	$$\angle\lim_{z\to\xi} f(z)=\eta,$$
	that is, $f$ has limit $\eta$ along any sequence converging to $\xi$ inside a cone with vertex $\xi$ and aperture $<\pi$.
	One could say that the dilation is playing the role of the absolute value of the derivative of $f$ at $\xi$, even if $f$ is in general  not even continuous at $\xi$.
	The  Julia--Wolff--Carath\'eodory theorem shows that this intuition is correct.
	\begin{theorem}[Julia--Wolff--Carath\'eodory]\label{JWC1d} Let $f\colon\D\to\D$ be a holomorphic self-map and let $\xi\in\partial\D$ be a point with finite dilation $\lambda$. Then there exists  $\eta\in\partial\D$ such that 
		$$\angle\lim_{z\to\xi} f(z)=\eta,\quad \angle\lim_{z\to\xi} f'(z)=\lambda_\xi\eta\bar{\xi}.$$
	\end{theorem}
	
	In his book on function theory in the unit ball \cite{Rudin}, Rudin proved in 1980 a remarkable generalization of this theorem for holomorphic maps $f\colon \B^d\to \B^q$. There are two main issues which make this setting radically different from the case of the unit disc. 
	The first  (obvious) one, is that instead of dealing with just a derivative, we have now a whole Jacobian matrix to consider, and as we will see the asymptotics of different entries will be different.
	
	The second issue, more subtle, is that in $\B^d$ there are two relevant generalizations of the concept of non-tangential convergence to a boundary point, and both are  different   from the concept one would at first expect, that is, convergence inside  a  cone with vertex  $\xi$ and aperture $<\pi$. 
	Indeed, as first remarked in 1969 by Kor\'anyi  \cite{Kor} (see also  Stein \cite{KorSte,Stein})
	in his work on  Fatou's theorem in several complex variables, such definition of convergence is not natural and too restrictive in the unit ball, ultimately due to fact that the Kobayashi distance $k_{\B^d}$ of the ball behaves differently in the complex tangential directions w.r.t. the complex normal direction,  in other words it is {\sl anisotropic}.
	Instead, Kor\'anyi introduced a natural concept of convergence to $\xi$: a sequence  $K$-converges to $\xi$ if it is contained in a $K$-region with vertex  $\xi$.  
	Such regions have an anisotropic shape: they look like cones near $\xi$ if intersected with the complex normal line, but they are tangent to the boundary $\partial \B^d$ in the complex tangential directions. $K$-convergence of a sequence $(z_n)$  can be simply characterized by looking at the normal and complex tangent components of $z_n-\xi$, which we denote $N(z_n)$ and $T(z_n)$ respectively.
	Indeed $z_n \stackrel{K}\to \xi$ iff 
	\begin{equation}\label{cf1}
		N(z_n)=O(\delta_{\B^d}(z_n)),\quad T(z_n)=O(\delta_{\B^d}(z_n)^{\frac{1}{2}}),
	\end{equation}
	where $\delta_{\B^d}$ denotes the distance from the boundary (see Remark \ref{restrictedball} for a discussion about this characterization).
	The Julia lemma easily generalizes to this setting: if a point $\xi$ has finite dilation then there exists a point $\eta\in \partial \B^q$ such that $f(z_n)\to\eta$ for every sequence $z_n \stackrel{K}\to \xi$, that is, $f$ has $K$-limit $\eta$ at $\xi$.

	The second relevant notion of convergence is slightly more restrictive but still more general than convergence inside cones. It stems out of the Lindel\"of principle, which is a main ingredient in the proof of the Julia--Wolff--Carath\'eodory theorem. The classical Lindel\"of principle in one variable states that a bounded holomorphic function   defined on $\D$ which admits limit along a continuous curve with endpoint $\xi\in \partial \D$ actually admits non-tangential limit at $\xi$.
	Surprisingly enough, the $K$-limit is not the right generalization of the non-tangential limit to consider for a Lindel\"of principle in the ball. There are indeed bounded holomorphic functions defined on $\B^d$ which have limit on the radial segment with endpoint $\xi\in \partial \B^d$ but do not have $K$-limit at $\xi$, a simple example with $\xi=(1,0)$ is given by
	\begin{equation}\label{rudinexample}
		f(z,w)=\frac{w^2}{1-z^2},
	\end{equation}
	since $$f(t,\lambda\sqrt{1-t^2})=\lambda^2,\quad t\in [0,1),\lambda\in \D.$$
	However \v{C}irka \cite{cirka} proved that if $f$ admits limit along some ``special'' curve (e.g. the radial segment) with endpoint $\xi$, then
	there exists $L\in \C$ such that
	$f(z_n)\to L$ for every sequence  $(z_n)$ which $K'$-converge\footnote{Such sequences are also called {\sl restricted}.} to $\xi$, that is every sequence $(z_n)$ such that
	\begin{equation}\label{cf2}
		N(z_n)=O(\delta_{\B^d}(z_n)),\quad T(z_n)=o(\delta_{\B^d}(z_n)^{\frac{1}{2}}).
	\end{equation}
	One  says that $h$ has  $K'$-limit $L$ at $\xi$.
	It is easy to see that the existence of the $K$-limit implies the existence of the $K'$-limit, which in turn implies the existence of the  non-tangential limit defined with cones.

	Given two functions $f,g\colon \B^d\to \C$, with $g$ never zero, we write 
	$f=O_K(g)$  if the quotient $f/g$ is bounded on every $K$-region with vertex $\xi$, and we write $f=o_{K'}(g)$ if  $f/g$ has
	$K'$-limit 0 at $\xi$. We are now in condition to state Rudin's generalization of the Julia--Wolff--Carath\'eodory theorem, describing the asymptotics of all entries of the Jacobian matrix of $f$ and of its Jacobian determinant.
	
	\begin{theorem}[Rudin \cite{Rudin}]
		Let $f\colon\B^d\to\B^q$ be a holomorphic map and let  $\xi\in\partial\B^d$ such that
		$$\lambda_\xi:=\liminf_{z\to\xi}\frac{1-\|f(z)\|}{1-\|z\|}<+\infty.$$
		Then there exists $\eta\in \partial \B^q$ such that $f$ has $K$-limit $\eta$ at $\xi$. Moreover, if $\xi, v_1,\dots, v_{d-1}$ and $\eta,u_1,\dots, u_{q-1}$ are orthonormal  bases  of $\C^d$ and $\C^q$  respectively, then
		\begin{itemize}
			\item[(i)] $\langle df_z(\xi),\eta\rangle=O_K(1)$  and has  $K'$-limit $\lambda_\xi$ at $\xi$;
			\item[(ii)]  $\langle df_z(v_j),\eta \rangle$ is both $O_K(\delta_{\B^d}(z)^{\frac{1}{2}})$ and   $o_{K'}(\delta_{\B^d}(z)^{\frac{1}{2}})$ for all $j$;
			\item[(iii)]  $\langle df_z(\xi),u_i\rangle$ is both $O_K(\delta_{\B^d}(z)^{-\frac{1}{2}})$ and   $o_{K'}(\delta_{\B^d}(z)^{-\frac{1}{2}})$ for all $i$; 
			\item[(iv)]  $\langle df_z(v_j),u_i \rangle=O_K(1)$ for all $i,j$.
		\end{itemize}
		Finally if $d=q$, then $\det(df_z)=O_K(1).$
	\end{theorem}
	Notice that the vectors $(v_j)_{j=1}^{d-1}$ (resp. $(u_i)_{i=1}^{q-1}$) span the complex tangent space $T_\xi^\C\partial \B^d$ (resp. $T_\eta^\C\partial \B^q$).
	Rudin also shows that the exponents $\frac{1}{2}$ and $-\frac{1}{2}$ appearing in his result are sharp. One may   wonder about their geometrical significance, and we will give an answer to this question below.

	It is natural to try to generalize this result to other bounded domains in several complex variables.
	In  1990  Abate  \cite{Abatelindelof, AbateSymp}   generalized   
	Rudin's  theorem to the case of a holomorphic map $f\colon D\to D'$ between bounded strongly convex domains with $C^3$ boundary.
	A crucial point of his approach is the following: instead of working with objects defined in Euclidean terms, he redefines the dilation, horospheres, $K$-regions and $K'$-convergence in terms of the Kobayashi distance $k_D$, so as to fully exploit the fact that $f$ does not expand the Kobayashi distance.
	As an example, the dilation $\lambda_{\xi,p,p'}$ at a point $\xi\in \partial D$ is defined by
	\begin{equation}\label{defdil}\log\lambda_{\xi,p,p'}=\liminf_{z\to\xi}k_D(p,z)-k_{D'}(p',f(z)),
	\end{equation}
	where $p\in D,p'\in D'$ are given base-points.
	Notice that, if $f\colon\B^d\to \B^q$ is a holomorphic map and $\xi\in \partial \B^d$, then a  simple computation shows 
	\begin{equation}\label{dilatazionekob}
		\lambda_\xi:=\liminf_{z\to \xi}\frac{1-\|f(z)\|}{1-\|z\|}={\rm exp}\left(\liminf_{z\to\xi}{k_{\B^d}}(z,0)-k_{\B^q}(f(z),0)\right)=\lambda_{\xi,0,0}.
	\end{equation}

	Another important tool is Lempert's theory of complex geodesics, which for example yields the existence of well-defined horospheres in strongly convex domains with $C^3$ boundary.
	Other domains of $\C^d$ have been considered, such as polydiscs \cite{Abatepoly} and strongly pseudoconvex domains \cite{AbateSymp,BST}. A Julia-Wolff-Carath\'eodory  theorem was recently obtained  \cite{BrSho,AbateRaissy} for infinitesimal generators of holomorphic semigroups in the unit ball, see also \cite{Raissy}. 
	
	We are interested in the case of convex domains of finite type. 
	The D'Angelo finite type condition \cite{DAngelopaper} plays a capital role in several complex variables, due to its equivalence to the existence of subelliptic estimates for the $\bar{\partial}$-Neumann problem \cite{Kohn, Diederich-Fornaess, Catlin}.
	For a  convex domain with $C^\infty$ boundary, McNeal \cite{mcneal} (see also Boas-Straube \cite{BS}) showed that the D'Angelo finite type condition is equivalent to the finite line type condition, which we now recall.
	If $\xi\in \partial D$ and $v$ is a nonzero vector in $\C^d$, then the {\sl line type} $m_\xi(v)$ of $v$ is the order of vanishing of the defining function of $D$ along the complex affine line $\zeta\mapsto\xi+\zeta v$.
	The {\sl line type} of $\xi$ is defined as $\sup_{v\in \C^d\setminus\{0\}}m_\xi(v)$. 
	The domain $ D$ has finite line type if  there exists $L\geq 2$ such that the line type of  all $\xi\in\partial D$ is  $\leq L$.
	Every vector in $\C^d\setminus T^\C_\xi\partial D$ has line type 1, while every nonzero complex tangent vector   has line type $\geq 2$.
	In strongly convex domains every nonzero complex tangent vector has line type 2. In this paper we consider only the line type.
	
	A partial Julia--Wolff--Carath\'eodory theorem  in this setting was obtained by Abate-Tauraso \cite{AbTau} in 2002. They  studied the case of a holomorphic function $f\colon D\to \D$ defined on a bounded convex domain of finite type and with values in the unit disc. Their result needs  technical assumptions:  $D$ is assumed to be a convex domain of finite type  with $C^\infty$ boundary which is also strictly linearly convex, meaning that the complex tangent $T_\xi^\C \partial D$ intersects the boundary $\partial D$ only at the point $\xi$.   They  need to assume the existence of a complex geodesic $\varphi\colon \D\to D$ with $\varphi(1)=\xi$ such that the radial limit $\lim_{t\to 1}\varphi'(t)$ exists finitely. Notice that even in the  ``egg'' domain $\mathbb{E}_m:=\{z\in\C^2:|z_0|^2+|z_1|^{m}<1\}\subset \C^2$  
	there are complex geodesics with $\varphi(1)=(1,0)$ such that  $\|\varphi'(t)\|$ explodes as $\R\ni t\to 1^-$, see Example \ref{egggeodesics}.
	Finally their result is expressed in terms of regions which are smaller than the $K$-regions (which they call $T$-regions).
	Their result highlights however a very interesting phenomenon: as $z\to \xi$ the asymptotic behavior of the partial derivative $\frac{\partial f}{\partial v}(z)$ with $v\in \C^d\setminus\{0\}$ is controlled by the asymptotic behavior of the Kobayashi--Royden metric $\kappa_D(z,v)$, that is, if $s>0$ is such that
	$\kappa_D(z,v)=O_K(\delta_D(z)^{-s}),$ then $$\frac{\partial f}{\partial v}(z)=O_K(\delta_D(z)^{1-s}).$$
	Define the {\sl Kobayashi type} at $s_\xi(v)$ of a nonzero vector $v\in \C^d$ as the infimum of the set
	$$\{s>0: \kappa_D(z,v)=O_K(\delta_D(z)^{-s})\}.$$
	Abate--Tauraso \cite{AbTau} show that every vector in $\C^d\setminus T^\C_\xi\partial D$ has Kobayashi  type 1,  and that for all nonzero complex tangent vector $v$ one has 
	$\frac{1}{L}\leq s_\xi(v)\leq 1-\frac{1}{L},$ where $L$ is the line type of the point $\xi$.
	In \cite{Absurv} Abate conjectures that  the Kobayashi type equals the reciprocal of the line type\footnote{Abate considers the D'Angelo type in the direction $v$ but by a result of Boas--Straube \cite{BS} this coincides with the line type in the direction $v$} $m_\xi(v)$, and leaves as an open question whether the infimum is attained.
	
	In this paper we prove Abate's conjecture, we show that the infimum is attained, and  we give a complete generalization of Rudin's theorem for a holomorphic map $f\colon D\to D'$ between convex domains of finite line type, with minimal  regularity assumptions: we only require
	that there exists $L\geq 2$ such that $\partial D$ is $C^L$ and has line type $\leq L$ at all points of the boundary, and analogously for $D'$  (we actually need these assumptions only locally near $\xi$ and $\eta$ but for the sake of clarity we will ignore this here).
	As in Rudin's theorem we are able to describe the  asymptotics of all entries of the Jacobian matrix of $f$ and of its Jacobian determinant. Before stating our main result, we need to choose the right bases in which to write the Jacobian matrix.
	Given $\xi\in \partial D$, denote $n_\xi$ the outer normal versor of $\partial D$ at $\xi$. Set $v_0=n_\xi$ and complete it to an orthonormal  basis
	$(v_j)_{j=0}^{d-1}$ of $\C^d$ which respects the multitype flag introduced by Yu in \cite{Yu} (see Section \ref{scalingsect} for a detailed description of this construction).  
	For a point $\eta\in \partial D'$ we choose analogously an orthonormal basis $(u_i)_{i=0}^{q-1}$  of $\C^q$ with $u_0=n_\eta$.
	We denote the {\sl multitypes} at $\xi$ and $\eta$ as $(m_j)_{j=0}^{d-1}:=(m_\xi(v_j))$ 
	and $(n_i)_{i=0}^{q-1}:=(m_\eta(u_i))$. They coincide with the multitypes in the sense of Catlin if $\partial D$ is $C^\infty$ \cite{Yu}. 
	\begin{theorem}\label{maint}
		Let $D\subset \C^d$ and $D'\subset \C^q$ be bounded convex domains of finite type. Let $f\colon D\to D'$ be a holomorphic map.
		Assume that the dilation at $\xi$ is finite.
		Then there exists $\eta\in \partial D$ such that $f$ has $K$-limit $\eta$ at $\xi$, and  for all $i,j$,
		\begin{equation}\label{estimateimproves} \langle df_z(v_j),u_i \rangle=O_K\biggl(\delta_D(z)^{\frac{1}{n_i}-\frac{1}{m_j}}\biggr).
		\end{equation}
		Moreover
		\begin{itemize}
			\item[(i)] $\langle df_z(n_\xi),n_\eta\rangle$  has $K'$-limit $\alpha>0$ at $\xi$;
			\item[(ii)]  $\langle df_z(v_j),n_\eta \rangle=   o_{K'}\biggl(\delta_D(z)^{1-\frac{1}{m_j}}\biggr)$ for all\,  $1\leq j\leq d-1$;
			\item[(iii)]  $\langle df_z(n_\xi),u_i\rangle=o_{K'}\biggl(\delta_D(z)^{\frac{1}{n_i}-1}\biggr)$ for all\, $1\leq i\leq q-1$.
		\end{itemize}
		Finally if $d=q$, then $\det(df_z)=O_K\biggl(\delta_D(z)^{\sum_{j=0}^{d-1}\frac{1}{n_j}-\frac{1}{m_j}}\biggr).$
	\end{theorem}
	Some remarks are in order.
	We prove the following formula for the $K'$-limit $\alpha$ of $\langle df_z(n_\xi),n_\eta\rangle$:
	$$\alpha=\lambda_{\xi,p,p'}\frac{\Omega^D_\xi(p)}{\Omega^{D'}_\eta(p')},$$ where   $\Omega^D_\xi(z)\colon D\to \R$ is
	a generalization of Bracci--Patrizio--Trapani's pluricomplex Poisson kernel in strongly convex domains \cite{BP,BPT}. We define
	$$\Omega^D_\xi(z):=-\frac{1}{\angle\lim_{\zeta\to 1}\langle \varphi'(\zeta),n_\xi\rangle},$$ where $\varphi$ is a complex geodesic such that $\varphi(0)=z$ and endpoint $\xi$.
	Notice that in general the non-tangential limit of $\varphi'$ does not necessarily exist finite, but we show that the
	non-tangential limit of its normal component does.  Moreover, it is not known whether such complex geodesic is unique, so we have to prove that this is well-defined.
	
	Our result gives an interpretation for the geometrical significance of the exponents $\frac{1}{2}$ and $-\frac{1}{2}$ in Rudin's theorem: they are obtained as $1-\frac{1}{2}$ and $\frac{1}{2}-1$ respectively, due to the fact that in the ball the type of all nonzero complex tangential vectors is 2. It is also interesting to notice that for convex domains of finite type the Jacobian $\det df_z$ and the entries  $\langle df_z(v_j),u_i \rangle$ with $i,j\geq 1$ are in general no longer $O_K(1)$, instead their behavior now depends on the multitypes at $\xi$ and $\eta$. As an example,  if the harmonic mean of the multitype at $\eta$ is strictly smaller than the harmonic mean of the multitype at $\xi$, then 
	$$K\textrm{-}\lim_{z\to \xi} \det df_z=0.$$

	Another point to stress is that the choice of the bases  $(v_j)$ and $(u_i)$   is essential in this result. One could indeed expect that for any nonzero vectors $v\in \C^d,u\in \C^q$ one has
	$$\langle df_z(v),u \rangle=O_K\biggl(\delta_D(z)^{\frac{1}{m_\eta(u)}-\frac{1}{m_\xi(v)}}\biggr),$$ but this is false in general (see Example \ref{exJWC}). 
	Instead it follows from our result that $$\langle df_z(v),u \rangle=O_K\biggl(\delta_D(z)^{\frac{1}{M_\eta(u)}-\frac{1}{m_\xi(v)}}\biggr),$$
	where $M_\eta(u)$ is a dual notion of type called {\sl cotype} that we introduce.
	
	Finally, as a consequence of the minimality of our regularity  assumptions, Theorem \ref{maint} extends to the case of $C^2$ boundary Abate's Julia--Wolff--Carath\'eodory theorem for strongly convex domains \cite{Abatelindelof}. Notice that in this case  Lempert's theory of complex geodesics is not available.
	
	Our tools consist principally in  methods from Gromov hyperbolic theory and 
	a scaling technique along the normal direction, and both are new in the context of the Julia--Wolff--Carath\'eodory theorem. Gromov hyperbolicity theory was introduced in several complex variables by Balogh-Bonk \cite{BaBo}. Interactions between the two fields have been recently flourishing \cite{GS,BZ,Zimsub,BGZ,BNT,Fia,ADAF}, we cite in particular Zimmer's proof of the Gromov hyperbolicity of convex domains of finite type endowed with the Kobayashi distance \cite{Zim}, and the proof of the Julia lemma for proper geodesic Gromov hyperbolic metric spaces (\cite{AFGG}, see also \cite{AFGK}).
	The scaling technique is taken from Gaussier \cite{Gauss}, and is weighted with the multitype at the boundary point. 
	This is used in \cite{AFGG} to prove strong asymptoticity of complex geodesics and the existence of horospheres in convex domains of finite type. Due to the fact that we assume finite type only locally around $\xi\in \partial D$,$\eta\in \partial D'$ we  need to prove local versions of these statements. The study of strong asymptoticity  also allows us to perform several key computations along the normal direction instead of along a complex geodesic, bringing considerable  simplification.
	We will use the scaling method in particular to prove Abate's conjecture in \cite{Absurv} (Theorem \ref{KtypeLtype}) and the following extrinsic characterization of both $K$-convergence and $K'$-convergence, which will be important in the proof of our main theorem. Let $D\subset\C^d$ be a bounded convex domain of finite type, $\xi\in\partial D$, and let be $(v_j)_{j=0}^{d-1}$ be a  basis of $\C^d$ chosen as in Theorem \ref{maint}.
	\begin{theorem}[see Theorem \ref{extrinsicspecial}]
		Let $(z_n)$ be a sequence in $D$ converging to $\xi$. Then 
		\begin{itemize}
			\item[(i)]$z_n\stackrel{K}\to \xi$  if and only if 
			$$
			\langle z_n-\xi,v_j\rangle=O\left(\delta_D(z_n)^{1/m_j}\right) ,\quad \forall\, 0\leq j\leq d-1 ;
			$$
			\item[(ii)]
			$z_n\stackrel{K'}\to \xi$ if and only if 
			$\langle z_n-\xi,n_\xi\rangle=O(\delta_D(z_n))$, and
			$$
			\langle z_n-\xi,v_j\rangle=o\left(\delta_D(z_n)^{1/m_j}\right),\quad \forall\, 1\leq j\leq d-1.
			$$
		\end{itemize}
		
	\end{theorem}
	Compare this with \eqref{cf1} and \eqref{cf2}.
	Once again, the choice of the basis $(v_j)$ is essential in this result.
	
	\medskip {\bf Acknowledgements.} We want to thank  the anonimous Referee for the careful reading of the paper and for many useful comments.
	
	\section{Preliminaries}
	
	We use the following notations:
	\begin{itemize}
		\item We denote by $(e_j)_{j=0}^{d-1}$ the canonical basis of $\C^d$.
		\item For $u,v\in\C^d$ we denote with $\langle u,v\rangle$ the standard Hermitian product of $\C^d$ and $\|u
		\|:=\sqrt{\langle u,u\rangle}$ the  standard Euclidean norm of $\C^d$.
		\item Let $\D:=\{\zeta\in\C:|\zeta|<1\}$ and $\B^d:=\{z\in\C^d:\|z\|<1\}.$ We denote by $\H:=\{\zeta\in\C:\Re\zeta<0\}$ the left half-plane. We will use the left half-plane instead of the more conventional right (or upper) half-plane because in this way  the normal outer vector at the origin is 1, which makes the computations cleaner.
		\item Denote $$\mathscr{C}(\zeta)=\frac{\zeta-1}{\zeta+1},\quad \mathscr{C}^{-1}(\zeta)=\frac{1+\zeta}{1-\zeta}.$$
		The biholomorphism  $\mathscr{C}\colon\D\to\H$ is the Cayley transform which sends $1\in\partial\D$ to $0\in\partial\H$.
		\item If $D\subsetneq\C^d$ is a domain and $p\in\C^d$ let  
		$$\delta_D(p):=\min\{\|p-\xi\|:\xi\in\partial D\}.$$
		\item If $D\subset \C^d$ is a Kobayashi hyperbolic domain, we denote by $k_D$ its Kobayashi distance and by $\kappa_D$ its Kobayashi--Royden metric. For all $z\in D$ and $r>0$, we denote 
		$$B_D(z,r):=\{w\in D:k_D(z,w)<r\}.$$
		\item By {\sl curve} in $\C^d$ we always mean a \emph{continuous} function $\gamma\colon I\to \C^d$ defined on an interval $I$.
	\end{itemize}

	We use the following normalizations (see e.g. Kobayashi \cite{kobayascione}) for the
	Poincar\'e metric and distance in the disc $\D\subset \C$:
	$$\kappa_\D(z,v)=\frac{2|v|}{1-|z|^2},\quad k_\D(0,z)=\log\frac{1+|z|}{1-|z|}= 2 \,{\rm arctanh}(|z|).$$
	As a consequence, if  $D\subset \C^d$ is a domain,  the Kobayashi--Royden 
	pseudometric is  defined as follows: if $z\in D$, $v\in \C^d$,
	\begin{equation}\label{kr} \kappa_D(z,v)=\inf \left\{\frac{2}{R}\,:\, R>0        , \exists \varphi\in \textrm{Hol}(\D,D), \varphi(0)=z,\, \varphi'(0)=Rv   \right\}.
	\end{equation}
	 An important  advantage of this  choice of normalization  is that, as a result, the definition of a central concept such as   the dilation of a map at a boundary point is given by the same exact formula \eqref{defdil} in both the complex case and in the metric non-expanding case. In other words, the dilation of a holomorphic map $f$ coincides with the dilation of $f$ considered  as a non-expanding map between metric spaces,  see also  \cite{AFGG,AFGK}.
	Some authors normalize differently the Poincar\'e metric  in the disc, dividing it by 2. As a consequence they define the  Kobayashi--Royden pseudometric $\kappa_D(z,v)$ as  the above formula \eqref{kr} divided by 2.  Similar considerations hold for the Kobayashi distance. 
	
	We now review some basic definitions for convex domains of $\C^d$.
	\begin{definition}[$\C$-proper]
		A convex domain  $D\subset \C^d$ is {\sl $\C$-proper} if $D$ does not contain a complex affine line. 
	\end{definition}
	\begin{remark}Harris \cite{Harris} and Barth \cite{Barth} (see also \cite{Bracci-Saracco}) showed that for a convex domain $D\subset \C^d$ the following are equivalent:
		\begin{enumerate}
			\item $D$ is $\C$-proper,
			\item $D$ is Kobayashi hyperbolic,
			\item $D$ is complete Kobayashi hyperbolic.
		\end{enumerate} 
	\end{remark}

	\begin{definition}[Line type]
		\label{Def_Type}
		Let $D\subseteq \C^d$ be a convex domain. Let $L\geq 2$,  and assume that $\partial D$ is of class $C^L$ in a neighborhood of a boundary point   $\xi\in\partial D$. Let
		$r$ be a defining function of class $C^L$ for $\partial D$ in a neighborhood of $\xi$.
		We say that the point $\xi$ has {\sl finite (line)  type} $L$ if:
		\begin{enumerate}
			\item for every	$v\in\C^d\setminus\{0\}$   the function
			$\zeta \mapsto r(\xi+\zeta v)$, defined on  $\C$, has order of vanishing 
			smaller than or equal to $L$  at $0$ (we denote such number as $m_\xi(v)$ and we call it  the {\sl type} of $v$ at $\xi$), and
			\item there exists $v\in \C^{d}\setminus\{0\}$ with type $m_\xi(v)=L$.
		\end{enumerate}
		Notice that the type $m_\xi(v)$ of a vector is necessarily an even number.
		
		We say that a point $\xi\in \partial D$ has {\sl locally  finite  (line) type} if there exist $L\geq 2$ and a neighborhood $U$ of $\xi$ such that $\partial D\cap U$ is of class $C^L$ and every point in $\partial D\cap U$ has line type at most $L$.
		Finally, we say that the convex domain $D$ has {\sl finite (line) type} if there exists  $L\geq 2$ such that the boundary $\partial D$ is of class $C^L$ and every point has line type at most $L$.
	\end{definition}
	
	\begin{remark}
		At a strongly convex boundary point $\xi\in \partial D$ the line type is $2$. The same is true at a strongly pseudoconvex boundary point of $D$. 
		If  $\partial D$ is $C^\infty$  then 
		McNeal \cite{mcneal} proved that $\xi$  has finite line type $L$ if and only if it has finite D'Angelo type $L$.
		It then follows from \cite[Theorem 2 p. 131]{DAngelo} that if  $\partial D$ is $C^\infty$ then $\xi$ has finite line type if and only if it has locally finite line type.
	\end{remark}
	In this paper we consider only the line type, hence we will  call it type \emph{tout-court}.

	\section{Gromov hyperbolicity methods}
	
	Zimmer \cite{Zim} proved that if $D\subset\C^d$ is a  bounded convex domain of finite type, then the metric space $(D,k_D)$ is Gromov hyperbolic and its Gromov compactification is canonically homeomorphic to $\overline D$.  Gromov hyperbolicity is a coarse notion of negative curvature that has numerous applications in complex geometry (see for example  \cite{GS,BZ,Zimsub,BGZ,BNT,Fia,ADAF,AFGG,AFGK}).
	In this section we review the metric properties of bounded convex domains of finite type inherited from  Gromov hyperbolicity, and then we show how to adapt  them to our setting (that is, near a point of locally finite type of a $\C$-proper convex domain) via a localization result.

	\begin{definition}(Geodesics, almost-geodesics, quasi-geodesics)
		Let $D\subset \C^d$ be a Kobayashi hyperbolic domain. 
		A {\sl  geodesic} is a map $\gamma$ from an interval $I\subset \R$ to $D$ which is an isometry with respect to the Euclidean distance on $I$ and the Kobayashi distance on $D$, that is for all $s,t\in I$, $$k_D(\gamma(s),\gamma(t))=|t-s|.$$
		If the interval is $\R_{\geq 0}$ (resp. closed and bounded) we call $\gamma$ a {\sl geodesic ray}  (resp. {\sl geodesic segment}).
		
		We say that a map $\gamma\colon \R_{\geq 0}\to D$ is  an {\sl almost-geodesic ray} if for all $\varepsilon>0$ there exists $T_\varepsilon>0$ such that for all $s,t\geq T_\varepsilon$,
		$$|t-s|-\varepsilon\leq k_{D}(\gamma(s),\gamma(t))\leq |t-s|.$$
		
		Now fix $A\geq1$, $B\geq0$. 
		A (non-necessarily continuous) map $\gamma\colon \R_{\geq 0}\to D$ is  a $(A,B)$ {\sl quasi-geodesic ray} if for every $t,s\geq0$
		$$A^{-1}|t-s|-B\leq k_D(\gamma(s),\gamma(t))\leq A|t-s|+B.$$ Clearly every almost-geodesic is a $(1,B)$ quasi-geodesic for some $B\geq 0$.

		Finally, a map  $\gamma\colon [0,T)\to D$, where $T\in (0,+\infty]$, is said to have {\sl endpoint} $\gamma(T)=\xi\in \partial D$ if $\lim_{t\to T}\gamma(t)=\xi$.
	\end{definition}
	\begin{remark}
	If $D\subset \C^d$ is  a $\C$-proper convex domain and  $\xi\in\partial D$ is  a point of  locally finite type, then it follows by
	the generalized Hopf-Rinow theorem that any two points of $D$ can be connected by a geodesic segment, and it follows from Proposition \ref{convCgeo} below that, given any point  $p\in D$, there exists a geodesic ray $\gamma$ with endpoint $\xi$ and with $\gamma(0)=p$.
	\end{remark}
	
	\begin{definition}(Gromov product)
		Let $D\subset \C^d$ be a Kobayashi hyperbolic domain. For all $z,w,p\in D$ define the {\sl Gromov product} $(z|w)_p$ as follows:
		$$(z|w)_p:=\frac{1}{2}[k_D(z,p)+k_D(w,p)-k_D(z,w)].$$ 
	\end{definition}
	
	\begin{theorem}\label{shadowingft}
		Let $D\subset \C^d$ be a bounded convex domain of finite type. 
		\begin{enumerate}
			
			\item[(i)] (Boundary and Gromov product) If $\xi,\eta$ are two distinct points in $\partial D$, then 
			$$\limsup_{(z,w)\to(\xi,\eta)}(z|w)_p<+\infty.$$
			
			\item[(ii)] \label{4point}  (4-point condition) there exists $\delta\geq 0$ such that for all $z_1,z_2,z_3,p\in D$ we have
			$$(z_1|z_2)_p\geq\min\{(z_1|z_3)_p,(z_2|z_3)_p\}-\delta.$$
			
			\item[(iii)] (Asymptoticity of (1,B) quasi-geodesics)  Let $\gamma_1,\gamma_2\colon\R_{\geq 0}\to D$ be two $(1,B)$ quasi-geodesic rays with the same endpoint $\xi\in \partial D$. Then
			$$\sup_{t\geq0} k_D(\gamma_1(t),\gamma_2(t))<\infty.$$
		\end{enumerate}
		
	\end{theorem}
	\begin{proof}
		It is proved in \cite{Zim} that $(D,k_D)$ is Gromov hyperbolic and that the identity map $\id_D\colon D\to D$ extends to an homeomorphism between the Gromov compactification and $\overline D$.  
		Points (i) and (ii) are classical results for proper geodesic Gromov hyperbolic metric spaces, see e.g. \cite{BH,CDP}. For a proof of  (iii), see  \cite[Lemma 5.8]{AFGG}.
	\end{proof}

	In order to adapt these properties to our setting, we will use the following localization argument.
	\begin{remark}\label{localizationremark}
		
		Let $D\subset\C^d$ be a $\C$-proper convex domain and let $\xi\in\partial D$   be a point of  locally finite type.
		Let $L\geq 2$ be given by Definition \ref{Def_Type}. Let $W$ be a neighborhood of $\xi$ such that 
		\begin{enumerate}
			\item $\tilde{D}:=D\cap W$ is a bounded  convex domain;
			\item the boundary $\partial \tilde{D}$ is of class $C^L$;
			\item   every point of $\partial \tilde{D}$ has line type at most $L$.
		\end{enumerate}
		The domain  $\tilde D$ enjoys the visibility property (see for instance \cite{BNT} for the definition of visibility), hence it follows from  \cite[Theorem 1.4]{BNT} that  for every neighborhood  $V\subset\subset W$  of $\xi$  there exists $C\geq 0$ such that
		\begin{equation}\label{BNTloc}k_D(z,w)\leq k_{\tilde{D}}(z,w)\leq k_D(z,w)+C,\quad  \forall\, z,w\in V.\end{equation}
		
	\end{remark}

	Let $D\subset\C^d$ be a $\C$-proper convex domain. We denote by $\partial^*D$ the boundary of $D$ in the one-point compactification of $\C^d$.
	\begin{proposition}\label{shadowing}
		Let $D\subset \C^d$ be a $\C$-proper convex domain, and let $\xi\in \partial D$ be a point of locally finite type.
		\begin{itemize}
			\item[(i)]
			Let $\eta\in \partial^* D\setminus\{\xi\}$ and $p\in D$. Then 
			$$\limsup_{(z,w)\to(\xi,\eta)}(z|w)_p<+\infty.$$
			
			\item[(ii)] \label{4pointloc}  There exist $\delta'\geq 0$ and a neighborhood $V$ of $\xi$ such that  for all $z_1,z_2,z_3,p\in D\cap V$ we have
			$$(z_1|z_2)_p\geq\min\{(z_1|z_3)_p,(z_2|z_3)_p\}-\delta'.$$

			\item[(iii)]   If $\gamma_1,\gamma_2\colon \R_{\geq 0}\to D$ are two $(1,B)$ quasi-geodesic rays with endpoint $\xi$, then
			$$\sup_{t\geq0}k_D(\gamma_1(t),\gamma_2(t))<+\infty.$$
		\end{itemize}
	\end{proposition}
	\proof
	Let $W$, $\tilde{D}$ and $C$ be given by Remark \ref{localizationremark}. Let $V\subset \subset W$ be a neighborhood of $\xi$.
	
	(i)  
	Denote  by $(\cdot|\cdot)^D_p$   the Gromov product with respect to $k_D$.
	Notice that we can suppose that $p\in V$, indeed for all $z,w,p,q\in D$, by the triangle inequality we have $|(z|w)^D_p-(z|w)^D_q|\leq k_D(p,q)$.
	
	By contradiction,  assume there exist $\eta\in\partial^* D\backslash\{\xi\}$ and $z_n\to\xi$, $w_n\to\eta$ such that
	$$(z_n|w_n)^D_p\to+\infty.$$
	For all $n\geq 0$  let  $\gamma_n\colon [0,T_n]\to D$ be a  geodesic segment with $\gamma_n(0)=z_n$, $\gamma_n(T_n)=w_n$. Let $B\subset \subset V$ be a small Euclidean ball centered in $\xi$ such that  $\eta\notin B$. 
	Let $N\geq 0$ such that $z_n\in B$ for all $n\geq N$. Then for all $n\geq N$
	there exists $t_n\in(0,T_n]$ such that $\gamma_n(t_n)\in D\cap \partial B$. Set $w'_n:=\gamma_n(t_n)$. We have
	\begin{align*}
		2(z_n|w_n)^D_p&=k_D(z_n,p)+k_D(w_n,p)-k_D(z_n,w_n)\\&=k_D(z_n,p)+k_D(w_n,p)-k_D(z_n,w'_n)-k_D(w'_n,w_n)\\&\leq k_D(z_n,p)+k_D(w'_n,p)-k_D(z_n,w'_n)=2(z_n|w'_n)^D_p.
	\end{align*}
	Since $p\in V$ and $z_n,w'_n\in V$ for all $n\geq N$, by (\refeq{BNTloc}) we have
	$$(z_n|w_n')^{\tilde{D}}_p\geq (z_n|w_n')^D_p-\frac{C}{2},$$
	and thus $(z_n|w_n')^{\tilde{D}}_p\to+\infty.$
	It easily follows that the sequence $(w_n')$ is not relatively compact in $\tilde D$, and thus, up to a subsequence, we can assume that it converges to a point $\eta'\in \partial\tilde{D}\cap\partial B$. This contradicts (i) of Theorem \ref{shadowingft} because $\eta'\neq\xi$.

	(ii) By (\refeq{BNTloc}) and (ii) in Theorem \ref{shadowingft} there exists $\delta\geq0$ such that for all $z_1,z_2,z_3,p\in V$
	$$(z_1|z_2)^{D}_p\geq\min\{(z_1|z_3)^{D}_p,(z_2|z_3)^{D}_p\}-\frac{3C}{2}-\delta.$$

	(iii) Let $T\geq 0$ be such that  $\gamma_1(t),\gamma_2(t)\in V$ for all $t\geq T$.
	By \eqref{BNTloc} the curves  $\gamma_1|_{[T,+\infty)}$ and $\gamma_2|_{[T,+\infty)}$ are $(1,B+C)$ quasi-geodesics with respect to $k_{\tilde{D}}$, so we conclude with (iii) of Theorem \ref{shadowingft}.
	\endproof

	Point (i) of the previous proposition immediately implies the following.
	\begin{corollary}\label{convergestesso}
		Let $D\subset \C^d$ be a $\C$-proper convex domain, and let $\xi\in \partial D$ be a point of locally finite type. If $(z_n),(w_n)$ are sequences in $D$ such that $z_n\to\xi$ and $(k_D(z_n,w_n))_n$ is bounded, then $w_n\to\xi$.
	\end{corollary}

	We end this section with an estimate of the Kobayashi distance near the boundary which will be useful later on.
	\begin{lemma}\label{stimekob}
		Let $D\subset \C^d$ be a $\C$-proper convex domain, and let $\xi\in \partial D$ be a point of locally finite type.
		There exists a neighborhood $V$ of $\xi$ such that for all $p\in D$ there exists $c\geq 0$ such that
		$$-\log\delta_D(z)-c\leq k_D(z,p)\leq -\log\delta_D(z)+c, \quad \forall\,z\in D\cap V.$$
	\end{lemma}
	\proof
	Let $\tilde D$ and $V$ be given by Remark \ref{localizationremark}.
	By   \cite[Lemma 1.1, Lemma 1.3]{AbTau} given $z_0\in  V\cap D$ there exists a constant $c'\geq 0$ such that 
	$$-\log\delta_{\tilde D}(z)-c'\leq k_{\tilde D}(z,z_0)\leq -\log\delta_{\tilde D}(z)+c',\quad\forall z\in \tilde D.$$
	Up to taking a smaller neighborhood $V$ of $\xi$ we clearly have $\delta_D=\delta_{\tilde D}$ on $V\cap D$. 
	We conclude using \eqref{BNTloc}.
	\endproof
	
	\section{Complex geodesics}
	
	\begin{definition}[Complex geodesics]\label{Cgeo}
		Let $D\subset \C^d$ be a Kobayashi hyperbolic domain. An {\sl extremal map} is a holomorphic map $\varphi\colon \D\to D$ 
		such that there exist distinct points  $\zeta_1,\zeta_2\in \D$ satisfying
		$$k_D(\varphi(\zeta_1),\varphi(\zeta_2))=k_\D(\zeta_1,\zeta_2).$$
		A {\sl complex geodesic} is a holomorphic map $\varphi\colon \D\to D$ such that $k_D(\varphi(\zeta_1),\varphi(\zeta_2))=k_\D(\zeta_1,\zeta_2)$ for all $\zeta_1,\zeta_2\in \D$.  
		If $\varphi\colon \D\to D$  is a complex geodesic, a holomorphic map $\tilde{\rho}\colon D\to\D$  is called a {\sl left inverse} of $\varphi$ if $\tilde \rho \circ\varphi=\id_\D.$ 
		Analogous definitions hold for holomorphic maps defined in the left half-plane $\H$. In particular we say that a holomorphic map $\varphi\colon \H\to D$ is a complex geodesic if  $k_D(\varphi(\zeta_1),\varphi(\zeta_2))=k_\H(\zeta_1,\zeta_2)$ for all $\zeta_1,\zeta_2\in \H$. 
	\end{definition}
	
	\begin{proposition}\label{extrcgeo}
		Let $D\subset \C^d$ be a $\C$-proper convex domain.
		Then 
		\begin{itemize}
			\item[(i)] every extremal map $\varphi\colon \D\to D$ is a complex geodesic;
			\item[(ii)] for any two distinct points $z,w\in D$ there exists a complex geodesic $\varphi\colon \D\to D$ such that $\varphi(0)=z$ and $\varphi(\tanh(k_D(z,w)/2))=w$; 
			\item[(iii)] every complex geodesic admits a left inverse.
		\end{itemize}
	\end{proposition} 
	\proof
	Follows from \cite{Bracci-Saracco} and Proposition 11.1.4, Proposition 11.1.7 and Theorem 11.2.1 in \cite{JarPfl}.
	\endproof
	The following result was proved by Abate \cite{Abateeglistesso} in the case of bounded strongly convex domains with $C^3$ boundary. Our proof is based on a Gromov hyperbolicity method.
	\begin{proposition}\label{convCgeo}
		Let $D\subset\C^d$ be a $\C$-proper convex domain and let $\xi\in\partial D$ be a  point of locally finite  type.
		Let $p\in D$. Then there exists a  complex geodesic $\varphi\colon \D\to D$ with $\varphi(0)=p$ and 
		$\lim_{\R\ni t\to 1^-}\varphi(t)=\xi.$
	\end{proposition}
	\begin{proof}
		
		Let $p\in D$.  Let $(z_n)$ be a sequence in $D$ converging to $\xi$.
		For all $n\geq 0$ let $\varphi_n$ be a complex geodesic such that $\varphi_n(0)=p$ and $\varphi_n(t_n)=z_n$, 
		where $t_n:=\tanh(k_D(p,z_n)/2)\to 1.$
		
		Since $D$ is taut,  up to a subsequence we can assume that $\varphi_n\to \varphi$ uniformly on compact sets. Clearly $\varphi$ is a complex geodesic of $D$.
		We   show that $\lim_{t\to1^-}\varphi(t)=\xi$. Assume by contradiction that this is false, that is,  there exists a sequence  $(s_k)$ in $[0,1)$ converging to $1$ such that $\varphi(s_k)$ converges to a point  $\eta\in \partial^* D$ different from $\xi$.
		Since $\varphi_n\to\varphi$, it follows that for all $k\geq0$ we can find $n_k$ such that
		$$\|\varphi(s_k)-\varphi_{n_k}(s_k)\|<\frac{1}{k},$$
		in particular $\varphi_{n_k}(s_k)\to\eta$. 
		
		Now$$(\varphi_{n_k}(t_{n_k})|\varphi_{n_k}(s_k))_p=\min\{k_\D(0,t_{n_k}),k_\D(0,s_k)\}\to+\infty,$$
		which contradicts (i) of Proposition \ref{shadowing}.

	\end{proof}

	\begin{definition}
		Denote by $\gamma_\D\colon \R_{\geq 0}\to \D$ the geodesic ray in $\D$ with starting point $0$ and endpoint $1$, that is 
		$\gamma_\D(t)={\rm \tanh} (t/2).$
		Denote by $\gamma_\H\colon \R_{\geq 0}\to D$ the geodesic ray in $\H$ with starting point $-1$ and endpoint $0$,
		that is $\gamma_\H(t)=\mathscr{C}\circ \gamma_\D=-e^{-t}.$

		Let now $D\subset \C^d$ be a $\C$-proper convex domain.
		If  $\varphi\colon \D\to D$ (resp. $\varphi\colon \H\to D$) is a holomorphic map, we denote by $\tilde \varphi\colon \R_{\geq 0}\to D$ the curve $\tilde \varphi=\varphi\circ \gamma_\D$ (resp. $\varphi\circ \gamma_\H$).
		If the curve $\tilde \varphi$ has an endpoint $\xi\in \partial D$, then we say that $\xi$ is the   {\sl endpoint} of $\varphi$.
		Notice that if $\varphi$ is a complex geodesic, then  $\tilde \varphi$ is a  geodesic ray.

	\end{definition}
	
	\begin{definition}\label{contacttemp}
		Let $D\subset\C^d$ be a $\C$-proper convex domain. Let $\varphi\colon \D\to D$   (resp. $\varphi\colon \H\to D$)  be a holomorphic map. We say that the point $1\in \partial \D$ (resp. $0\in\partial \H$) is a   {\sl regular contact point} of $\varphi$ 
		if the curve $ \tilde\varphi$ is an almost-geodesic ray of $D$ with endpoint of locally finite type $\xi\in\partial D$.
		By Corollary \ref{convergestesso}  this implies
		that $\varphi$ has non-tangential limit $\xi$ at  $1$ (resp. at 0).
	\end{definition}
	
	\begin{remark}
		In Section \ref{Julia} we will show that  regular contact points  can be defined equivalently as 
		points of finite dilation (see Definition \ref{defregcont} and Corollary \ref{eqdefcontact}). We will also extend the definition of regular contact points  to the case of holomorphic maps between two  $\C$-proper convex domains.
		
	\end{remark}

	\begin{remark}
		Clearly if $\varphi\colon \D\to D$ (resp. $\varphi\colon \H\to D$) is a complex geodesic in a $\C$-proper convex domain $D$ with endpoint of locally finite type $\xi\in \partial D$, then 1 (resp. 0) is a  regular contact point. 
	\end{remark}
	Let $D\subset \C^d$ be a bounded strongly convex domain with $C^3$ boundary, and let $\varphi\colon \D\to D$ be a complex geodesic with endpoint $\xi\in \partial D$. Then $\varphi$ extends to $\overline \D$ as a $C^1$ map , and $\varphi'(1)$ is transversal to $T_\xi D$ (see e.g. \cite{Abatebook}). This result plays an important role in Abate's proof of the Julia--Wolff--Carath\'eodory theorem in strongly convex domains.
	If $D$ is  a bounded convex domain  of  finite type the situation  is radically different. Indeed, in this case  
	$\|\varphi'(t)\|$ may explode and $\varphi(t)$ may converge to $\xi$ tangentially when $\R\ni t\to1^-$,
	as the following example shows.
	\begin{example}[Complex geodesics of the egg domain]\label{egggeodesics}
		Given an even integer $m\geq 2$ define the {\sl egg domain} 
		\begin{equation}\label{eggdomain}
			\mathbb{E}_m:=\{(z_0,z_1)\in\C^2:|z_0|^2+|z_1|^{m}<1\}\subset \C^2.
		\end{equation}
		The points of $\partial \mathbb{E}_m$ with 
		$z_1=0$ have type $m$, and all other points have type 2.
		For all $a\in\C$ the map
		$\varphi_a\colon \D\to \mathbb{E}_m$ defined by 
		\begin{equation}\label{formulageodetiche}
			\varphi_a(\zeta):=\left(\frac{\zeta+|a|^{m}}{1+|a|^{m}},a\left(\frac{1-\zeta}{1+|a|^{m}}\right)^{2/m}\right)
		\end{equation}
		is a complex geodesic with endpoint $(1,0)$, see \cite{Huang} and \cite{JPZ}. We are using the principal value of the $m$-th root. 
	\end{example}
	However, one  can still say something about the derivative of $\varphi$: it is proved in  \cite[Lemma 5.12]{AFGG} that the normal component of  $\varphi'(z)$  admits a positive non-tangential limit as $z\to 1$. In the next result, which will be crucial in what follows,  we show  that the same is true with only a local finite type assumption  around $\xi$ and for any holomorphic map $\varphi\colon \D\to D$  with a regular contact point  at 1 and endpoint $\xi$.
	
	\begin{proposition}\label{normalderivative}
		Let $D\subset\C^d$ be a $\C$-proper convex domain and let $\xi\in\partial D$ be a point of locally finite type.
		Let $\varphi\colon \D\to D$ (resp. $\varphi\colon \H\to D$)  be a holomorphic map with a regular contact point  at 1 (resp. at 0) and endpoint $\xi$. Then the function 
		$$z\mapsto\langle \varphi'(z),n_\xi\rangle$$
		admits a positive non-tangential limit as $z\to 1$ (resp. as $z\to  0$), which we denote by $\varphi'_N(1)$ (resp.  $\varphi'_N(0)$).
	\end{proposition}
	
	\proof
	Assume first that the domain of $\varphi$ is $\H$.
	Up to a translation we can assume that $\xi=0$. Moreover up to a unitary change of coordinates of $\C^d$, we can assume that the outer normal versor $n_0$ is $e_0$, and thus the real tangent plane of $\partial D$ at $0$ is $\{{\rm Re}\,z_0=0$\}.
	By convexity of $D$ it follows that $\pi(D)\subseteq\H,$ where
	$\pi\colon\C^d\to\C$ denotes the projection to the coordinate $z_0$. Write $\varphi(\zeta)=(\varphi_0(\zeta),\varphi_1(\zeta))\in\C\times\C^{d-1}$ and notice that $\langle \varphi(\zeta),n_\xi\rangle=\varphi_0(\zeta)$.

	By \cite[Proposition 11.1]{Zim} there exists $\varepsilon>0$ such that the inner normal segment at the origin $s\mapsto(-\varepsilon e^{-s},0)$ is a $(1,\log 2)$ quasi-geodesic ray\footnote{Actually due to the different normalization of the Kobayashi distance, in \cite[Proposition 11.1]{Zim} it is shown that   $s\mapsto(-\varepsilon e^{-2s},0)$ is a $(1,\log \sqrt 2)$ quasi-geodesic} with respect to $k_D$, so by (iii) in Proposition \ref{shadowing} {there exists $M\geq 0$ such that, for all $t\geq 0$,
		\begin{align*}k_\H(\varphi_0(-e^{-t}),-e^{-t})&\leq k_\H(\varphi_0(-e^{-t}),-\varepsilon e^{-t})+k_\H(-\varepsilon e^{-t},-e^{-t})\\&\leq k_D(\tilde\varphi(t),(-\varepsilon e^{-t},0))+|\log\varepsilon|\leq M.\end{align*} 
		Hence $$\liminf_{t\to0^-}k_\H(t,-1)-k_\H(\varphi_0(t),-1)\leq M.$$
		Taking into account \eqref{dilatazionekob}, this shows that the map $\mathscr{C}^{-1}\circ \varphi_0\circ\mathscr{C}\colon \D\to \D$ has finite dilation at $1$ and has non-tangential limit 1 at 1. 
		By   the one-dimensional Julia--Wolff--Carath\'eodory theorem (Theorem \ref{JWC1d})   it follows that  $\varphi_0'(z)$ has positive non-tangential limit at $0$.

		Now assume that $\varphi\colon \D\to D$ is a holomorphic map with a regular contact point at $1$ with non-tangential limit $\xi$.
		Define
		$\psi\colon \H\to D$  as
		$\psi=\varphi\circ \mathscr{C}^{-1}$. Since $(\mathscr{C}^{-1})'(0)=2$ we have
		\begin{equation}\label{changenormal}
			\varphi'_N(1)=\angle\lim_{z\to 1} \langle \varphi'(z),n_\xi\rangle=\frac{1}{2}\angle\lim_{z\to 0} \langle \psi'(z),n_\xi\rangle=\frac{1}{2}\psi'_N(0)>0.
		\end{equation}
		
		\endproof

		\section{Multitype}
		
		We review the flag of complex subspaces of $\C^d$ introduced by Yu \cite{Yu}.
		Let $D\subset \C^d$ be a $\C$-proper convex domain and let $\xi\in \partial D$ be a point of finite type $L\geq 2$. Set $m_\xi(0):=+\infty$.
		For all integers $m\geq1$ define
		$$S_m:=\{v\in\C^d:m_\xi(v)\geq m\}.$$
		Yu shows, using the convexity of $D$, that the sets $S_m$ are complex linear subspaces.
		It immediately follows from the definition that $S_1=\C^d$,  $S_{m_2}\subseteq S_{m_1}$ for $m_1\leq m_2$ and that $S_{m}=\{0\}$ if $m>L$. Moreover $S_2$ coincides with the complex tangent subspace $T_\xi^\C\partial D$, indeed
		$$v\in S_2\iff \frac{\partial}{\partial \zeta}r(\xi+\zeta v)|_{\zeta=0}=0\iff \sum_{j=0}^{d-1}\frac{\partial r}{\partial z_j}(\xi)v_j=0.$$

		Now set $l_1=1$. For all $j\geq 1$  define recursively $l_{j}>l_{j-1}$ as the smallest integer such that $S_{l_{j}}\subsetneq S_{l_{j-1}}$. This procedure stops when we find an integer $j$ such that 
		$l_{j}=L$. We denote such integer $k$.
		We thus obtain the {\sl multitype  flag} of complex subspaces of $\C^d$:
		$$\{0\}\subsetneq S_L\subsetneq S_{l_{k-1}}\subsetneq\cdots\subsetneq S_{1}=\C^d.$$
		We say that  an orthonormal basis $v_0,\dots ,v_{d-1}$ of $\C^d$ is  a {\sl multitype basis} at $\xi$ if
		$v_0$ is the outer normal versor $n_\xi$ at $\xi$, and if the basis $v_0,\dots ,v_{d-1}$
		is adapted to  the multitype  flag (and presented in reverse order),  that is,
		for all  $j=0,\dots, k-1$ the last ${\rm dim}\, S_{l_{k-j}}$ vectors  are a basis of the subspace  $S_{l_{k-j}}$.
		For all $j=1,\dots , d-1$ set  $m_j:=m_\xi(v_j)$. 
		
		Following Yu, we call the vector  $(m_0,m_1,\cdots,m_{d-1})\in \N^d$ the {\sl (linear) multitype}  of $\partial D$ at $\xi$. If $\partial D$ is smoothly bounded, then it is proved in \cite{Yu} that $(m_0,m_1,\cdots,m_{d-1})$ coincides with the multitype in the sense  Catlin (see \cite{Cat}). Notice that $m_0=1$ and $2\leq m_j\leq L$ if $j=1,\dots,d-1$.

		By an affine unitary change of coordinates we can assume that $\xi=0$ and that  $v_j =e_j$ for $j=0,\dots, d-1$. 
		We call the new coordinates {\sl multitype coordinates}. In  multitype coordinates the outer normal  versor at the origin is $n_0=e_0$ and
		the hypersurface $\partial D$ has the following defining function in a neighborhood of the origin (see e.g. \cite{Gauss})
		\begin{equation}\label{normalform}
			r(z)=\Re z_0+H(z_1,\cdots,z_{d-1})+R(z),
		\end{equation}
		where $H\colon\C^{d-1}\longrightarrow\R$ is a convex non-negative polynomial which is 
		\begin{enumerate}
			\item {\sl  non-degenerate}, that is the set $\{H=0\}$ does not contain any complex line;
			\item {\sl weighted homogeneous} with respect to $(m_1,\dots,m_{d-1})$, i.e. for all $t>0$ and $(z_1,\cdots,z_{d-1})\in\C^{d-1}$ we have
			$$H\left(t^{1\slash m_1}z_1,\dots,t^{1\slash m_{d-1}}z_{d-1}\right)=tH(z_1,\cdots,z_{d-1}).$$
		\end{enumerate}
		
		The remainder satisfies
		$$R(z)=o\bigg(|z_0|+\sum_{j=1}^{d-1}|z_j|^{m_j}\bigg).$$
		
		We end this section introducing a dual version of the type of a vector at $\xi$. Denote by $(\C^d)^*$ the dual vector space of $\C^d$.
		\begin{definition}
			Let $D\subset \C^d$ be a $\C$-proper convex domain and let $\xi\in \partial D$ be a point of finite type $L\geq 2$.
			If $\theta\in (\C^d)^*\setminus\{0\}$, then we define its {\sl cotype} $M_\xi(\theta)$ at $\xi$ as the integer $1\leq l_{j_0}\leq L$, where $$j_0:=\max_{1\leq j\leq k}\{j:\theta\not\in {\rm Ann}(S_j)\},$$
			and $ {\rm Ann}(S_j)$ denotes the annihilator of $S_j$ in $(\C^d)^*$.
			Notice that, since $S_j$ is a decreasing sequence of subspaces, if $\theta\in  {\rm Ann}(S_j)$, then $\theta\in  {\rm Ann}(S_i)$ for all $i\geq j$.
			If $v\in \C^d\setminus\{0\}$ we
			define the {\sl cotype}  $M_\xi(v)$ of $v$ as the cotype of the linear functional $w\mapsto \langle w,v\rangle $.
		\end{definition}
		\begin{remark}
			If $v_0,\dots, v_{d-1}$ is a multitype basis at $\xi$ and $v=\sum a_j v_j\in \C^{d}\setminus\{0\}$, then 
			it is easy to see that $M_\xi(v)=\max\{m_j: a_j\neq0\}$.
			Notice that $m_\xi(v)=\min\{m_j: a_j\neq0\}$, so in particular we have $M_\xi(v_j)=m_\xi(v_j)$ for all $j=0,\dots,d-1$.

		\end{remark}
		\section{Scaling}\label{scalingsect}
		In this section we recall Gaussier's scaling in the normal direction and we use it to prove several properties of complex geodesics with endpoint of locally finite type.
		\begin{definition}
			If $H\colon\C^{d-1}\to \R$ is a convex non-negative non-degenerate weighted homogeneous polynomial 
			we call the domain $$D_H:=\{(z_0,w)\in\C\times \C^{d-1}: \Re z_0+H(w)<0\},\quad $$ a {\sl scaling model}.
			If $D$ is a $\C$-proper convex domain with a point of locally finite type at the origin  in multitype coordinates,
			then we say that the domain $D_H$, where $H$ is the polynomial given by \eqref{normalform}, is the {\sl scaling model} of $D$.
		\end{definition}
		We recall that the {\sl Hausdorff distance} between two compact sets $X,Y\subset \C^d$ is given by
		$$d_H(X,Y):=\max\left\{\sup_{x\in X}\inf_{y\in Y}\|x-y\|,\sup_{y\in Y}\inf_{x\in X}\|x-y\|\right\}.$$
		Moreover, a sequence of convex domains $(D_n)$ in $\C^d$ converges in the {\sl local Hausdorff topology} to a convex domain $D_\infty\subset \C^d$ if, for all $r>0$, 
		$$\lim_{n\to+\infty}d_H(\overline{D_n}\cap \overline{r\B^d},\overline{D_\infty}\cap \overline{r\B^d})=0.$$
		
		The following theorem is proved in \cite{Gauss}.
		\begin{theorem}
			Let $D\subset\C^d$ be a $\C$-proper convex domain  with a point of locally finite type at the origin  in multitype coordinates.
			Let $(\lambda_n)$ be a sequence in $\R_{>0}$ converging to $+\infty$ and for all $n\geq 0$ define the linear map $A_n\colon \C^d\to \C^d$ by
			$$A_n(z)=(\lambda_n z_0,\lambda_n^{1/m_1}z_1,\cdots,\lambda_n^{1/m_{d-1}}z_{d-1}).$$ Then the domain $A_nD$ converges in the local Hausdorff topology to the scaling model $D_H$ of $D$. 
		\end{theorem}
		
		Points (i) and (ii) of the following corollary are proved in \cite[Lemma 4.4]{Zim} and \cite[Theorem 9.1]{GausZim}   respectively. For lack of a reference of point (iii) we give a short proof.
		\begin{corollary}\label{zimmerconvergence}
			In the assumptions of the previous theorem,
			\begin{itemize}
				\item[(i)] every compact subset $K\subset D_H$ is eventually contained in the domain $A_nD$;
				\item[(ii)] we have $\lim_{n\to+\infty}\kappa_{A_nD}=\kappa_{D_H}$ and $\lim_{n\to+\infty}k_{A_nD}=k_{D_H}$
				uniformly on compact sets of $D_H\times\C^d$ and $D_H\times D_H$, respectively;
				\item[(iii)] if $0<r<R$ and $z\in D_H$, then eventually $B_{A_nD}(z,r)\subseteq B_{D_H}(z, R).$
			\end{itemize}
		\end{corollary}
		\proof
		Let $0<r<R$ and let $z\in D_H$. Notice that by (i) the point $z$ is eventually contained in $A_nD$.
		Fix $0<\delta<1$ such that $k_\D(0,\delta)>r$. Then by \cite[Lemma 4.6]{Zim} there exists $N_1\geq 0$ such that for all $n\geq N_1$, if $\varphi\colon \D\to A_nD$ is holomorphic with $\varphi(0)=z$, then $\varphi(\delta\D)\subset D_H$. 
		
		We claim that there exists $C> 0$ such that  for all $n\geq N_1$ we have $\overline{B_{A_nD}(z,r)}\subseteq \overline{B_{D_H}(z,C)}$. Indeed, let $n\geq N_1$, let $w\in \overline{B_{A_nD}(z,r)}$ and  consider  a complex geodesic $\varphi\colon \D\to A_nD$ such that $\varphi(0)=z$ and $\varphi(t)=w$ with $t:=\tanh(k_D(z,w)/2)$.
		The map $\varphi_\delta:\D\to \C^d$ given by $\varphi_\delta(\zeta):=\varphi(\delta\zeta)$ has image in $D_H$. Notice that $t\leq \tanh(r/2)<\delta$, so
		$$k_{D_H}(z,w)=k_{D_H}(\varphi_\delta(0),\varphi_\delta(t/\delta))\leq k_\D(0,t/\delta)\leq k_\D(0,\tanh(r/2)\slash\delta)=:C,$$
		which proves the claim. 
		The result now follows using (ii) on the compact subset $K:=\overline{B_{D_H}(z,C)}$.
		\endproof
		Notice that in  multitype coordinates the inner normal segment $\sigma\colon[t_0,1)\to\C^d$ with endpoint $\xi=0$ is given by 
			$\sigma(t)=(t-1,0)\in\C\times\C^{d-1}.$
		\begin{corollary}\label{scalingrelcomp}
			In the assumptions of the previous theorem,
			if  $(z_n)$ is a sequence in $D$ converging to 0 and there exists $M\geq 0$ such that $k_D(z_n,\sigma)\leq M$ for all $n\geq 0$, then there exist $\lambda_n\to +\infty$ and $N\geq 0$ such that the sequence
			$(A_nz_n)_{n\geq N}$ is relatively compact in $D_H$.
		\end{corollary}
		\begin{proof}
		 Let $(a_n)$ be a sequence of points in $\R_{>0}$   such that $k_D(z_n,(-a_n,0))\leq M$  for all $n\geq 0$. Clearly  $a_n\to 0$. Set $\lambda_n:=\frac{1}{a_n}.$ Then for all $n\geq 0$ we have  $A_n(-a_n,0)=(-1,0)$  and 
		$k_{A_nD}(A_nz_n,(-1,0))\leq M$. By (iii) of  Corollary  \ref{zimmerconvergence}, given  $M'>M$ there exists $N\geq 0$ such that for all $n\geq N$ we have 
		$$A_nz_n\in B_{A_nD}((-1,0),M)\subset B_{D_H}((-1,0),M')\subset\subset D_H.$$
		
		\end{proof}
		
		\begin{definition}
			Let $D\subset\C^d$ be a $\C$-proper convex domain and let $\xi$ be a point of locally finite type in $\partial D$.
			Two holomorphic maps  $\varphi, \psi\colon \D\to D$  with a regular contact point at 1 and with endpoint $\xi$ are {\sl strongly asymptotic}
			if there exists
			$T\in \R$  such that
			$$\displaystyle \lim_{t\to +\infty}k_D(\tilde \varphi(t), \tilde\psi(t+T))=0.$$	
		\end{definition}

		In the next proposition we prove that if we rescale any holomorphic map $\varphi:\H\to D$ with a  regular contact point at $0$ and endpoint at the origin we obtain in the limit the ``slice'' complex geodesic $\zeta\mapsto(\varphi'_N(0)\zeta,0)$ of the model domain $D_H$.
		This considerably generalizes  \cite[Lemma 5.13]{AFGG}, where  $D$ is assumed  to be a bounded convex domain of  finite type and where $\varphi$ is assumed to be a complex geodesic.
		
		\begin{proposition}\label{scaling1}
			Let $D\subset\C^d$ be a $\C$-proper convex domain with a point of locally finite type at the origin in multitype coordinates.
			Let $D_H$ be the scaling model of $D$.
			Let $\varphi\colon \H\to D$ be a holomorphic map with a regular contact point at $0$ and endpoint the origin.  Let $(\lambda_n)$ be a sequence in $\R_{>0}$ converging to $+\infty$. Then the sequence $(A_n\varphi(\lambda_n^{-1}\cdot))$ converges  uniformly on compact subsets to the complex geodesic $\hat\varphi\colon \H\to D_H$ defined by $$\hat{\varphi}(\zeta)=(\varphi'_N(0)\zeta,0).$$ 
		\end{proposition}
		\begin{proof}
			By the proof of \cite[Lemma 5.13]{AFGG} the sequence  $(A_n\varphi(\lambda_n^{-1}\cdot))$ converges (up to a subsequence) uniformly on compact subsets to a holomorphic
			map $\hat\varphi\colon \D\to D_H$ such that if we write $$\hat\varphi(\zeta)=(\hat\varphi_0(\zeta),\hat\varphi_1(\zeta))\in\C\times\C^{d-1},$$ we have
			$\hat\varphi_0(\zeta)=\varphi'_N(0)\zeta.$
			Fix $s,t<0$. We claim that $k_{D_H}(\hat \varphi(s),\hat \varphi(t)))=k_\H(s,t).$
			Indeed
			$$ k_{D_H}(\hat \varphi(s),\hat \varphi(t))=\lim_{n\to+\infty} k_{A_nD}(A_n\varphi(\lambda_n^{-1}s),A_n\varphi(\lambda_n^{-1}t))\\=\lim_{n\to+\infty} k_{D}(\varphi(\lambda_n^{-1}s),\varphi(\lambda_n^{-1}t)).$$
			Clearly for all $n\geq 0$ we have $k_{D}(\varphi(\lambda_n^{-1}s),\varphi(\lambda_n^{-1}t))\leq  k_\H(\lambda_n^{-1}s,\lambda_n^{-1}t)=k_\H(s,t)$.
			On the other hand, fix $\varepsilon>0$.  
			If $n$ is large enough, then $\lambda_n^{-1}s$ and $\lambda_n^{-1}t$ are $>-1,$ and thus
			$$ k_{D}(\varphi(\lambda_n^{-1}s),\varphi(\lambda_n^{-1}t))= k_D(\tilde \varphi(-\log(-\lambda_n^{-1}s)),\tilde \varphi(-\log(-\lambda_n^{-1}t))).$$
			Since $\tilde\varphi$ is an almost geodesic, if $n$ is large enough we have $$k_D(\tilde \varphi(-\log(-\lambda_n^{-1}s)),\tilde \varphi(-\log(-\lambda_n^{-1}t)))\geq \left|\log\frac{s}{t}\right|-\varepsilon=k_\H(s,t)-\varepsilon,$$
			and thus $k_{D_H}(\hat \varphi(s),\hat \varphi(t)))\geq k_\H(s,t)-\varepsilon$, and the claim is proved.
			It follows  that $\hat \varphi\colon \H\to D_H$ is an extremal map and thus by Proposition \ref{extrcgeo} the map  $\hat\varphi$ is a complex geodesic.
			It now follows from the proof of \cite[Lemma 5.13]{AFGG} that $\hat\varphi_1\equiv0$, and thus
			$\hat{\varphi}(\zeta)=(\varphi'_N(0)\zeta,0).$
		\end{proof}

		\begin{corollary}\label{approchfintyp+}
			Let $D\subset\C^d$ be a $\C$-proper convex domain and let $\xi\in\partial D$ be a point of locally finite type.
			Let $\varphi,\psi\colon\D\to D$ be two holomorphic maps with a regular contact point at 1 and endpoint $\xi$, and satisfying $\varphi'_N(1)=\psi'_N(1)$. Then	$$\angle\lim_{z\to 1}k_{D} (\varphi(z),\psi(z))=0.$$
		\end{corollary}
		
		\proof
		Consider   multitype coordinates.  We work with $\H$ instead of $\D$, the result will follow by composing with the Cayley transform.
		Let  $\varphi,\psi\colon\H\to D$ be two holomorphic maps with a regular contact point at 0 and endpoint $\xi$, and satisfying $\varphi'_N(0)=\psi'_N(0)=:c$.
		By contradiction, suppose that there exists a sequence $(\zeta_n)$ in $\H$ converging non-tangentially to 0 such that
		$$\lim_{n\to+\infty}k_{D} (\varphi(\zeta_n),\psi(\zeta_n))=M>0.$$
		For all $n\geq 0$ set $\lambda_n:=|\Re\zeta_n|^{-1}$ and  define the  maps $\varphi_n(z):=A_n\varphi(\lambda_n^{-1}z)$ and $\psi_n(z):=A_n\psi(\lambda_n^{-1}z)$. By Proposition \ref{scaling1} the two sequences converge uniformly on compact subsets to
		$$\hat{\varphi}(\zeta)=(c\zeta,0)=\hat{\psi}(\zeta).$$
		
		Since $\zeta_n\to0$ non-tangentially, there exists $A>0$ such that $|\Im\zeta_n|<A|\Re\zeta_n|$, and thus $\lambda_n\zeta_n$ converges up to a subsequence to $\hat{\zeta}\in\{\zeta\in\C: \Re\zeta=-1, |\Im\zeta|\leq A\}\subset\subset\H$. It follows that
		$$0<M=\lim_{n\to+\infty}k_{D} (\varphi(\zeta_n),\psi(\zeta_n))=\lim_{n\to+\infty}k_{A_nD} (A_n\varphi(\zeta_n),A_n\psi(\zeta_n))=k_{D_H}(\hat\varphi(\hat\zeta),\hat\psi(\hat\zeta))=0,$$
		which gives  a contradiction. 
		\endproof
		\begin{corollary}\label{approaching}
			Let $D\subset\C^d$ be a $\C$-proper convex domain and let $\xi\in\partial D$ be a point of  locally finite type. 
			Let $\varphi,\psi\colon\D\to D$ be two holomorphic maps with a regular contact point at 1 and with endpoint $\xi$. 
			Then $\varphi$ and $\psi$ are strongly asymptotic.
		\end{corollary}
		\begin{proof}
			Let $\lambda:=\varphi'_N(0)/\psi'_N(0)$, and define a holomorphic map $\theta\colon\H\to D$ by $\theta(z)=\psi(\lambda z)$.
			Then $$\tilde \theta(t)=\psi(-e^{\log\lambda-t})=\tilde\psi(t-\log\lambda),$$
			and $\theta'_N(0)=\varphi'_N(0)$, hence the result follow from Corollary \ref{approchfintyp+}.
		\end{proof}
		
		As another application of the scaling method we prove the following lemma, which will be an important tool in the proof of our main Theorem \ref{maint}. The result is interesting already in the particular case of a complex geodesic $\varphi\colon \D\to D$.
		\begin{lemma}\label{normalrestricted}
			Let $D\subset\C^d$ be a $\C$-proper convex domain and let $\xi\in\partial D$ be a point of locally finite type. 
			Let $\varphi\colon \D\to D$ be a holomorphic map with a regular contact point at 1 and with endpoint $\xi$, let $n_\xi$ denote the outer normal versor in $\xi$. Then  if we parametrize the inner normal segment at $\xi$  as
			$$\sigma(t)=\xi+(t-1)\varphi_N'(1)n_\xi, \quad t\in[t_0,1),$$
			we have
			$$\lim_{t\to1^-}k_D(\varphi(t),\sigma(t))=0.$$
		\end{lemma}
		\proof
		Consider  multitype coordinates. Then $\sigma(t)=(\varphi'_N(1)(t-1),0)$. Set $\psi:=\varphi\circ\mathscr{C}^{-1}$, $\theta(t):=\sigma\circ \mathscr{C}^{-1}(t)=\left(\frac{2t}{1-t}\varphi'_N(1),0\right)$. We will show that
		$$\lim_{t\to0^-}k_D(\psi(t),\theta(t))=0.$$
		Assume by contradiction that this is not the case. Then there exist a sequence $(t_n)$ in $\R_{<0}$ converging to 0 and $M>0$ such that
		$$\lim_{n\to+\infty}k_D(\psi(t_n),\theta(t_n))=M>0.$$
		Set $\lambda_n:=|t_n|^{-1}$. By Proposition \ref{scaling1}  the rescaled sequence $(A_n\psi(\lambda^{-1}_n z))$ converges uniformly on compact subsets to the map $\zeta\mapsto(2\varphi_N'(1)\zeta,0)$ (since by \eqref{changenormal}  $\psi'_N(0)=2\varphi_N'(1)$).
		On the other hand,  the rescaled sequence $(A_n\theta(\lambda^{-1}_n t))$ converges uniformly on compact subsets to the map $t\mapsto (2\varphi_N'(1)t,0)$, so with an argument similar to the one in Corollary \ref{approchfintyp+} we obtain a contradiction.
		\endproof

		The strong asymptoticity of complex geodesics yields  the existence of horospheres centered at points of locally finite type of $\partial D$. The proof is similar  to the proof of \cite[Theorem 3.5]{AFGG}, but due to its importance in this paper we include a shortened version for the convenience of the reader.
		
		\begin{theorem}[Existence of horospheres]\label{existencehorospheres} Let $D\subset\C^d$ be a $\C$-proper convex domain and let $\xi\in\partial D$ be a point of locally finite type. Fix a base-point  $p\in D$. Then as $w\to \xi$  the function $ k_D(\cdot,w)-k_D(w,p)$ converges uniformly on compact subsets of $D$ to a function $h_{\xi,p}$.
		\end{theorem}
		
		\proof
	
		Let $\varphi\colon \D\to D$ be a complex geodesic with endpoint $\xi$.
		Then the family of functions $$( k_D(\cdot,\varphi(t))-k_D(\varphi(t),\varphi(0)))_{t\in[0,1)}=( k_D(\cdot,\varphi(t))-t)_{t\in[0,1)}$$ is pointwise non-increasing, locally uniformly bounded  family of 1-Lipschitz functions. Hence as $t\to1^-$ it converges to a function which can be easily seen to be $1$-Lipschitz, and thus by Dini's theorem the convergence is uniform on compact subsets. It follows that 
		for every $p\in D$ the family $$( k_D(\cdot,\varphi(t))-k_D(\varphi(t),p))_{t\in[0,1)}$$ converges uniformly on compact subsets to a function $B_\varphi(z,p)$.
		 By Corollary \ref{approaching} two complex geodesics of $D$ with endpoint $\xi$ are strongly asymptotic. It easily follows (see  \cite[Proposition 3.3]{AFGG}), that  the function  $B_{\varphi}(\,\cdot\,,p)$ does not depend on the choice of $\varphi$.
		
		We show that for any sequence $w_n\to \xi$ in $D$ we have
		$$k_D(\,\cdot\,,w_n)-k_D(w_n,p)\stackrel{n\to+\infty}\longrightarrow B_{\varphi}(\,\cdot\,,p),$$ uniformly on compact subsets of $D$.
		By the Ascoli--Arzel\`a Theorem it is enough to prove that for all $z\in D$, every convergent subsequence $\left(k_D(z,w_{n_{k}})-k_D(w_{n_{k}},p)\right)$ of the sequence  $(k_D(z,w_n)-k_D(w_n,p))$ converges to  $B_{\varphi}(z,p)$.
		Let $(\varphi_k\colon\D\to D)$ be a sequence of complex geodesics with $\varphi(0)=z$ and $\varphi(r_k)=w_{n_k}$ with $r_k\in[0,1)$. Up to extracting another subsequence we can assume that $(\varphi_k)$ converges uniformly on compact subsets to a complex geodesic $\varphi$. Arguing as in   Proposition \ref{convCgeo} we see that $\varphi$ has  endpoint $\xi$.
		Similarly we obtain a sequence   $(\psi_k\colon\D\to D)$ of complex geodesics with $\psi_k(0)=p$ and $\psi_k(s_k)=w_{n_k}$ with $s_k\in[0,1)$, converging uniformly on compact subsets to a complex geodesic $\psi$ with endpoint $\xi$.
		For $a\in (-1,1)$ define the automorphism of the disc $\tau_a(z)=\frac{z-a}{1-az}$. Since $\tau'(1)=\frac{1+a}{1-a}$, there exists $a\in (-1,1)$ such that 
		$\tau_a'(1)=\varphi_N'(1)/\psi'_N(1).$
		Define $\Psi=\psi\circ \tau_a$ and $\Psi_k=\psi_k\circ \tau_a$ for all $k$.
		Then $\Psi'_N(1)=\varphi'_N(1)$ and thus by Corollary \ref{approchfintyp+} we have 
		$$\displaystyle \lim_{t\to1^-} k_D(\varphi(t), \Psi(t))=0.$$
		Fix $|a|\leq t<1$. 
		For $k$ large enough we have
		\begin{align*}
			k_D(z, w_{n_k})&=k_D (z, \varphi_k(t))+k_D(\varphi_k(t), w_{n_k}),\\[5pt]
			k_D(w_{n_k},p)&\leq k_D(w_{n_k},\varphi_k(t))+ k_D(\varphi_k(t),p),\\[5pt]
			k_D(w_{n_k},p)&= k_D(w_{n_k},\Psi_k(t))+k_D (\Psi_k(t),p)\\
			&\geq  k_D(w_{n_k},\varphi_k(t))+k_D (\varphi_k(t),p)- 2 k_D(\varphi_k(t),\Psi_k(t)),
		\end{align*}
		and thus
		\begin{align*}
			k_D(z, w_{n_k})-k_D(w_{n_k},p)&\geq k_D (z,\varphi_k(t))- k_D(\varphi_k(t),p),\\[5pt]
			k_D(z, w_{n_k})-k_D(w_{n_k},p)&\leq k_D (z,\varphi_k(t))- k_D(\varphi_k(t),p)+2 k_D(\varphi_k(t),\Psi_k(t)).
		\end{align*}
		Taking the limit as $k\to \infty$ we obtain 
		$$ k_D (z,\varphi(t))- k_D(\varphi(t),p)\leq \lim_{k\to\infty}k_D(z,w_{n_k})-k_D(w_{n_k},p)\leq  k_D (z,\varphi(t))- k_D(\varphi(t),p)+2 k_D(\varphi(t),\Psi(t)),$$
		and by letting $t\to1^-$ it follows $ \displaystyle \lim_{k\to\infty}k_D(z,w_{n_k})-k_D(w_{n_k},p)=  B_\varphi(z,p)$, as desired.

\endproof
		
\begin{definition}
Let $D\subset\C^d$ be a $\C$-proper convex domain and let $\xi\in\partial D$ be a point of locally finite type. The {\sl horosphere} centered at $\xi$ of radius $R>0$ and with base-point $p\in D$ is the level set
$$E_p(\xi,R)=\{z\in D\colon h_{\xi,p}(z)<\log R\}.$$
\end{definition}
	\begin{remark}\label{intersezchiusura} Notice that if $p,q\in D$, then 
			\begin{equation*}\label{cambiaorofunzione}
				h_{\xi,q}=h_{\xi,p}+h_{\xi,q}(p),
			\end{equation*} hence changing the base-point leaves the family of horospheres centered in $\xi$ invariant, and amounts to multiplying the radius  by a fixed constant. 
		\end{remark}
		
\begin{proposition}	
Let $D\subset\C^d$ be a $\C$-proper convex domain and let $\xi\in\partial D$ be a point of locally finite type.
Then, for all $R>0$,
\begin{equation}\label{ginevra1} \overline{E_p(\xi,R)}^*\cap\partial^* D=\{\xi\},
\end{equation}	where $\overline{E_p(\xi, R)}^*$ denotes the closure of $E_p(\xi, R)$ in the one-point compactification of $\C^d$.
Moreover,
\begin{equation}\label{ginevra2} \bigcap_{R>0}\overline{E_p(\xi,R)}^*=\{\xi\}.
\end{equation}
	\end{proposition}
\proof
The proof is similar to \cite[Proposition 6.5]{AFGG}, but we include it here for completeness.
 Fix $R>0$ and $p\in D$. Let $\gamma\colon\R_{\geq0}\to D$ be a geodesic ray with $\gamma(0)=p$ and endpoint $\xi$. Then for all $t\geq0$
$$h_{\xi,p}(\gamma(t))=-t,$$
which means that $\gamma(t)\in E_p(\xi,R)$ if $t>-\log R$, and so $\xi\in \overline{E_p(\xi,R)}^*\cap\partial^* D$.
Conversely, given a   sequence  $(z_n)$ in $D$ converging to $\eta\in\partial^* D\backslash\{\xi\}$, we want to show that $\lim_{n\to+\infty}h_{\xi,p}(z_n)=+\infty$. Let $(w_m)$ be a sequence in $D$ converging to $\xi$. By (i) of Proposition \ref{shadowing} there exist $M\geq 0$ and $N\geq 0$ such that for all $n,m\geq N$
$$(z_n|w_m)_p:=\frac{1}{2}[k_D(z_n,p)+k_D(w_m,p)-k_D(z_n,w_m)]\leq M.$$
Now for all $n\geq N$
$$h_{\xi,p}(z_n):=\lim_{m\to\infty}k_D(z_n,w_m)-k_D(w_m,p)\geq-2M+k_D(z_n,p)$$
so $\lim_{n\to+\infty}h_{\xi,p}(z_n)=+\infty$. This proves \eqref{ginevra1}, and  \eqref{ginevra2} easily follows.
\endproof
		
		\section{$K$-convergence and $K'$-convergence}
		In this section we discuss $K$-convergence and $K'$-convergence (or restricted convergence)  to a boundary point of locally finite type of a $\C$-proper convex domain. We first deal with $K$-convergence.
		Given a boundary point $\xi$ of a  bounded strongly convex domain with $C^3$ boundary $D$,
		Abate (see e.g. \cite{Abatebook}) introduced a generalization of the classical Kor\'anyi regions in the ball, namely the {\sl $K$-region} with vertex $\xi$, base-point $p\in D$ and amplitude $M>1$ as the set
		$$K_p(\xi,M):=\{z\in D\colon h_{\xi,p}(z)+k_D(z,p)<2\log M\}.$$
		Notice that $h_{\xi,p}(z)+k_D(p,z)\geq 0$.
		$K$-regions are comparable to Kor\'anyi--Stein admissible regions \cite{Stein,KorSte} and to Krantz admissible regions \cite{krantz}.
		Thanks to the existence of horospheres  (Theorem \ref{existencehorospheres}) this definition carries on verbatim to the case where $D$ is a $\C$-proper convex domain and $\xi\in \partial D$ is a point of locally finite type, and allows us to define $K$-convergence and $K$-limits in such domains.
		\begin{definition}
			Let $D\subset\C^d$ be a $\C$-proper convex domain and let $\xi\in\partial D$ be a point of locally finite type. Let $(z_n)$ be a sequence in $D$ converging to $\xi$. We say that  $(z_n)$ {\sl $K$-converges} to $\xi$ if it is contained in a $K$-region with vertex $\xi$. 
			Let $f\colon D\to \C$ be a function.
			We say that $f$ has {\sl $K$-limit} $L\in \C$ at $\xi$ if for every sequence $(z_n)$ $K$-converging to $\xi$ we have $f(z_n)\to L$. In this case we write $$K\textrm{-}\lim_{z\to\xi}f(z)=L.$$
			Finally we say that  $f$ is {\sl $K$-bounded} at $\xi$ if it is bounded in every $K$-region with vertex $\xi$.
		\end{definition}
		Notice that the previous definitions do not depend on the chosen base-point. Indeed by Remark \ref{intersezchiusura}	for all $p,q\in D$
		\begin{equation}\label{Kregpolo}
			K_q(\xi,M)\subseteq K_p\left(\xi,Me^{\frac{h_{\xi,p}(q)+k_D(q,p)}{2}}\right).
		\end{equation}

		The notion of geodesic region  was introduced in \cite{AAG} (for bounded strongly convex domains with $C^3$ boundary) as regions which are comparable to $K$-regions and  thus define the same notions of convergence, limit and boundedness, but are often much easier to work with. See also \cite{AFGG,AFGK}, where such notion is used to define a generalization of the $K$-limit  to the context of proper geodesic Gromov hyperbolic metric spaces.

		\begin{definition}[Geodesic region]
			Let $D\subset\C^d$ be a $\C$-proper convex domain and let $\xi\in\partial D$ be a point of locally finite type.
			Let $\gamma\colon \R_{\geq 0}\to D$ be a  geodesic ray with endpoint $\xi$, and let $R>0$.
			We define the {\sl geodesic region} $A(\gamma,R)$ as the open connected subset defined by
			$$A(\gamma,R)=\left\{z\in D: \inf_{t\geq0}k_D(z,\gamma(t))<R\right\}.$$
			We say that the point $\xi$ is the {\sl vertex} of $A(\gamma,R)$.
			If $\varphi$ is a complex geodesic with endpoint $\xi$, by a slight abuse of notation we denote
			by $A(\varphi,R)$ the geodesic region $A(\tilde \varphi,R),$ where
			$\tilde \varphi\colon \R_{\geq 0}\to D$ is the  geodesic ray $\tilde \varphi(t)=\varphi({\rm tanh}(t/2)).$
		\end{definition}
		
		\begin{lemma}[Properties of  geodesic regions]\label{kvshoro}
			Let $D\subset\C^d$ be a $\C$-proper convex domain and let $\xi\in\partial D$ be a point of locally finite type. Let $\gamma\colon \R_{\geq 0}\to D$ be a  geodesic ray with endpoint $\xi$.
			\begin{itemize}
				\item[(i)] If $\theta\colon \R_{\geq 0}\to D$ is another  geodesic ray with endpoint $\xi$, then there exists $S>0$ such that, for all $R>0$, 
				$$A(\gamma,R)\subseteq A(\theta, R+S), \quad A(\theta,R)\subseteq A(\psi, R+S).$$
				\item[(ii)] For all $R>0$ we have that  $\overline{A(\gamma, R)}^*\cap \partial^* D=\{\xi\}$,
				where $\overline{A(\gamma, R)}^*$ denotes the closure of $A(\gamma, R)$ in the one-point compactification of $\C^d$.
				\item[(iii)] Let $R>0$, then for all $z\in A(\gamma,R)$ $$h_{\xi,\gamma(0)}(z)<  2R-k_D(z,\gamma(0));$$
				in particular a sequence converging  to $\xi$ inside a geodesic region is eventually contained in every horosphere centered at $\xi$.
			\end{itemize}
		\end{lemma}
		\begin{proof}
			By (iii) of Proposition \ref{shadowing} there exists $S>0$ such that
			$k_D(\gamma(t),\theta(t))\leq S$ for each $t\geq0$, which implies (i). Point (ii) immediately follows from Corollary \ref{convergestesso}.
			We now prove (iii). Let $z\in A(\gamma,R)$ and consider $t^*\geq0$ such that $k_D(z,\gamma(t^*))<R$, then
			\begin{align*}h_{\xi,\gamma(0)}(z)&=\lim_{t\to1^-}k_D(z,\gamma(t))-k_D(\gamma(t),\gamma(0))\\&\leq\limsup_{t\to1^-} k_D(z,\gamma(t^*))+k_D(\gamma(t^*),\gamma(t))-k_D(\gamma(t),\gamma(0))\\&=  k_D(z,\gamma(t^*))-k_D(\gamma(t^*),\gamma(0))\\&\leq  2k_D(z,\gamma(t^*)) -k_D(z,\gamma(0))\\&<2R-k_D(z,\gamma(0)).
			\end{align*}

		\end{proof}
		
		We now show that geodesic regions and $K$-regions are comparable. The proof of \cite[Lemma 7.8]{AAG}
		does not work in this context since it is based on the Gromov hyperbolicity of the domain $D$.
		
		\begin{proposition}\label{abatek}
			Let $D\subset\C^d$ be a $\C$-proper convex domain and let $\xi\in\partial D$ be a point of locally finite type. Then every geodesic region with vertex $\xi$ is contained in a $K$-region with vertex $\xi$, and vice versa.
		\end{proposition}
		
		\proof
		Let $\gamma\colon\R_{\geq 0}\to D$ be a geodesic ray with endpoint $\xi$ and set $p=\gamma(0)$. Let $z\in A(\gamma,R)$, then by (iii) in Lemma \ref{kvshoro} we have
		$$h_{\xi,p}(z)+k_D(z,p)<2R,$$
		which implies $A(\gamma,R)\subseteq K_p(\xi,e^R)$.
		
		We now prove the converse. 
		
		Let $\delta'>0$ and $V$ be given by (ii) of Proposition \ref{shadowing}. 
		Let $\gamma\colon\R_{\geq 0}\to D$ be a geodesic ray with endpoint $\xi$ that is contained in $V$ and set $q:=\gamma(0)$.
		We will show that every $K$-region $K_q(\xi,M)$ is contained in a geodesic region with vertex $\xi$. By 
		\eqref{Kregpolo} this will prove the result.
		Notice that $K_q(\xi,M)\cap V^\complement$ is relatively compact in $D$, so we can find $R>0$ such that $ K_q(\xi,M)\cap V^\complement\subseteq A(\gamma,R)$. Now assume that $z\in K_q(\xi,M)\cap V$,  and consider $z_\gamma:=\gamma(k_D(z,q))$. Let $t>k_D(z,q)$, then by (ii) in Proposition \ref{shadowing}
		\begin{equation}\label{fourpoints}
			(z|z_\gamma)_q\geq\min\{(z|\gamma(t))_q,(z_\gamma|\gamma(t))_q\}-\delta'.
		\end{equation}
		Now
		$$(z|\gamma(t))_q=\frac{1}{2}[k_D(z,q)+k_D(\gamma(t),q)-k_D(z,\gamma(t))]\stackrel{t\to+\infty}\longrightarrow \frac{1}{2}[k_D(z,q)-h_{\xi,q}(z)]$$
		and
		$$(z_\gamma|\gamma(t))_q=\frac{1}{2}[k_D(z_\gamma,q)+k_D(\gamma(t),q)-k_D(z_\gamma,\gamma(t))]=k_D(z_\gamma,q)=k_D(z,q).$$
		Expanding the Gromov product in left-hand side of \eqref{fourpoints} we have
		$$2k_D(z,q)-k_D(z,z_\gamma)\geq\min\{k_D(z,q)-h_{\xi,q}(z),2k_D(z,q)\}-2\delta',$$
		which implies
		$$k_D(z,\gamma)\leq k_D(z,z_\gamma)\leq\max\{h_{\xi,q}(z)+k_D(z,q),0\}+2\delta'\leq2\log M+2\delta',$$
		so $K_q(\xi,M)\subseteq A(\gamma,R')$, where $R':=\max\{R,2\log M+2\delta'\}.$
		\endproof

		There are many instances in which $K$-convergence is the natural generalization to several complex variables of non-tangential convergence in the disc (see e.g. the Julia Lemma in Section \ref{Julia}). There is however a notable exception, namely the Lindel\"of principle. Indeed, one can find a bounded holomorphic function defined on the ball $\B^d$ which admits radial limit at $e_0$ but does not admit $K$-limit at $e_0$, see e.g. \eqref{rudinexample}.  This is the main reason to consider a second, more restrictive definition of convergence to a boundary point $\xi$, for which the Lindel\"of principle holds: $K'$-convergence (also called restricted convergence). For the classical  definition in the unit ball we refer to \cite{Rudin}. $K'$-convergence has been  generalized by Abate   (see e.g. \cite{Abatebook}) to the case of strongly convex domains, and by Abate--Tauraso \cite[Section 3]{AbTau} to the case of convex domains of finite type.
		In our context  we introduce the following definition.
		\begin{definition}\label{restricted}
			Let $D\subset\C^d$ be a $\C$-proper convex domain and let $\xi\in\partial D$ be a point of locally finite type. Let $\varphi\colon\D\to D$ be a complex geodesic  with endpoint $\xi$, and let  $(z_n)$ be a sequence in $D$ converging to $\xi$.
			We say that  $(z_n)$  {\sl $K'$-converges} to $\xi$ (or is  {\sl restricted}) if it $K$-converges to $\xi$ and if
			\begin{equation}\label{eqrestricted}
				\lim_{n\to+\infty}k_D(z_n,\varphi(\D))=0.
			\end{equation}
			If $f\colon D\to\C$ is a function, we say that  $f$ has  {\sl $K'$-limit} (or {\sl restricted $K$-limit})  $L\in\C$ at $\xi$  if $f(z_n)\to L$ for any restricted sequence $(z_n)$ converging to $\xi$. In this case we write
			$$K'\textrm{-}\lim_{z\to\xi}f(z)=L.$$
		\end{definition}
		
		\begin{remark}\label{rmkspecialnontg}
			Let $(z_n)$ be a restricted sequence converging to $\xi$ and let $\varphi\colon \D\to D$ be a complex geodesic with endpoint $\xi$.
			Let $(\zeta_n)$ be a sequence in $\D$ such that $k_D(z_n,\varphi(\zeta_n))\to 0$. Then $\zeta_n\to 1$ (non-tangentially), since for all $R>0$,
			\begin{equation}\label{geodesicintersection}
				\varphi (A^{\D}({\sf id}_{\D},R))=A^D(\varphi,R)\cap\varphi(\D).
			\end{equation}
			
		\end{remark}
		\begin{remark}\label{rmkAbrestricted}
			This  notion of restricted convergence  is a priori more general than the one given by Abate--Tauraso \cite[Section 3]{AbTau} for a convex domain of finite type $D\subset \C^d$. Given  a complex geodesic $\varphi$  with endpoint $\xi\in\partial D$ they choose a  left inverse $\tilde \rho\colon D\to \D$ of $\varphi$. Their definition is (equivalent to) the following:
			a sequence $(z_n)$ converging to  $\xi$ is restricted if it $K$-converges to $\xi$ and if $$k_D(\gamma(t),\varphi(\tilde \rho(z_n)))\to0.$$ Clearly  a sequence with this property satisfies (\refeq{eqrestricted}), but the converse is unclear. It is also unclear whether Abate--Tauraso's definition depends on the choice of the complex geodesic $\varphi$. The next result show that  Definition \ref{restricted} does not depend on the choice of the complex geodesic $\varphi$.
		\end{remark}
		\begin{proposition}
			Let $D\subset\C^d$ be a $\C$-proper convex domain and let $\xi\in\partial D$ be a point of locally finite type.
			Let $\varphi,\psi$ be two complex geodesics with the same endpoint $\xi$.
			Then a sequence $(z_n)$ in $D$ converging to $\xi$ is restricted w.r.t. $\varphi$ if and only if it is restricted w.r.t. $\psi$.
		\end{proposition}
		\proof
		Let $\zeta_n\in\D$ such that $k_D(z_n,\varphi(\zeta_n))\to 0$. By Remark \ref{rmkspecialnontg} $\zeta_n\to1$ non-tangentially. Up to a change of parametrization of $\psi$ we can suppose $\varphi'_N(1)=\psi'_N(1)$, so by Proposition \ref{approchfintyp+}
		$$k_D(z_n,\psi(\zeta_n))\leq k_D(z_n,\varphi(\zeta_n))+k_D(\varphi(\zeta_n),\psi(\zeta_n))\to 0.$$
		\endproof

		We now  prove  the following extrinsic characterization for $K$-convergence and for $K'$-convergence. Notice the analogy with the classic definitions in the ball (see (2.2.33) and (2.2.34) in \cite{Abatebook}).
		\begin{theorem}\label{extrinsicspecial}
			Let $D\subset\C^d$ be a $\C$-proper convex domain and let $\xi\in\partial D$ be a point of locally finite type.  Let $(v_j)_{j=0}^{d-1}$ be a multitype  basis at $\xi$, and 
			let $(z_n)$ be a sequence in $D$. Then 
			\begin{itemize}
				\item[(i)]$(z_n)$ $K$-converges to $\xi$  if and only if 
				\begin{equation}\label{extrinsicspecial1}
					\langle z_n-\xi,v_j\rangle=O\left(\delta_D(z_n)^{1/m_j}\right),\quad \forall \,0\leq j\leq d-1;
				\end{equation} 
				\item[(ii)]
				$(z_n)$  $K'$-converges to $\xi$ if and only if 
				\eqref{extrinsicspecial1} holds for $j=0$ and
				\begin{equation}\label{extrinsicspecial2}
					\langle z_n-\xi,v_j\rangle=o\left(\delta_D(z_n)^{1/m_j}\right), \quad \forall \,1\leq j\leq d-1.
				\end{equation}
			\end{itemize}
		\end{theorem}
		\proof
		For the sake of clarity, in this proof we denote the sequence  in $D$ converging to the origin by $(z^{(n)})$ instead of $(z_n)$.
		Consider  multitype coordinates at $\xi$. 
		Let  $\varphi\colon \H\to D$ be a complex geodesic with endpoint the origin. Up to a change of parametrization, we may assume $\varphi'_N(0)=1$. Set $\lambda_n:=\delta_D(z^{(n)})^{-1}$ and consider the scaling
		$$A_n(z)=(\lambda_n z_0,\lambda_n^{1/m_1}z_1,\cdots,\lambda_n^{1/m_{d-1}}z_{d-1}).$$
		Notice that in multitype coordinates the sequence $(A_nz^{(n)})$ is bounded in $\C^d$ if and only if (\refeq{extrinsicspecial1}) holds for all $j=0,\dots,d-1$.
		
		We   prove (i).
		Assume that the sequence $(A_nz^{(n)})$ is bounded  in $\C^d$. Let $r$ be the local defining function of $D$ near the origin as in (\refeq{normalform}). We have that
		$$\lambda_nr(A_n^{-1}z)\to r_{D_H}(z):=\Re z_0+H(z_1,\cdots,z_{d-1})$$
		uniformly on compact sets of $\C^d$. It is well known (see, for example, \cite[Chapter 2, Lemma 2.5]{Range}) that 
		for $z\in D$ close to the origin the defining function $r$ is bi-Lipschitz to the distance function $\delta_D$, that is there exist $0<c<C$ such that 
		$$-C\delta_D(z)\leq r(z)\leq -c\delta_D(z).$$
		It follows that 
		$$\limsup_{n\to+\infty}r_{D_H}(A_nz^{(n)})=\limsup_{n\to+\infty}\lambda_nr(A_n^{-1}A_nz^{(n)})=\limsup_{n\to+\infty}\frac{r(z^{(n)})}{\delta_D(z^{(n)})}\leq-c<0,$$ which implies that the limit set of the sequence $(A_nz^{(n)})$ is contained in a compact subset $K$ of $D_H$. By Proposition \ref{scaling1} the sequence $(A_n\varphi(\lambda_n^{-1}\cdot))$ converges uniformly on compact subsets to  $\hat\varphi(\zeta)=(\zeta,0)$, so by Corollary \ref{zimmerconvergence} we have
		$$\limsup_{n\to+\infty}k_D(z^{(n)},\varphi(-\lambda_n^{-1}))=\limsup_{n\to+\infty}k_{A_nD}(A_nz^{(n)},A_n\varphi(-\lambda_n^{-1}))\leq \max_{z\in K}k_{D_H}(z,(-1,0))<+\infty,$$
		which means that the sequence $(z^{(n)})$ is contained in a geodesic region. 
		
		Conversely,  assume that the sequence $(z^{(n)})$ $K$-converges to $\xi$.  Let $(t_n)$  be a sequence in  $(0,1]$ converging to 0 such that
		$k_D(z^{(n)},\varphi(-t_n))\leq M$ for all $n\geq0$. We claim that there are two constants $0<c_1<C_1$ such that \begin{equation}\label{Mercereq}
			c_1\leq \lambda_nt_n\leq C_1. 
		\end{equation}
		Indeed, by Lemma \ref{stimekob} we can find $c_2\geq 0$ such that
		$$c_2\geq|k_D(\varphi(-t_n),\varphi(-1))+\log\delta_D(\varphi(-t_n))|=|k_\H(-t_n,-1)+\log\delta_D(\varphi(-t_n))|=\left|\log\frac{\delta_D(\varphi(-t_n))}{t_n}\right|.$$
		By Proposition 2.4 in \cite{Mercer} we have 
		$k_D(p,q)\geq \left|\log\frac{\delta_D(p)}{\delta_D(q)}\right|$   for all $p,q\in D$, so
		$$\left|\log\frac{t_n}{\delta_D(z^{(n)})}\right|\leq\left|\log\frac{t_n}{\delta_D(\varphi(-t_n))}\right|+\left|\log\frac{\delta_D(\varphi(-t_n))}{\delta_D(z^{(n)})}\right|\leq c_2+  k_D(\varphi(-t_n),z^{(n)})\leq c_2+ M,$$
		which proves the claim.
		Thus the sequence $(-\lambda_nt_n)$ is contained in $[-C_1,-c_1]\subset\subset \H$. Hence the limit set of the sequence $$A_n\varphi(-t_n)=A_n\varphi(\lambda_n^{-1}(-\lambda_nt_n))$$ 
		is contained in $[-C_1,-c_1]\times \{0\}\subset\subset  D_H$. Recalling that $k_D(z^{(n)},\varphi(-t_n))\leq M$, by (iii) in Corollary \ref{zimmerconvergence} 
		it follows that the limit set of the sequence  $(A_nz^{(n)})$ is relatively compact in $D_H$. Hence $(A_nz^{(n)})$ is bounded in $\C^d$.
	
	We now prove (ii).
	Assume that the sequence $(z^{(n)})$ $K'$-converges to $\xi$. By definition there exists $(\zeta_n)$ in $\H$ such that $k_D(z^{(n)},\varphi( \zeta_n))\to0$. 
	By Remark \ref{rmkspecialnontg} $\zeta_n\to 0$ non-tangentially, so there exists $A>0$ such that $|\Im\zeta_n|\leq A|\Re\zeta_n|$. It follows that
	$$k_D(z^{(n)},\varphi(\Re\zeta_n))\leq k_D(z^{(n)},\varphi(\zeta_n))+k_{\H}(\zeta_n,\Re \zeta_n)\leq k_D(z^{(n)},\varphi(\zeta_n))+2\arcsinh\frac{A}{2},$$ which is bounded. Arguing as above (setting $t_n:=-\Re\zeta_n$) we get that there exists $C_1>0$ such that 
	$$c_1\leq-\lambda_n\Re\zeta_n\leq C_1,$$ which implies that  the sequence $(\lambda_n\zeta_n)$ is contained in a compact subset of  $\H$.
	Let $(n_k)$ be a subsequence such that $(A_{n_k}z^{(n_k)})$ converges to a point $\hat z\in D_H$ and   $(\lambda_{n_k}\zeta_{n_k})$ converges to a point $\hat\zeta\in \H.$
	We have
	$$0=\lim_{k\to+\infty}k_D(z^{(n_k)},\varphi(\zeta_{n_k}))=\lim_{k\to+\infty}k_{A_{n_k}D}(A_{n_k}z^{(n_k)},A_{n_k}\varphi(\zeta_{n_k}))=k_{D_H}(\hat z,(\hat\zeta,0)),$$
	thus $A_{n_k}z^{(n_k)}\to (\hat\zeta,0)$, which means that 
	$\lambda_{n_k}^{1/m_j}z^{(n_k)}_j\to0$ for all $j=1,\cdots,d-1$,
	which implies (\refeq{extrinsicspecial2}).
	
	Conversely, assume (\refeq{extrinsicspecial2}), and we show that  $k_D(z^{(n)},\varphi(z^{(n)}_0))\to 0$.
	{By (\refeq{extrinsicspecial1}) the sequence $(\lambda_n z_0^{(n)})$ is bounded in $\C$. Moreover,  from   the convexity of $D$, we have $\lambda_n:=\delta_D(z^{(n)})\leq -\Re z_0^{(n)}$, so $(\lambda_nz^{(n)}_0)$ is contained in a compact subset of $\H$. Let $(n_k)$ be a subsequence such that $(\lambda_{n_k}z^{(n_k)}_0)$ converges to a point  $\hat\zeta\in \H.$
		By (\refeq{extrinsicspecial2}) we have  $A_{n_k}z^{(n_k)}\to (\hat\zeta,0)$, and
		by Proposition \ref{scaling1} we have
		$$A_{n_k}\varphi(\lambda_{n_k}^{-1}\lambda_{n_k}z_0^{(n_k)})\to (\hat\zeta,0).$$
		Thus
		$$k_D(z^{(n_k)},\varphi(z^{(n_k)}_0))=k_{A_{n_k}D}(A_{n_k}z^{(n_k)},A_{n_k}\varphi(z^{(n_k)}_0))\to 0.$$

		\endproof
		
		\begin{remark}\label{reformulation}
			Notice that the (i) and (ii) in previous theorem immediately imply the following:
			\begin{itemize}
				\item[(i)]  $\langle z-\xi,v_j\rangle=O_K\left(\delta_D(z)^{1/m_j}\right)$ for all  $0\leq j\leq d-1$;
				\item[(ii)] $\langle z-\xi,v_j\rangle=o_{K'}\left(\delta_D(z)^{1/m_j}\right)$ for all $1\leq j\leq d-1$.
			\end{itemize}
		\end{remark}

		\begin{remark}\label{typevscotype} It is natural to wonder whether, if a sequence $(z^{(n)})$ $K$-converges to $\xi$, one has for all $v\in \C^d\setminus\{0\}$,
			$$\langle z^{(n)}-\xi,v\rangle=O\left(\delta_D(z^{(n)})^{1/m_\xi(v)}\right).$$
			This turns out to be false in general, as the following example shows: 
			let $m\geq2$ be an even integer and  let $\mathbb{E}_m\subset \C^2$ be the egg domain defined in Example \ref{egggeodesics}. Let $\xi:=(1,0)\in\partial\mathbb{E}_m$. Fix $\lambda\in\D^{*}$ and consider the curve $\gamma\colon [0,1)\to\mathbb{E}_m $
			with endpoint $\xi$ defined as  $\gamma(t)=(t,\lambda\sqrt[m]{1-t^2}),$ which satisfies
			$$\delta_{\mathbb{E}_m}(\gamma(t))\approx1-t.$$
			Let $t_n\nearrow 1$ and set $z^{(n)}:=\gamma(t_n).$ By the previous theorem   $z^{(n)}$ converges to $\xi$ inside a $K$-region.
			Notice that the outer normal versor $n_\xi$ at $\xi$ is $(1,0)$ and that the complex tangent at $\xi$ is $T_\xi^\C\partial \mathbb{E}_m=\{z_0=0\}$.
			Now let $v\in\C^2\setminus\{z_0z_1=0\}$. The type of $v$ is  $m_\xi(v)=1$ while its cotype  is $M_\xi(v)=m$.
			A simple computation shows that
			$\langle z^{(n)}-\xi,v \rangle$ is  $O((1-t_n)^{1/m})$ but is not $O(1-t_n)$.
			In general it follows by simple linear algebra that for all $v\in \C^d\setminus\{0\}$ one has
			$$\langle z^{(n)}-\xi,v\rangle=O\left(\delta_D(z^{(n)})^{1/M_\xi(v)}\right),$$
			where $M_\xi(v)$ is the cotype of $v$.
			Similar considerations hold for the restricted convergence.
		\end{remark}

			\begin{remark}\label{restrictedball}In the unit ball $\B^d$ the conditions (\refeq{extrinsicspecial1}) for $j=1,\dots,d-1$ are redundant, that is, 
			a sequence $(z_n)$ in $\B^d$ $K$-converges to $\xi\in \partial \B^d$ if and only if 
		$\langle z_n-\xi,n_\xi\rangle=O\left(\delta_{\B^d}(z_n)\right).$ Indeed one has, for all $M>1$,
			\begin{equation}\label{koranyiball}
				K_0(\xi,M)=\left\{z\in \B^d:\frac{|1-\langle z,\xi\rangle|}{1-\|z\|}<M\right\}.
			\end{equation}
			\end{remark}
			The next result shows that the same is true  in $\C$-proper convex domains at strongly linearly convex boundary points. We recall that a $C^2$  boundary point of a domain is called \textit{strongly linearly convex} if the Hessian of some (equivalently, any) defining function is definite positive on the complex tangent space.
			Notice that every  strongly linearly convex point of a convex domain has line type 2.
\begin{corollary}
		Let $D\subset\C^d$ be a $\C$-proper convex domain and let $\xi\in\partial D$ be a strongly linearly convex boundary point. Let $(z_n)$ be a sequence in $D$, then
		$(z_n)$ $K$-converges to $\xi$ if and only if 
		$$\langle z_n-\xi,n_\xi\rangle=O\left(\delta_D(z_n)\right).$$
\end{corollary}
\proof
By (i) in Theorem \ref{extrinsicspecial} is enough to prove that, if $$\langle z_n-\xi,n_\xi\rangle=O\left(\delta_D(z_n)\right),$$ then
$$\langle z_n-\xi,\tau\rangle=O\left(\delta_D(z_n)^{1/2}\right)$$
for all $\tau\in T_\xi^\C\partial D$.
Since $\xi$ is a strongly linearly convex boundary point, by \cite[Proposition 22]{KNO} there exists $M\geq 0$ such that 
$$\|z_n-\xi\|\leq M |\langle z_n-\xi,n_\xi\rangle|^{1/2},$$
hence for all $\tau\in T_\xi^\C\partial D$ we have
$$|\langle z_n-\xi,\tau\rangle|\leq \|z_n-\xi\|\leq M|\langle z_n-\xi,n_\xi\rangle|^{1/2}=O\left(\delta_D(z_n)^{1/2}\right).$$
\endproof

			\begin{remark}
		The  conditions (\refeq{extrinsicspecial1}) for $j=1,\dots,d-1$ are however not redundant in general.
Consider as an  example the  {\sl tube} domain
			$$D:=\{(z_0,z_1)\in\C^2: \Re z_0+2(\Re z_1)^2<0\},$$
a $\C$-proper convex domain which is  biholomorphic to the Siegel Half-space 
			 $$\H_2:=\{(z_0,z_1)\in \C^2: {\rm Re}\,z_0+|z_1|^2<0\}$$ 
			via the biholomorphism $\Psi\colon \H_2\to D$ given by
				$$\Psi(z_0,z_1)=\left(z_0-z_1^2,z_1\right),\quad \Psi^{-1}(z_0,z_1)=\left(z_0+z_1^2,z_1\right).$$ In particular $D$ is biholomorphic to the ball $\B^2$.
			Every point of $\partial D$ is strongly pseudoconvex (and thus has line type 2)
			and no point of $\partial D$ is strongly linearly convex.  
			 Let $\alpha>0$ and $t_n\searrow0$, and consider the sequence $z^{(n)}:=(-t_n,it_n^\alpha)$ converging to the origin. The outer normal versor at the origin is $n_0=(1,0)$. Notice that   $\tau=(0,i)\in T_0^\C\partial D$ and  that $\delta_D(z_n)= t_n$ and
			$$\langle z^{(n)},n_0\rangle=-t_n,\quad \langle z^{(n)},\tau\rangle=t_n^{\alpha}.$$
			Hence we have $\langle z^{(n)},n_0\rangle=O(\delta_D(z^{(n)}))$ for all $\alpha>0$, but  $\langle z^{(n)},\tau\rangle=O(\delta_D(z^{(n)})^{1/2})$ if and only if $\alpha\geq\frac{1}{2}$.
				\end{remark}
							
We end this section with another corollary of Theorem \ref{extrinsicspecial}.
		\begin{corollary}\label{rk->nt}
			Let $D\subset\C^d$ be a $\C$-proper convex domain, let $\xi\in\partial D$ be a point of locally finite type, and let $f\colon D\to\C$ be a function. Then
			\begin{itemize}
				\item[(i)] if $f$ has $K$-limit $L$ at $\xi$, then  $f$ has  $K'$-limit $L$ at $\xi$;
				\item[(ii)] if $f$ has  $K'$-limit $L$ at $\xi$, then  $f$ has non-tangential limit $L$ at $\xi$ (in the sense of cones).
			\end{itemize}
		\end{corollary}
		\proof
		(i) follows from the definition of restricted sequence. We prove (ii).
		Let $(z_n)$ be a sequence in $D$ converging to $\xi$ inside a cone of aperture $<\pi$.
		Let $B\subseteq D$ be an Euclidean ball internally tangent at $\xi$, then a simple geometric argument shows that there exists $A>1$ such that $\|z_n-\xi\|\leq A\delta_B(z_n)$ and obviously $\delta_B(z_n)\leq\delta_D(z_n)$,
		thus by Theorem \ref{extrinsicspecial} the sequence $(z_n)$ is restricted in $D$.
		\endproof
		\section{Kobayashi type}

		The following definition was introduced by Abate in \cite{Absurv}, see also Abate--Tauraso \cite{AbTau}.
		
		\begin{definition}[Kobayashi type]\label{KRtype}
			Let $D\subset\C^d$ be a $\C$-proper convex domain and $\xi\in\partial D$. Let $v\in\C^d$, $v\neq 0$, then the {\sl Kobayashi type} $s_\xi(v)$ of $v$ at $\xi$ is the number
			$$s_\xi(v):=\inf\{s>0:\kappa_D(z,v)=O_K(\delta_D(z)^{-s})\}.$$
		\end{definition}
		
		Abate  conjectures the following  in \cite{Absurv}: ``the Kobayashi type might be the inverse of the D'Angelo type of $\partial D$ at $\xi$ along the direction $v$ (that is, the highest order of contact of $\partial D$ with a complex curve tangent to $v$ at $\xi$)''. He also leaves as an open question whether the infimum  in the definition of $s_\xi(v)$ is always attained.
		In this section we prove the conjecture (the equality between D'Angelo type at $\xi$ along the direction $v$ and $m_\xi(v)$ can be proved as in \cite{BS}), and we answer the question in the positive.

			{\begin{theorem}\label{KtypeLtype}
			Let $D\subset\C^d$ be a $\C$-proper convex domain and let $\xi\in\partial D$ be point of locally finite type.  Let $v\in\C^d\setminus\{0\}.$ Then the function $$\delta_D(z)^{1/m_\xi(v)}\kappa_D(z,v)$$
			and its reciprocal are $K$-bounded at $\xi$.
In particular, the infimum in the definition of $s_\xi(v)$ is attained and
$$s_\xi(v)=\frac{1}{m_\xi(v)}.$$
	\end{theorem}
		
		\proof 
			Let $(m_j)_{j=0}^{d-1}$  be the multitype at $\xi$ and consider  multitype coordinates at $\xi$. 
			Write $v=\sum_{j=0}^{d-1} a_je_j$.
			Denote $$j_0:=\min\{j:a_j\neq 0\},$$ and notice that 
			$m_\xi(v)=m_{j_0}$.
			We  first show that  $\delta_D(z)^{1/m_\xi(v)}\kappa_D(z,v)$ is $K$-bounded.
			Assume by contradiction that there exists a sequence $z^{(n)}\to 0$ inside a geodesic region such that $$\delta_D^{1/m_\xi(v)}(z^{(n)})\kappa_D(z^{(n)},v)\to+\infty.$$
			Set $\lambda_n:=\delta_D(z^{(n)})^{-1}$ and consider the scaling maps
			$$A_n(z)=(\lambda_n z_0,\lambda_n^{1/m_1}z_1,\cdots,\lambda_n^{1/m_{d-1}}z_{d-1}).$$
			Arguing as in (i) in Theorem \ref{extrinsicspecial}, we obtain that up to a subsequence
			$$A_nz^{(n)}\to z^{(\infty)}\in D_H.$$
			Now
			$$+\infty=\lim_{n\to+\infty}\lambda_n^{-1/m_\xi(v)}\kappa_D(z^{(n)},v)=\lim_{n\to+\infty}\kappa_{A_nD}(A_nz^{(n)},\lambda_n^{-1/m_\xi(v)}A_nv).$$
			Notice that $$\lambda_n^{-1/m_\xi(v)}A_nv=\sum_{j=j_0}^{d-1}\lambda_n^{1/m_j-1/m_\xi(v)}a_je_j.$$
			Since $m_j\geq m_\xi(v)$ for $j=j_0,\cdots,d-1$, it follows that $$\lambda_n^{-1/m_\xi(v)}A_nv\to v_\infty:=\sum_{j\colon m_j=m_\xi(v)}a_je_j.$$ Finally
			$$+\infty=\lim_{n\to+\infty}\kappa_{A_nD}(A_nz^{(n)},\lambda_n^{-1/m_\xi(v)}A_nv)=\kappa_{D_H}(z^{(\infty)},v_\infty)<\infty,$$
			obtaining a contradiction.
			Now we show that the reciprocal is $K$-bounded. Assume by contradiction that there exists a sequence $z^{(n)}\to 0$ inside a geodesic region such that $$\delta_D^{1/m_\xi(v)}(z^{(n)})\kappa_D(z^{(n)},v)\to0.$$
			Reasoning as before and noting that $v_\infty\neq 0$, we obtain
			$$0=\lim_{n\to+\infty}\lambda_n^{-1/m_\xi(v)}\kappa_D(z^{(n)},v)=\lim_{n\to+\infty}\kappa_{A_nD}(A_nz^{(n)},\lambda_n^{-1/m_\xi(v)}A_nv)=\kappa_{D_H}(z^{(\infty)},v_\infty)>0,$$
			obtaining a contradiction.
			\endproof
		
	\begin{remark} The previous result can be rephrased as the following estimate for the Kobayashi metric: if $D\subset\C^d$ is a $\C$-proper convex domain and  $\xi\in\partial D$ is a point of locally finite type, then for all  $v\in\C^d\setminus\{0\}$, $p\in D$, $M>1$ there exist $c,C>0$ such that for all $z\in K_p(\xi,M)$ we have
$$\label{stimaLee}\frac{c}{\delta_D(z)^{1/m_\xi(v)}}\leq \kappa_D(z,v)\leq\frac{C}{\delta_D(z)^{1/m_\xi(v)}}.
$$
This estimate was obtained along the normal segment by Lee in \cite{Lee} (see also \cite{NPZ} for the $\C$-convex case).
		\end{remark}
		
		\section{Lindel\"of principle}
		
		The purpose of this section is to prove the following version of the Lindel\"of principle. 	
		
		\begin{definition}
			A {\sl restricted curve} $\gamma:[0,1)\to D$ with endpoint $\xi$ is a curve such that for all sequences $(t_n)$ in $[0,1)$ converging to 1 we have that the  sequence $(\gamma(t_n))$ is restricted in the sense of Definition \ref{restricted}.
			Clearly this is equivalent to saying that $\gamma([0,1))$ is contained in a $K$-region with vertex $\xi$ and that $k_D(\gamma(t),\varphi(\D))\stackrel{t\to 1^-}\longrightarrow 0,$ where $\varphi$ is a complex geodesic with endpoint $\xi$.
			We say that  curve  $\gamma:[0,1)\to D$ with endpoint $\xi$ is {\sl strongly restricted} if $\gamma([0,1))$  is contained in a $K$-region with vertex $\xi$ and if there exists a (continuous) curve $\tilde \gamma\colon [0,1)\to \D$ such that
			\begin{equation}\label{stronglyrestricted}\lim_{t\to1^-}k_D(\gamma(t),\varphi(\tilde \gamma(t)))=0,
			\end{equation}
			where $\varphi$ is a complex geodesic with endpoint $\xi$.
		\end{definition}
		
		\begin{remark}
			If $\gamma$ is strongly restricted, the curve $\varphi\circ \tilde \gamma$ is contained in a geodesic region with vertex $\xi$, and thus by \eqref{geodesicintersection} it follows that the curve $\tilde \gamma(t)$ has endpoint 1 and is  non-tangential.  Notice also that the definition of strongly restricted curve does not depend on  the choice of the complex geodesic $\varphi$.
		\end{remark}
		The definition of strongly restricted curve is only slightly more stringent than the definition of restricted curve, as the following lemma shows.
		\begin{lemma}\label{exadmissible}
			Let $D\subset\C^d$ be a $\C$-proper convex domain and let $\xi\in\partial D$ be a point of locally finite type.
			Then any  rectifiable restricted  curve with endpoint $\xi$ is strongly restricted.
		\end{lemma}
		\begin{proof}
			Define a sequence $(t_n)$ as follows. Set $t_0=0$. Then define inductively $t_{n+1}$ as the minimum $t\geq t_{n}$ such that $k_D(\gamma(t),\gamma(t_n))\geq 1/n.$
			It easily follows from the fact that $\gamma$ is rectifiable that $t_n\to 1$.
			Set $z_n:=\gamma(t_n)$. Since the sequence $(z_n)$ is restricted, there exists a sequence $(\zeta_n)$ in $\D$ such that $k_D(z_n,\varphi(\zeta_n))\to 0$. Define the curve $\tilde \gamma\colon[0,1)\to \D$ as follows: for all $n\geq 0$ set
			$\tilde \gamma(t_n):=\zeta_n,$ and then interpolate for all $t\in  [0,1)$ using (suitably reparametrized) geodesic segments.
			It is easy to see that $\lim_{t\to1^-}k_D(\gamma(t),\varphi(\tilde \gamma(t)))=0.$
		\end{proof}

		\begin{theorem}[Lindel\"{o}f Principle]\label{localLindelof}
			Let $D\subset\C^d$ be a $\C$-proper  convex domain and let $\xi\in\partial D$ be a point of locally finite type. Let $f\colon D\to \C$  be a holomorphic function which is $K$-bounded at $\xi$, and assume that there exists a strongly restricted curve $\gamma\colon[0,1)\to D$  such that $f(\gamma(t))\to L$ as $t\to 1$. Then $f$ has restricted $K$-limit $L$ at $\xi$.
		\end{theorem}
		\begin{remark}
			This statement is only apparently similar to \cite[Theorem 3.2]{AbTau}, given our more general notion of restricted convergence, see Remark \ref{rmkAbrestricted} (moreover, we assume $K$-boundedness instead of $T$-boundedness).
			The proof is also different, as it is only based  on the one variable Lindel\"{o}f Principle, the Royden localization lemma  and metric  arguments. 
		\end{remark}

		\begin{lemma}[Royden Localization Lemma]\cite{RoyLocL}
			Let $D\subset\C^d$ be a Kobayashi hyperbolic domain and let $U\subset D$ be a domain. Then
			$$\kappa_{U}(z,v)\leq \coth(k_D(z,D\backslash U)/2) \kappa_D(z,v)$$
			for any $z\in U$ and $v\in\C^d$, where $k_D(z,D\backslash U):=\inf_{w\in D\backslash U}k_D(z,w)$.
		\end{lemma}
		
		\begin{lemma}\label{goodprodev}
			Let $D\subset\C^d$ be a $\C$-proper convex domain and let $\xi\in\partial D$ be a point of locally finite type.
			Let $(z_n)$ be a restricted sequence converging to $\xi$, and let $\varphi\colon \D\to D$ be a complex geodesic with endpoint $\xi$. Let $R>0$ be such that  $z_n\in A(\varphi,R)$ for all $n\geq 0$.
			Let  $(\zeta_n)$ be a sequence in $\D$ such that 
			$k_D(z_n,\varphi(\zeta_n))\to 0.$ Then for all $R'>R$ the sequence $(\varphi(\zeta_n))$ is eventually contained in $A(\varphi,R'$) and we have
			$$\lim_{n\to+\infty}k_{A(\varphi,R')}(z_n,\varphi(\zeta_n))=0.$$
		\end{lemma}
		\proof
		For all $n\geq 0$, let $\gamma_n\colon[0,T_n]\to D$ be a geodesic segment between $z_n$ and $\varphi(\zeta_n)$ with respect to $k_D$. 
		Fix $R<S<R'$. Clearly $\gamma_n([0,T_n])$ is eventually contained in  $A(\gamma,S)$.
		By the Royden localization lemma, we have
		$$\kappa_{A(\varphi,R')}(z,v)\leq B\kappa_{D}(z,v), \ \ \  z\in A(\varphi,S), v\in\C^d,$$
		where $B:=\coth(R'-S)>1$, which implies
		\begin{align*}k_{A(\varphi,R')}(z_n,\varphi(\zeta_n))&\leq \int_0^{T_n} \kappa_{A(\varphi,R')}(\gamma_n(t),\gamma_n'(t))dt\\&\leq B\int_0^{T_n} \kappa_{D}(\gamma_n(t),\gamma_n'(t))dt\\&= Bk_D(z_n,\varphi(\zeta_n))\to0.
		\end{align*}
		\endproof
		
		\begin{remark}It follows from the previous lemma that, if $\gamma\colon[0,1)\to D$ is a strongly restricted curve such that $\gamma(t)\in A(\varphi,R)$ for all $t\in[0,1)$ and $\tilde{\gamma}\colon[0,1)\to\D$  is a curve satisfying \eqref{stronglyrestricted}, then for all $R'>R$ we have that 
			$\varphi(\tilde\gamma(t))$ is eventually contained in $A(\varphi,R')$ and
			$$\lim_{t\to 1^-}k_{A(\varphi,R')}(\gamma(t),\varphi(\tilde\gamma(t)))=0.$$
		\end{remark}
		
		\begin{proof}[Proof of Theorem \ref{localLindelof}]
			Let  $\varphi\colon \D\to D$ be a complex geodesic with endpoint $\xi$ and 
			let $\tilde \gamma\colon[0,1)\to \D$ be a curve satisfying \eqref{stronglyrestricted}.
			Let $R>0$ such that $\gamma(t)\in A(\varphi,R)$ for all $t\in[0,1)$, and 	let $R'>R$. 
			Since $f$ is $K$-bounded at $\xi$ there exists $S>0$ such that $f(A(\varphi,R'))\subset\subset S\D$, and thus by Lemma \ref{goodprodev}
			$$k_{S\D}(f(\gamma(t)),f(\varphi(\tilde\gamma(t))))\leq k_{A(\varphi,R')}(\gamma(t),\varphi(\tilde \gamma(t)))\stackrel{t\to 1^-}\longrightarrow 0,$$ 
			thus
			$$
			\lim_{t\to1^-}|f(\gamma(t))-f(\varphi(\tilde \gamma(t)))|=0.
			$$
			Since $f(\gamma(t))\to L$ as $t\to 1^-$ it follows that $f(\varphi(\tilde \gamma(t)))\to L$. Applying the classical Lindel\"{o}f Theorem to the holomorphic function $f\circ\varphi\colon \D\to \C$ we have that $f\circ\varphi$ has non-tangential limit $L$ at $1$.

			Let now $(z_n)$ be a  restricted sequence in $D$ converging to $\xi$, and let  $(\zeta_n)$ be a sequence in $\D$ such that 
			$k_D(z_n,\varphi(\zeta_n))\to 0.$
			Arguing as above we obtain that 
			\begin{equation}\label{LindEQ}
				\lim_{n\to+\infty}|f(z_n)-f(\varphi(\zeta_n))|=0.
			\end{equation}
			Since  $(\zeta_n)$ converges to $1$ non-tangentially, we have  $f(\varphi(\zeta_n))\to L$. By \eqref{LindEQ} we obtain $f(z_n)\to L$ as $n\to+\infty$, and we are done.
		\end{proof}

		\section{Julia lemma}\label{Julia}
		
		The main result in this section is the generalization to our setting of the intrinsic Julia lemma proved by Abate for bounded strongly convex domains with $C^3$ boundary \cite[Theorem 2.4.16, Proposition 2.7.15]{Abatebook}. For bounded convex domains of finite type  this is already done in \cite[Theorem 6.28]{AFGG}. In our local finite type assumptions the statement and the proof have to be adapted since the classical argument which gives the existence of the $K$-limit at $\xi$  does not work. 	 Lemma \ref{takescare} below is the main tool to deal with this issue.

		\begin{definition}[Dilation]
			Let $D\subset\C^d$ and $D'\subset\C^{q}$ be  $\C$-proper convex domains.  Let $\xi\in\partial D$ be a
			point of locally finite type.
			Let $f\colon D\to D'$ be a holomorphic map, and fix base-points $p\in D,p'\in D'$. The {\sl dilation} $\lambda_{\xi,p,p'}\in(0,+\infty]$ of $f$ at $\xi$ is the number defined by
			$$\log\lambda_{\xi,p,p'}=\liminf_{z\to\xi}k_D(z,p)-k_{D'}(f(z),p').$$
			If $\lambda_{\xi,p,p'}<+\infty,$ we call a sequence $(z_n)$ in $D$ converging to $\xi$  {\sl dilation minimizing}  if 
			$$\lim_{n\to+\infty}k_D(z_n,p)-k_{D'}(f(z_n),p')=\log\lambda_{\xi,p,p'},$$ while we say that  $(z_n)$ has {\sl bounded dilation} if the sequence $(k_D(z_n,p)-k_{D'}(f(z_n),p'))$ is bounded from above. Notice that the notion of bounded dilation is base-point independent, that is, if a sequence has bounded dilation w.r.t base-points $p,p'$, then by the triangle inequality it has bounded dilation also w.r.t. any $q,q'$.
			In  case $D'=\D$ we simply denote $\lambda_{\xi,p}:=\lambda_{\xi,p,0}$.
		\end{definition}
		\begin{remark}
			Notice that $\log\lambda_{\xi,p,p'}$ cannot be $-\infty$. Indeed,
			$$k_D(z,p)-k_{D'}(f(z),p')\geq k_{D'}(f(z),f(p))-k_{D'}(f(z),p')\geq- k_{D'}(f(p),p').$$
		\end{remark}

		\begin{remark}
			A sequence $(z_n)$ in $D$ satisfies  $h_{\xi,p}(z_n)\to -\infty$ if and only if 
			it is eventually contained in every horosphere centered at $\xi$.
			This is the case for example if $(z_n)$ $K$-converges to $\xi$ (by (iii) in Lemma \ref{kvshoro}).
		\end{remark}
		\begin{lemma}\label{takescare}
			Let $D\subset\C^d$ and $D'\subset\C^{q}$ be  $\C$-proper convex domains.  Let $\xi\in\partial D,\eta\in \partial D'$ be 
			points of locally finite type, and let $p\in D,p'\in D'$.
			Let $f\colon D\to D'$ be a holomorphic map such that  $\lambda_{\xi,p,p'}<+\infty$. Let $(z_n)$ be a sequence in $D$ converging to $\xi$ inside a horosphere centered at $\xi$ such that  $f(z_n)\to \eta$.
			If  $(x_n)$  is a  sequence  in $D$ converging to $\xi$  with bounded dilation, then we have $f(x_n)\to \eta$.
		\end{lemma}

		\proof
		Let $A\geq 0$ be such that for all $n\geq0$ we have $$k_D(x_n,p)-k_{D'}(f(x_n),p')\leq A,$$ so $k_{D'}(f(x_n),p')\geq k_D(x_n,p)-A\to+\infty$. Hence, up to subsequence, $(f(x_n))$ converges to a point $\nu\in \partial^*D'.$
		Now  \begin{align*}
			2(f(x_n)|f(z_m))_{p'}&=k_{D'}(f(x_n),p')+k_{D'}(f(z_m),p')-k_{D'}(f(x_n),f(z_m))
			\\&\geq k_{D}(x_n,p)+k_{D'}(f(z_m),p')-k_{D'}(f(x_n),f(z_m))-A
			\\&\geq k_{D}(x_n,p)+k_{D'}(f(z_m),p')-k_{D}(x_n,z_m)-A.
		\end{align*}
Since $h_{\xi,p}(z_m):=\lim_{n\to+\infty}k_D(z_m,x_n)-k_D(x_n,p)$ is bounded from above there exists a constant $M\geq 0$ and a subsequence $(n_m)$ such that for all $m\geq 0$ we have
$$k_D(z_m,x_{n_m})-k_D(x_{n_m},p)\leq M.$$
Hence $$\lim_{m\to+\infty}(f(x_{n_m})|f(z_m))=+\infty, $$
which implies by (i) in Proposition \ref{shadowing} that $\nu=\eta$.
		\endproof
		
		\begin{lemma}\label{basepointindependent}
			Let $D\subset\C^d$ and $D'\subset\C^{q}$ be  $\C$-proper convex domains. Let $\xi\in\partial D,\eta\in \partial D'$ be points of locally finite type, let $p,q\in D$ and $p',q'\in D'$. Let $f\colon D\to D'$ be a holomorphic map such that $\lambda_{\xi,p,p'}<+\infty$ and assume that there exists a sequence $(z_n)$ in $D$ converging to $\xi$ inside a  horosphere centered at  $\xi$ such that
			$f(z_n)\to \eta$.
			Then
			\begin{itemize}
				\item[(i)] $\log\lambda_{\xi,q,q'}=\log\lambda_{\xi,p,p'}+h^D_{\xi,p}(q)+h^{D'}_{\eta,q'}(p');$
				\item[(ii)] the notion of dilation minimizing is base-point independent, that is any dilation minimizing sequence w.r.t. base-points $p,p'$ is dilation minimizing also w.r.t. any $q,q'$.
			\end{itemize}
		\end{lemma}
		\begin{proof}
			Let $w_n\to \xi$ be a dilation minimizing sequence w.r.t. $p,p'$. We have 
			\begin{align*}k_D(q,w_n)-k_{D'}(q',f(w_n))&=k_D(p,w_n)-k_{D'}(p',f(w_n))+k_D(q,w_n)-k_D(p,w_n)\\
				&+k_{D'}(p',f(w_n))-k_{D'}(q',f(w_n)).
			\end{align*}
			By Lemma \ref{takescare} $f(w_n)\to \eta$, and by Theorem \ref{existencehorospheres}
			it follows that 
			\begin{equation}\label{mindil}
				k_D(q,w_n)-k_{D'}(q',f(w_n))\stackrel{n\to+\infty}\longrightarrow \log\lambda_{\xi,p,p'}+h^{D}_{\xi,p}(q)+h^{D'}_{\eta,q'}(p').
			\end{equation}
			Hence
			$$\log\lambda_{\xi,q,q'}\leq \log\lambda_{\xi,p,p'}+h^{D}_{\xi,p}(q)+h^{D'}_{\eta,q'}(p'),$$
			and considering a minimizing sequence w.r.t. $q,q'$ we obtain (i). Notice that  \eqref{mindil} together with (i) implies  (ii). 
		\end{proof}

		\begin{proposition}[Julia Lemma]\label{JuliaLemma}
			Let $D\subset\C^d$ and $D'\subset\C^{q}$ be  $\C$-proper convex domains. Let $\xi\in\partial D,\eta\in \partial D'$ be points of locally finite type. Let $f\colon D\to D'$ be a holomorphic map, and let $p\in D$ and $p'\in D'$. The following are equivalent:
			\begin{enumerate}
				\item  $\lambda_{\xi,p,p'}<+\infty$ and  there exists a sequence  $(z_n)$ in $D$ converging to $\xi$ inside  a horosphere centered at $\xi$ such that $f(z_n)\to \eta$;
				\item  for all $z\in D$ $$h^{D'}_{\eta,p'}(f(z))-h^D_{\xi,p}(z)\leq \log \lambda_{\xi,p,p'},$$ or equivalently, 
				$$f(E^D_{p}(\xi,R))\subseteq E^{D'}_{p'}(\eta,\lambda_{\xi,p,p'}R),\quad \forall R>0;$$
				\item 
				the function $h^{D'}_{\eta,p'}(f(z))-h^D_{\xi,p}(z)$ is bounded from above;
				\item $\lambda_{\xi,p,p'}<+\infty$  and $K\textrm{-}\lim_{z\to\xi}f(z)=\eta$.
			\end{enumerate}
		\end{proposition}

		\begin{proof}
			We prove $(1)\Rightarrow(2)$.
			Let $(w_n)$ be a sequence in $D$ converging to $\xi$ such that $$k_D(w_n,p)-k_{D'}(f(w_n),p')\to \log\lambda_{\xi,p,p'}. $$
			By Lemma \ref{takescare} we have $f(w_n)\to\eta$. Then (2) follows from the metric Julia  Lemma \cite[Lemma 6.14]{AFGG}.
			$(2)\Rightarrow (3)$ is trivial.
			
			We prove $(3)\Rightarrow (4)$.
			If $(w_n)$ is a sequence $K$-converging to $\xi$,  then by (iii) in Lemma \ref{kvshoro}  $h^D_{\xi,p}(w_n)\to-\infty$, and thus
			$h^{D'}_{\eta,p'}(f(w_n))\to-\infty.$ By Remark \ref{intersezchiusura} it follows that $f(w_n)\to\eta$, and thus  $\eta$ is the $K$-limit of $f$ at $\xi$. 
			Moreover, if  $\gamma$ is a geodesic ray with $\gamma(0)=p$ and endpoint $\xi$, then for all $t\geq 0$ we have
			\begin{equation}\label{disuginteress}k_D(p,\gamma(t))-k_{D'}(p',f(\gamma(t)))\leq h^{D'}_{\eta,p'}(f(\gamma(t)))-h^D_{\xi,p}(\gamma(t)), 
			\end{equation}
			and thus $\lambda_{\xi,p,p'}<+\infty.$ $(4)\Rightarrow(1)$ is trivial.

		\end{proof}

		\begin{remark}\label{esisteKlim}
			If the codomain $D'\subset \C^{q}$ is a bounded convex domain of finite type, then it is easy to see that if $\lambda_{\xi,p,p'}<+\infty,$ then there exists $\eta\in \partial D'$ such that    $K\textrm{-}\lim_{z\to\xi}f(z)=\eta$.
			
		\end{remark}

		\begin{definition}\label{defregcont}
			Let $D\subset\C^d$ and $D'\subset\C^{q}$ be  $\C$-proper convex domains.  Let $\xi\in\partial D$ be a point of locally finite type.  Let $f\colon D\to D'$ holomorphic, and fix base-points $p\in D,p'\in D'$. We say that $\xi$ is a {\sl regular contact point} if $\lambda_{\xi,p,p'}<+\infty$ and if  $K\textrm{-}\lim_{z\to\xi}f(z)=\eta$, where $\eta$ is a point of locally finite type in $\partial D'$.
		\end{definition}
		\begin{remark}\label{dilationcambiopolo}
			By Lemma \ref{basepointindependent} the previous definition does not depend on the chosen base-points.
		\end{remark}

		The next   result  is an important consequence of the Julia Lemma.
		\begin{proposition}\label{prop:JFCGeneral}
			Let $D\subset\C^d$ and $D'\subset\C^{q}$ be  $\C$-proper convex domains.  
			Let $f\colon D\to D'$ be a holomorphic map, and let $p\in D$ and $p'\in D'$.  Let $\xi\in\partial D$, $\eta\in \partial D'$ be   points of locally finite type and assume that $\xi$ is a regular contact point with $K\textrm{-}\lim_{z\to\xi}f(z)=\eta$.
			Then,  if $\gamma\colon\R_{\geq 0}\rightarrow D$ is a geodesic ray with  endpoint $\xi$, we have
			\begin{equation}
				\label{limitgeo}
				\lim_{t\to+\infty} k_D(\gamma(t),p) - k_{D'}(f(\gamma(t)),p') = \log \lambda_{\xi,p,p'}.
			\end{equation} 
			As a consequence
			\begin{equation}\label{conilsup}
				\sup_{z\in D}\left(h^{D'}_{\eta,p'}(f(z))-h^D_{\xi,p}(z)\right)=\log\lambda_{\xi,p,p'}.
			\end{equation}
		\end{proposition}
		\begin{proof}
			Set $q:=\gamma(0)$. Arguing as in 
			\cite[Lemma 3.12]{AFGK} using  (2) in Proposition \ref{JuliaLemma}
			we have that $$\lim_{t\to+\infty} k_D(\gamma(t),q) - k_{D'}(f(\gamma(t)),p') = \log \lambda_{\xi,q,p'}.$$
			Then \eqref{limitgeo} follows from (ii) in Lemma \ref{basepointindependent}.
			Equation \eqref{conilsup} now follows using \eqref{disuginteress}.
		\end{proof}

		The definition of regular contact point given above is different from the one given in Definition \ref{contacttemp} for holomorphic maps with domain of definition $\D$. The next result shows that the two definitions are in fact equivalent. The proof is similar to \cite[Proposition 3.8, Lemma 3.13]{AFGK}  so we omit it.
		\begin{proposition}\label{eqdefcontact}
			Let $D\subset\C^d$ and $D'\subset\C^{q}$ be  $\C$-proper convex domains. Let $\xi\in\partial D,\eta\in \partial D'$ be points locally finite type, and let $p\in D,p'\in D'$.
			Let $f\colon D\to D'$ be a holomorphic map. The following are equivalent.
			\begin{enumerate}
				\item $\lambda_{\xi,p,p'}<+\infty$ and $K\textrm{-}\lim_{z\to\xi}f(z)=\eta$;
				\item there exists a geodesic ray $\gamma$ with endpoint $\xi$ such that $f\circ \gamma$ is a $(1,B)$ quasi-geodesic with endpoint $\eta$;
				\item for all geodesic rays $\gamma$ with endpoint $\xi$ the curve $f\circ \gamma$ is an almost-geodesic with endpoint $\eta$.
			\end{enumerate}
			
		\end{proposition}

		We now prove other  consequences of the Julia lemma. The first  is that $K$-regions are mapped into $K$-regions.
		\begin{corollary}\label{georintogeor}
			Let $D\subset\C^d$ and $D'\subset\C^{q}$ be  $\C$-proper convex domains.  Let $\xi\in\partial D,\eta\in\partial D'$ be
			points of locally finite type, and let $p\in D$.
			Let $f\colon D\to D'$ be holomorphic, and assume that $\xi$ is a regular contact point with $K\textrm{-}\lim_{z\to\xi}f(z)=\eta$.
			Then for all $M>1$,
			$$f(K_{p}(\xi,M))\subseteq K_{f(p)}(\eta,\lambda_{\xi,p,f(p)}^{1/2}M).$$
		\end{corollary}
		\proof
		Let $z\in K_p(\xi,M)$, then by the Julia Lemma
		\begin{align*}&h_{\eta,f(p)}(f(z))+k_{D'}(f(z),f(p))\leq h_{\xi,p}(z)+\log\lambda_{\xi,p,f(p)}+k_D(z,p)\\ &< 2\log M+\log\lambda_{\xi,p,f(p)}=2\log (\lambda_{\xi,p,f(p)}^{1/2}M).
		\end{align*}
		\endproof

		Another corollary of the Julia Lemma is a chain rule for the dilation.
		\begin{corollary}\label{chainrule}
			Let $D\subset \C^d$,  $D'\subset \C^{d'}$, and $D''\subset \C^{d''}$ be  $\C$-proper convex domains.
			Let $f\colon D\to D'$, $g\colon D'\to D''$ be holomorphic maps and assume that 
			$\xi\in \partial D$ is a regular contact point of $f$ with $K\textrm{-}\lim_{z\to \xi}f(z)=\eta\in \partial D'$, and that
			$\eta$ is a regular contact  point of $g$ with $K\textrm{-}\lim_{w\to \eta}g(w)=\zeta\in \partial D''$. Then $\xi$ is a regular contact point of $g\circ f$ with  $K\textrm{-}\lim_{z\to \xi}g(f(z))=\zeta$ and
			$$\lambda_{\xi,p,p''}(g\circ f)=\lambda_{\xi,p,p'}(f)\cdot\lambda_{\eta,p',p''}(g),$$
			where  $p\in D, p'\in D',p''\in D''$.
		\end{corollary}
		\begin{proof}
			By the Julia Lemma we obtain 
			$$h_{\zeta,p''}(g(f(z)))\leq h_{\xi,p}(z)+\log\lambda_{\xi,p,p'}(f)+\log\lambda_{\eta,p',p''}(g), \quad \forall z\in D.$$
			Again by  the Julia Lemma this implies that $\xi$ is a  regular contact point of the map $g\circ f$,   
			and that the boundary point $\zeta$ is the $K$-limit of $g\circ f$ at $\xi$. By \eqref{conilsup} we obtain
			$$\log\lambda_{\xi,p,p''}(g\circ f)\leq \log\lambda_{\xi,p,p'}(f)+\log\lambda_{\eta,p',p''}(g).$$
			
			On the other hand, if $\gamma$ is a geodesic ray in $D$ with  endpoint $\xi$, by  \eqref{limitgeo} we have
			\begin{align*}
				&\log\lambda_{\xi,p,p''}(g\circ f)=\lim_{t\to+\infty}k_D(\gamma(t),p)-k_{D''}(g(f(\gamma(t))),p'')\\
				&\geq \lim_{t\to+\infty}k_D(\gamma(t),p)-k_{D'}(f(\gamma(t)),p')+\liminf_{t\to +\infty} k_{D' }(f(\gamma(t)),p')-k_{D''}(g(f(\gamma(t))),p'')\\
				&\geq \log\lambda_{\xi,p,p'}(f)+\log\lambda_{\eta,p',p''}(g).
			\end{align*}
		\end{proof}

		We end this section with an extrinsic characterization of regular contact points.

		\begin{proposition}\label{ratiodeltas}Let $D\subset\C^d$ and $D'\subset\C^{q}$ be  $\C$-proper convex domains. Let $\xi\in\partial D,\eta\in \partial D'$ be points of locally finite type. Let $f\colon D\to D'$ be a holomorphic map and  assume  $K\textrm{-}\lim_{z\to\xi}f(z)=\eta$.
			The following are equivalent.
			\begin{enumerate}
				\item $\xi$ is a regular contact point;
				\item  $\delta_{D'}(f(z))/\delta_D(z)$ is $K$-bounded at $\xi$;
				\item there exists a sequence $(z_n)$ in $D$, $K$-converging to $\xi$, such that 
				the sequence $(\delta_{D'}(f(z_n))/\delta_D(z_n))$ is bounded.
			\end{enumerate}
		\end{proposition}
		\proof We prove $(1)\Rightarrow (2)$. Let $\gamma\colon\R_{\geq 0}\to D$ be a geodesic ray with endpoint $\xi$. Since $\xi$ is a regular contact point, by \eqref{limitgeo} there exists $A\geq 0$ such that for all $t\geq0$
		$$k_D(\gamma(t),p)-k_{D'}(f(\gamma(t)),p')\leq A.$$ Fix $R>0$ and let $z\in A(\gamma,R)$ be close to $\xi$, then by definition there exists $t^*\geq 0$ such that $k_D(z,\gamma(t^*))<R$. Now using Lemma \ref{stimekob} for $D$ near $\xi$ and $D'$ near $\eta$ we can find $c\geq 0$ such that
		\begin{align*}
			\log\frac{\delta_{D'}(f(z))}{\delta_D(z)}&\leq k_D(z,p)-k_{D'}(f(z),p')+c
			\\&\leq k_D(z,\gamma(t^*))+k_D(\gamma(t^*),p)+k_{D'}(f(z),f(\gamma(t^*)))-k_{D'}(f(\gamma(t^*)),p')+c
			\\&\leq 2R+A+c.
		\end{align*}
		The implication $(2)\Rightarrow(3)$ is trivial, and 
		$(3)\Rightarrow(1)$ follows from Lemma \ref{stimekob}.
		\endproof
		
		The previous result can be strengthened if the codomain $D'\subset\C^{q}$ is a bounded convex domain of finite type.
		\begin{proposition}\label{ratiodeltas+}
			Let $D\subset\C^d$  be a  $\C$-proper convex domain, and let $D'\subset\C^{q}$ be a bounded convex domain of finite type. Let $\xi\in\partial D$ be a
			point of locally finite type. Let $f\colon D\to D'$ be a holomorphic map.
			Then the following are equivalent:
			\begin{enumerate}
				\item $\xi$ is a regular contact point; 
				\item $\delta_{D'}(f(z))/\delta_D(z)$ is $K$-bounded at $\xi$;
				\item $\liminf_{z\to\xi}\delta_{D'}(f(z))/\delta_D(z)<+\infty.$
			\end{enumerate}
		\end{proposition} 
		\begin{proof}
			If $\xi$ is a  regular contact point then (2) holds by the previous proposition. (2) $\Rightarrow$ (3) is trivial.
			If (3)  holds, let $(w_n)$ be a sequence in $D$ converging to $\xi$ such that 
			$\delta_{D'}(f(w_n))/\delta_D(w_n)$ is bounded. Up to a subsequence we may assume that the sequence $(f(w_n))$ converges to a point  in $\eta\in\overline {D'}$. Since $\delta_D(w_n)\to 0$ we have $\delta_{D'}(f(w_n))\to 0$, and thus $\eta\in \partial D'$. By  Lemma \ref{stimekob} and Remark \ref{esisteKlim} $\xi$ is a regular contact point.
		\end{proof}

		\section{Pluricomplex Poisson kernel and normalized dilation}
		
		In our generalization of the Julia-Wolff-Carath\'eodory theorem we want to establish a relationship between  the dilation $ \lambda_{\xi,p,p'}(f)$ of a holomorphic map at a regular contact point $\xi\in \partial D$ (with $K$-limit $\eta\in \partial D'$) and the restricted $K$-limit of the normal component of the normal directional derivative
		$\langle df_z(n_\xi),n_\eta\rangle$. The fact that the dilation $\lambda_{\xi,p,p'}(f)$ depends on the  base-points $p\in D,p'\in D'$ is, from this point of view, a disturbance. We now introduce, in the setting of  $\C$-proper convex domains of finite type,  a function which generalizes  {the pluricomplex Poisson kernel} originally defined  (with completely different methods) by  Bracci--Patrizio--Trapani \cite{BP,BPT} on strongly convex domains with $C^\infty$ boundary,  and we use it  to normalize the dilation and eliminate the dependence on the base-points.
		
		Recall that by Corollary \ref{approaching} two  complex geodesics  with the same endpoint $\xi\in \partial D$ are strongly asymptotic.
		\begin{lemma}\label{Thoro}
			Let $D\subset\C^d$ be a $\C$-proper convex domain and let $\xi\in\partial D$ be of locally finite type.
			Let $\varphi,\psi\colon\D\to D$ be two complex geodesics  with the same endpoint $\xi$.
			Let  $T\in \R$ be such that $k_D(\tilde\varphi(t),\tilde\psi(t+T))\to 0.$ Then  
			$$T=h_{\xi,\varphi(0)}(\psi(0)).$$
			In particular, if $\varphi(0)=\psi(0)$, then $T=0.$
		\end{lemma}
		\begin{proof}
			We have
			$$h_{\xi,\varphi(0)}(\psi(0))=\lim_{t\to+\infty}k_D(\tilde\psi(0),\tilde \varphi(t))-k_D(\tilde\varphi(t),\tilde \varphi (0))=\lim_{t\to+\infty}k_D(\tilde \psi(0),\tilde\psi(t+T))-k_D(\tilde\varphi(t),\tilde \varphi(0))=T.$$
		\end{proof}
		
		\begin{proposition}\label{Tderiv}
			Let $D\subset\C^d$ be a $\C$-proper convex domain and let $\xi\in\partial D$ be a point of locally finite type.  Let $\varphi,\psi\colon\D\to D$ be two complex geodesics  with the same endpoint $\xi$.
			Let  $T\in \R$ be such that $k_D(\tilde\varphi(t),\tilde\psi(t+T))\to 0.$
			Then 
			\begin{equation}\label{formulaT}
				T=\log\frac{\psi'_N(1)}{\varphi_N'(1)},
			\end{equation}
			and thus
			\begin{equation}\label{formulaT2}
				h_{\xi,\varphi(0)}(\psi(0))=\log\frac{\psi'_N(1)}{\varphi_N'(1)}.
			\end{equation}
			In particular if $\varphi(0)=\psi(0)$, then $\psi'_N(1)=\varphi_N'(1).$
		\end{proposition}
		
		\begin{proof}
			Let $a\in(-1,1)$ be such that $\frac{1-a}{1+a}=\frac{\varphi'_N(1)}{\psi_N'(1)}$ and let   $\tau_a\colon\D\to\D$ be the automorphism of the disc given by $\tau_a(\zeta)=\frac{\zeta+a}{1+a\zeta}$. Notice that $\tau'_a(1)=\frac{1-a}{1+a}.$ 
			The complex geodesic $\psi_a:=\psi\circ\tau_a$ satisfies $(\psi_a)'_N(1)=\varphi'_N(1)$, so by Theorem \ref{approchfintyp+} we have
			$$\lim_{t\to+\infty}{k_D}(\tilde \varphi(t),\tilde{\psi}_a(t))=0.$$
			A simple computation shows that $\tilde{\psi}_a(t)=\tilde \psi\left(t+\log\frac{\psi'_N(1)}{\varphi_N'(1)}\right)$.
			Equation \ref{formulaT2}  follows combining  \eqref{formulaT} with Lemma \ref{Thoro}.
		\end{proof}
		\begin{definition}
			Let $D\subset\C^d$ be a $\C$-proper convex domain and let $\xi\in\partial D$ be a point of locally finite type.
			For all $z\in D$ we define the {\sl pluricomplex Poisson kernel}  $\Omega_\xi\colon D\to (-\infty,0)$ as
			$$\Omega_\xi(z)=-\frac{1}{\varphi'_N(1)},$$ where $\varphi\colon \D\to D$ is a complex geodesic such that $\varphi(0)=z$ and with endpoint $\xi$. Notice that it is not known whether there is a \emph{unique} complex geodesic connecting the point $z$ to $\xi$, however $\Omega_\xi$ is well-defined thanks to Proposition \ref{Tderiv}.
		\end{definition}

		\begin{remark}\label{omegahorofunction}
			It follows from Proposition \ref{Tderiv} that if $z,w\in D$, then
			\begin{equation}\label{orosferepoisson}
				h_{\xi,z}(w)=\log |\Omega_\xi(z)|-\log  |\Omega_\xi(w)|,
			\end{equation}
			hence $\Omega_\xi$ is continuous and its level sets are exactly the horospheres centered in $\xi$.
		\end{remark}
		
		The function $\Omega^\D_1$ coincides with the classical (negative) Poisson kernel in the disc:
		$$\Omega^\D_1(\zeta)=-\frac{1-|\zeta|^2}{|1-\zeta|^2}.$$
		Moreover, the function $\Omega_\xi$ is a multiple of the classical Poisson kernel in the unit disc when restricted to a complex geodesic with endpoint $\xi$.

		\begin{corollary}\label{forPoigeo}
			Let $D\subset\C^d$ be a $\C$-proper convex domain and let $\xi\in\partial D$ be a point of locally finite type. Let $\varphi\colon \D\to D$ be a complex geodesic with endpoint $\xi$.
			Then, for all $\zeta\in \D$,
			$$\Omega_\xi^D(\varphi(\zeta))= |\Omega_\xi^D(\varphi(0))|\Omega_1^\D(\zeta)= \frac{\Omega_1^\D(\zeta)}{\varphi'_N(1)}.$$
		\end{corollary}
		\begin{proof}
			By  \eqref{orosferepoisson}, $$h^\D_{1,0}(\zeta)=-\log|\Omega^\D_1(\zeta)|,\quad \zeta\in \D.$$
			Again by \eqref{orosferepoisson}  we obtain, for all $\zeta\in \D$,
			$$-\log|\Omega^\D_1(\zeta)|= h^\D_{1,0}(\zeta)=h^D_{\xi,\varphi(0)}(\varphi(\zeta))=\log|\Omega^D_\xi(\varphi(0))|-\log|\Omega^D_\xi(\varphi(\zeta))|.$$
		\end{proof}

		\begin{example}\label{poisegg}
			Using the previous corollary and  formula \eqref{formulageodetiche}  one easily obtains an explicit formula for the pluricomplex Poisson kernel of the egg domain $\mathbb{E}_m$ at the point $\xi=(1,0)$, that is
			$$\Omega^{\mathbb{E}_m}_\xi(z_0,z_1)=-\frac{1-|z_0|^2-|z_1|^m}{|1-z_0|^2}.$$
		\end{example}

		\begin{corollary}
			If $D\subset\subset \C^d$ is strongly convex with $C^\infty$ boundary, the function $\Omega_\xi$ coincides with the pluricomplex Poisson kernel introduced by Bracci--Patrizio--Trapani \cite{BP,BPT}. 
		\end{corollary}
		\proof
		See (1.2) in \cite{BPT}.
		\endproof
		
		If $D\subset\subset \C^d$  is strongly convex with $C^\infty$ boundary, the pluricomplex Poisson kernel $\Omega_\xi$ solves the following homogeneous Monge-Amp\`ere equation with a simple singularity at the boundary:
		\begin{equation}\label{MAprob}\begin{cases*}
				\Omega_\xi\mbox{ is plurisubharmonic in }D\\
				(\partial\overline{\partial}\Omega_\xi)^d=0\ \mbox{ in }D\\
				\Omega_\xi<0 \ \mbox{ in }D\\
				\Omega_\xi(\eta)=0 \ \mbox{ if }\eta\in\partial D\backslash\{\xi\}\\
				\Omega_\xi(z)\approx\|z-\xi\|^{-1} \mbox{ as }z\to\xi\mbox{ non-tangentially}
			\end{cases*}
		\end{equation}
		In a forthcoming paper \cite{toappear}  with F. Bracci  we prove that  in  convex domains of finite type the function $\Omega_\xi$ solves an equation analogous to \eqref{MAprob}.
		\begin{definition}[Normalized dilation]
			Let $D\subset\C^d$ and $D'\subset\C^{q}$ be  $\C$-proper convex domains and $f\colon D\to D'$ be a holomorphic map. 
			Let $\xi\in\partial D,\eta\in \partial D'$ be points of locally finite type.
			Assume that $\xi$ is a regular contact point with $K\textrm{-}\lim_{z\to\xi}f(z)=\eta$. We define the {\sl  normalized dilation} of $f$ at $\xi$ as
			$$\alpha_\xi:=\lambda_{\xi,p,p'}\frac{\Omega^D_\xi(p)}{\Omega^{D'}_\eta(p')}\in (0,+\infty),$$
			where $p\in D,p'\in D'$.
		\end{definition}
		The next result shows that the normalized dilation is well-defined, that is, it does not depend on the chosen base-points $p\in D,p'\in D'.$
		\begin{lemma}
			Let $D\subset\C^d, D'\subset \C^{q}$ be  $\C$-proper convex domains and let $\xi\in\partial D,\eta\in \partial D'$ be   points of locally finite type. 
			Let $f\colon D\to D'$ be a  holomorphic map, and assume that $\xi$ is a regular contact point  with $K\textrm{-}\lim_{z\to\xi}f(z)=\eta$.
			Then the positive real number  $$\lambda_{\xi,p,p'}\frac{\Omega^D_\xi(p)}{\Omega^{D'}_\eta(p')}\in (0,+\infty)$$ is independent on $p\in D,p'\in D'$.
		\end{lemma}
		\begin{proof}
			Let  $p,q\in D,p',q'\in D'$. By Lemma \ref{basepointindependent} and  \eqref{orosferepoisson} we have
			\begin{align*}&\log \lambda_{\xi,p,p'} - \log \lambda_{\xi,q,q'}+\log |\Omega^D_\xi(p)|-\log |\Omega^D_\xi(q)|-\log |\Omega^{D'}_\eta(p')|+\log |\Omega^{D'}_\eta(q')|\\
				=& h^D_{\xi,q}(p)+h^{D'}_{\eta,p'}(q')+ h^D_{\xi, p}(q)+h^{D'}_{\eta,q'}(p')=0.
			\end{align*}
		\end{proof}
		
		\begin{remark} 
			Notice that, if $D'=\D$, then  $\alpha_\xi=\lambda_{\xi,p,0}|\Omega^D_\xi(p)|$.
		\end{remark}
		\begin{remark}
			It immediately follows from Corollary \ref{chainrule} that the normalized dilation satisfies a chain rule:
			\begin{equation*}\label{omegachainrule}\alpha_\xi( g\circ f)=\alpha_\xi(f)\cdot\alpha_\eta(g).
			\end{equation*}
		\end{remark}
		\begin{lemma}\label{dilgeoinv}
			Let $D\subset\C^d$ be a $\C$-proper convex domain and let $\xi\in\partial D$ be a point of locally finite type.
			Let $\varphi\colon \D\to D$ be a complex geodesic with endpoint $\xi$. Then $1$ is a regular contact point of $\varphi$ and
			$$\alpha_{1}(\varphi)=\frac{1}{|\Omega^D_\xi(\varphi(0))|}=\varphi_N'(1).$$ 
			Moreover, if
			$\tilde{\rho}\colon D\to \D$ is a  left inverse of $\varphi$, 
			then $\xi$ is a regular contact point of $\tilde{\rho}$, and
			$$ \alpha_\xi(\tilde{\rho})=|\Omega^D_\xi(\varphi(0))|=\frac{1}{\varphi_N'(1)}.$$
		\end{lemma}
		\begin{proof}
			Notice that since $\varphi$ is a complex geodesic, we have $\lambda_{1,0,\varphi(0)}(\varphi)=1.$ Hence
			$$\alpha_1(\varphi)=\lambda_{1,0,\varphi(0)}(\varphi)\frac{\Omega_1^\D(0)}{\Omega_\xi^D(\varphi(0))}=\frac{1}{|\Omega^D_\xi(\varphi(0))|}.$$ Similarly one obtains the second statement.
			
		\end{proof}

		\section{The Julia--Wolff--Carath\'eodory theorem}
		
		\subsection{Statement of the result}
		Our main theorem is the following generalization of the Julia--Wolff-Carath\'eodory theorem.
		\begin{theorem}\label{genJWC}
			Let $D\subset\C^d$ and $D'\subset\C^{q}$ be  $\C$-proper convex domains.
			Let $\xi\in\partial D,\eta\in \partial D'$ be points of locally finite type with multitype respectively $(m_j)_{j=0}^{d-1}$ and $(n_i)_{i=0}^{q-1}$.
			Let $(v_j)_{j=0}^{d-1}$ (resp. $(u_i)_{i=0}^{q-1}$) be an orthonormal  multitype basis at $\xi$ (resp. at $\eta$). 
			Let  $f\colon D\to D'$ be a holomorphic map and assume that $\xi$ is a regular contact point with $K\textrm{-}\lim_{z\to\xi}f(z)=\eta$.
			Then for all\, $0\leq j \leq d-1$ and \,$0\leq i\leq q-1$,
			\begin{equation}\label{proofKbounded}
				\langle df_z(v_j),u_i\rangle=O_K\biggl( \delta_D(z)^{\frac{1}{n_i}-\frac{1}{m_j}}\biggr)
			\end{equation}
			and
			\begin{itemize}
				\item[(i)] $K'\textrm{-}\lim_{z\to\xi}\langle df_z(v_0),u_0\rangle=\alpha_\xi:=\lambda_{\xi,p,p'}\frac{\Omega^D_\xi(p)}{\Omega^{D'}_\eta(p')}\in (0,+\infty)$;
				\item[(ii)] 
				$\langle df_z(v_j),u_0\rangle=o_{ K'}\biggl(\delta_D(z)^{1-\frac{1}{m_j}}\biggr)$ for all\,  $1\leq j\leq  d-1$;
				\item[(iii)] 
				$ \langle df_z(v_0),u_i\rangle=o_{ K'}\biggl(\delta_D(z)^{\frac{1}{n_i}-1}\biggr)$  for all\, $1\leq i\leq  q-1$.
			\end{itemize}
		\end{theorem}

		\begin{remark}\label{JWCcorretta} (cf. Remark \ref{typevscotype}). It is  natural to wonder whether 
			\begin{equation}\label{formulafalsa}
				\langle df_z(v),u\rangle=O_K \biggl(\delta_D(z)^{\frac{1}{m_\eta(u)}-\frac{1}{m_\xi(v)}}\biggr)
			\end{equation}
			for all $v,u\in \C^d\setminus\{0\}$. This turns out to be false in general (see e.g. Example \ref{exJWC}). However, it follows from \eqref{proofKbounded} using simple linear algebra that
			$$\langle df_z(v),u\rangle=O_K \biggl(\delta_D(z)^{\frac{1}{M_\eta(u)}-\frac{1}{m_\xi(v)}}\biggr).$$

		\end{remark}
		It is worth stating Theorem \ref{genJWC} in the particular case of functions with values in the unit disc. This statement will actually play an important role in the proof of Theorem \ref{genJWC}.
		\begin{theorem}\label{JWC1}
			Let $D\subset\C^d$ be a $\C$-proper convex domain and let $\xi\in\partial D$ be a point of locally finite type.
			Let $f\colon D\to \D$ be a holomorphic function, and assume that $\xi$ is a regular contact point. 
			Then for all $v\in \C^d\setminus\{0\}$, $$ \frac{\partial f}{\partial v}(z)=O_{K}\left( \delta_D(z)^{1-\frac{1}{m_\xi(v)}}\right).$$ Moreover
			\begin{itemize}
				\item[(i)] $ K'\textrm{-}\lim_{z\to\xi} \frac{\partial f}{\partial n_\xi}(z)=\eta\alpha_\xi=\eta |\Omega^D_\xi(p)| \lambda_{\xi,p},$
				where $\eta= K\textrm{-}\lim_{z\to\xi}f(z)$,
				\item[(ii)] $\frac{\partial f}{\partial v}(z)=o_{K'}\left( \delta_D(z)^{1-\frac{1}{m_\xi(v)}}\right)$ if $v\in  T^\C_\xi\partial D\setminus\{0\}$.
			\end{itemize}
			Hence
			if $v\in \C^d$, 
			$$K'\textrm{-}\lim_{z\to\xi} \frac{\partial f}{\partial v}(z)=\langle v,n_\xi \rangle\eta\alpha_\xi=\langle v,n_\xi \rangle \eta \,|\Omega^D_\xi(p)| \lambda_{\xi,p}.$$
			
		\end{theorem}

		We divide the proof of Theorem \ref{genJWC} in three steps.
		\subsection{First step: proof of \eqref{proofKbounded}}

		\begin{remark} 
			If $\zeta\in \D$, then
			\begin{equation}\label{stimadisco}-\log(1-|\zeta|)\leq k_\D(0,\zeta)\leq \log2-\log(1-|\zeta|).
			\end{equation}
		\end{remark}

		\begin{lemma}\label{lemmadistanza}
			
			Let $D\subset\C^d$ be a $\C$-proper convex domain and let $\xi\in\partial D$ be a point of locally finite type.
			Let $\varphi\colon \D\to D$ be a complex geodesic with endpoint $\xi$, and let $\tilde \rho\colon D\to \D$ be a left inverse of $\varphi$. 
			Then the function
			$$\frac{1-\tilde{\rho}(z)}{\delta_D(z)}$$
			and its reciprocal are $K$-bounded at $\xi$.
			
		\end{lemma}
		
		\proof Let $p:=\varphi(0)$. Denote  $\rho:=\varphi\circ\tilde \rho\colon D\to D$ and notice that it is a holomorphic retraction with image $\varphi(\D)$.  Fix $R>0$ and let $z\in A^D(\varphi,R)$. Let $t^*\in[0,1)$ such that $k_D(z,\varphi(t^*))<R$, then $k_D(\rho(z),\varphi(t^*))<R$ and thus $\tilde{\rho}(z)\in A^\D(\id_\D,R)$. By the proof of Proposition \ref{abatek} we have $\tilde{\rho}(z)\in K_0^\D(1,e^R)$. 
		Using \eqref{koranyiball}  and \eqref{stimadisco} we have
		$$|1-\tilde{\rho}(z)|\leq e^R(1-|\tilde{\rho}(z)|)\leq 2e^{R-k_\D(0,\tilde{\rho}(z))}=2e^{R-k_D(p,\rho(z))}.$$
		
		Let $V$ be the neighborhood of $\xi$ given by Lemma  \ref{stimekob}.
		Notice that $\xi$ is a regular contact point for $\rho$ with
		$ K\textrm{-}\lim_{z\to\xi} \rho(z)=\xi,$
		hence there exists a neighborhood  $U$  of $\xi$ such that $$\rho(U\cap A^D(\varphi,R))\subset V.$$
		Then for all $z\in U\cap A^D(\varphi,R)$ we have 
		$$|1-\tilde{\rho}(z)|\leq 2e^{c+R}\delta_D(\rho(z)),$$
		where $c>0$ is the constant given by  Lemma  \ref{stimekob}. Since  by Proposition \ref{ratiodeltas} the function $\delta_D(\rho(z))/\delta_D(z)$ is $K$-bounded, it follows that  $(1-\tilde{\rho}(z))\slash\delta_D(z)$ is $K$-bounded.
		
		To prove that the reciprocal is also $K$-bounded, notice that   $k_D(p,z)\geq k_D(p,\rho(z))=k_\D(0,\tilde{\rho}(z))$.  If $z\in V\cap D$,   by Lemma \ref{stimekob} and \eqref{stimadisco} we have
		$$\delta_D(z)\leq e^{c-k_D(p,z)}\leq e^{c-k_\D(0,\tilde{\rho}(z))}\leq e^c(1-|\tilde{\rho}(z)|)\leq e^c |1-\tilde{\rho}(z)|,$$
		hence the function $\delta_D(z)/(1-\tilde{\rho}(z))$ is bounded in $V\cap D$, and thus it is $K$-bounded.
		\endproof

		\begin{proposition}\label{fphisigmac}
			Let $D\subset\C^d$ and $D'\subset\C^{q}$ be  $\C$-proper convex domains and  let $\xi\in \partial D,\eta\in \partial D' $ be  points of locally finite type. Let $f\colon D\to D'$ be a holomorphic map. 
			Assume that $\xi\in\partial D$ is a regular contact point with $K\textrm{-}\lim_{z\to\xi}f(z)=\eta$.
			Then $f$ sends strongly restricted curves in $D$ with endpoint $\xi$ to strongly restricted curves in $D'$ with endpoint $\eta$.
		\end{proposition}
		\proof  We first prove the result in the case $D=\D$ and $\xi=1\in\partial\D$. 
		First of all we recall that by Proposition \ref{normalderivative} there exists $f'_N(1)>0$.
		Now let $\psi\colon\D\to D'$ be a complex geodesic with endpoint $\eta$ such that   $\psi'_N(1)=f'_N(1)$. 
		Let $\gamma\colon[0,1)\to\D$ be a curve converging to $1$ non-tangentially, then by Corollary \ref{approchfintyp+}
		$$\lim_{t\to1^-}k_D(f(\gamma(t)),\psi(\gamma(t)))=0,$$
		which implies that $f\circ\gamma$ is a strongly restricted curve in $D'$ with endpoint $\eta$.
		
		In the general case, let $\gamma\colon[0,1)\to D$ be a strongly restricted curve with endpoint $\xi$, and let  $\varphi\colon \D\to D$ be a complex geodesic with endpoint $\xi$. By definition there exists  a  non-tangential curve $\tilde\gamma\colon[0,1)\to\D$ with endpoint 1 such that
		$\lim_{t\to1^-}k_D(\gamma(t),\varphi(\tilde\gamma(t)))=0.$
		By the first part of the proof applied to the map $f\circ \varphi\colon\D\to D$ there exists a complex geodesic $\psi\colon \D\to D'$ with endpoint $\eta$ such that 
		$$\lim_{t\to1^-}k_{D'}(f(\varphi(\tilde\gamma(t))),\psi(\tilde\gamma(t)))=0,$$
		and thus
		\begin{align*}k_{D'}(f(\gamma(t)),\psi(\tilde\gamma(t)))&\leq k_{D'}(f(\gamma(t)),f(\varphi(\tilde\gamma(t))))+
			k_{D'}(f(\varphi(\tilde\gamma(t))),\psi(\tilde\gamma(t)))\\&\leq k_{D}(\gamma(t),\varphi(\tilde\gamma(t)))+
			k_{D'}(f(\varphi(\tilde\gamma(t))),\psi(\tilde\gamma(t)))\stackrel{t\to 1^-}\longrightarrow 0.\end{align*}
		Moreover, by  Corollary \ref{georintogeor} the curve $t\mapsto f(\gamma(t))$ is contained in a geodesic region with vertex $\eta$, hence it is strongly restricted.
		\endproof

		\begin{proposition}\label{fmenopifn} 
			Let $D\subset\C^d$ and $D'\subset\C^{q}$ be  $\C$-proper convex domains  and  let $\xi\in \partial D,\eta\in \partial D' $ be  points of locally finite type. Let
			$(u_i)_{i=0}^{q-1}$ be an orthonormal multitype basis at $\eta$ with multitype $(n_i)_{i=0}^{q-1}.$ Let  $f\colon D\to D'$ be a holomorphic map. 
			Assume that $\xi\in\partial D$ is a regular contact point with $K\textrm{-}\lim_{z\to\xi}f(z)=\eta$. Then
			\begin{itemize}
				\item[(i)] 
				$\langle f(z)-\eta,u_i\rangle=O_K(\delta_{D}(z)^{1/n_i}),$ for all \, $0\leq i\leq q-1;$
				\item[(ii)]
				$\langle f(z)-\eta,u_i\rangle= o_{K'}(\delta_{D}(z)^{1/n_i}),$ for all \, $1\leq i\leq q-1.$
			\end{itemize}
		\end{proposition}
		\proof
		Let $R>0$,  let $\varphi\colon \D\to D$ be a complex geodesic with endpoint $\xi$. By  Proposition \ref{ratiodeltas} there exists $C_1\geq 0$ such that for all $z\in A(\varphi,R)$
		$$\delta_{D'}(f(z))\leq C_1\delta_D(z).$$
		We now prove (i). We have \begin{equation*}\label{eqpi}\frac{|\langle f(z)-\eta,u_i\rangle|}{\delta_{D}(z)^{1/n_i}}\leq C_1^{1/n_i}\frac{|\langle f(z)-\eta,u_i\rangle|}{\delta_{D'}(f(z))^{1/n_i}}.\end{equation*}
		By Corollary \ref{georintogeor} $f$ sends $K$-regions with vertex $\xi$ into $K$-regions with vertex $\eta$, so the right-hand side is bounded by Remark \ref{reformulation}, which proves (i).
		
		We  prove (ii).    Let $\tilde \rho\colon D\to \D$ be a left inverse of $\varphi$, and let $1\leq i\leq q-1.$
		By Lemma \ref{lemmadistanza} there exists $C_2\geq 0$ such that $\delta_D(\varphi(t))\leq C_2|1-\tilde \rho(\varphi(t))|$ for all $t\in [0,1)$, so
		$$\frac{|\langle f(\varphi(t))-\eta,u_i\rangle|}{|1-\tilde{\rho}(\varphi(t))|^{1/n_i}}\leq C_2^{1/n_i}\frac{|\langle f(\varphi(t))-\eta,u_i\rangle|}{\delta_D(\varphi(t))^{1/n_i}}\leq (C_1C_2)^{1/n_i}\frac{|\langle f(\varphi(t))-\eta,u_i\rangle|}{\delta_{D'}(f(\varphi(t)))^{1/n_i}}\stackrel{t\to 1^-}\longrightarrow 0,$$
		where we used Remark \ref{reformulation} and the fact that, thanks to  Proposition \ref{fphisigmac}, the curve $t\mapsto f(\varphi(t))$ is (strongly) restricted.
		By point (i) above the holomorphic  function $\frac{\langle f(z)-\eta,u_i\rangle}{(1-\tilde{\rho}(z))^{1/n_i}}$ is $K$-bounded, and thus the 
		Lindel\"of principle (Theorem \ref{localLindelof}) yields
		$$K'\textrm{-}\lim_{z\to\xi}\frac{\langle f(z)-\eta,u_i\rangle}{(1-\tilde{\rho}(z))^{1/n_i}}=0.$$
		By Lemma \ref{lemmadistanza} we have (ii).
		\endproof
		
		\begin{proof}[Proof of \eqref{proofKbounded} in Theorem \ref{genJWC}]
			Let $\varphi\colon \D\to D$ be a complex geodesic with endpoint $\xi$, and let $\tilde \rho\colon D\to \D$ be a left inverse of $\varphi$.
			Let  $0\leq i\leq q-1$ and  $v\in \C^d\setminus\{0\}.$
			Thanks to Lemma \ref{lemmadistanza} it is enough to prove that the function
			$$(1-\tilde \rho(z))^{\frac{1}{m_\xi(v)}-\frac{1}{n_i}}\langle df_z(v),u_i\rangle$$ is $K$-bounded.
			Let $z\in A(\varphi, R)$.
			Since $D$ is convex, by  \cite[Lemma 11.1.2, Proposition 11.1.4, Theorem 11.2.1]{JarPfl} there exists\footnote{Notice that $\kappa_\D(0,1)=2$.} a  complex geodesic  $\psi\colon\D\to D$ such that $\psi(0)=z$ and  $$\psi'(0)=2v\slash \kappa_D(z,v).$$
			Let $r\in(0,1)$, then by the Cauchy formula
			$$\langle df_z(v),u_i\rangle=\frac{\kappa_D(z,v)}{2}\langle df_z(\psi'(0)),u_i\rangle=\frac{\kappa_D(z,v)}{4\pi i}\int_{|\zeta|=r}\frac{\langle f(\psi(\zeta))-\eta,u_i \rangle}{\zeta^2}d\zeta,$$
			so
			$$(1-\tilde{\rho}(z))^{1/m_\xi(v)-1/n_i}\langle df_z(v),u_i\rangle=$$$$\frac{1}{4\pi}\int_{-\pi}^{\pi}\frac{\langle f(\psi(re^{i\theta}))-\eta,u_i \rangle}{(1-\tilde{\rho}(\psi(re^{i\theta})))^{1/n_i}}\cdot\left(\frac{1-\tilde{\rho}(\psi(re^{i\theta}))}{1-\tilde{\rho}(z)}\right)^{1/n_i}\cdot\frac{\kappa_D(z,v)(1-\tilde{\rho}(z))^{1/m_\xi(v)}}{re^{i\theta}}d\theta.$$
			Notice that, since $\psi$ is a complex geodesic, $k_D(\psi(re^{i\theta}),z)=k_\D(0,r)$, so $\psi(re^{i\theta})\in A(\varphi,R_r)$ where $R_r:=R+k_\D(0,r)$, which implies by the proof of Proposition \ref{abatek}
			$$\tilde{\rho}(\psi(re^{i\theta}))\in A^\D(\id_\D,R_r)\subset K^\D_0(1,e^{R_r}).$$
			It follows that the first and third terms are bounded respectively by Proposition \ref{fmenopifn} and Theorem \ref{KtypeLtype} combined with Lemma \ref{lemmadistanza}.
			Let us consider the second term. Using \eqref{koranyiball} we have 
			\begin{align*}\label{rhozrhopsi}\left|\frac{1-\tilde{\rho}(\psi(re^{i\theta}))}{1-\tilde{\rho}(z)}\right|&\leq e^{R_r}\frac{1-|\tilde{\rho}(\psi(re^{i\theta}))|}{1-|\tilde{\rho}(z)|}\\&\leq2e^{R_r}\exp(k_\D(0,\tilde{\rho}(z))-k_\D(0,\tilde{\rho}(\psi(re^{i\theta}))))\\&\leq 2e^{R_r+k_\D(0,r)}.\end{align*}
			hence it is bounded.
		\end{proof}
		\subsection{Second step: proof of Theorem \ref{JWC1}}
		By the previous argument for all $v\in \C^d\setminus\{0\}$ the function $ \delta_D(z)^{\frac{1}{m_\xi(v)}-1}\frac{\partial f}{\partial v}(z)$ is $K$-bounded.
		\begin{proof}[Proof of (i) in Theorem \ref{JWC1}]
			By the Lindel\"of Principle (Theorem \ref{localLindelof}) it is sufficient to prove that the limit (i)   holds along the normal segment, which by Lemma \ref{normalrestricted} is a strongly restricted curve.
			First of all we prove that $$\lim_{t\to0^+}\bar{\eta}\frac{\partial f}{\partial n_\xi}(\xi-tn_\xi)$$
			exists and it is positive. Let $B\subset D$ be an Euclidean ball internally tangent to $\partial D$ at $\xi$. By Proposition \ref{ratiodeltas+} (noticing that, on the inner normal segment at $\xi$ we have $\delta_D=\delta_B$ close to $\xi$) it follows  that the restriction $f|_B\colon B\to \D$ has a regular contact point    at $\xi$ with $K$-limit $\eta$. Applying  Rudin's Julia-Wolff-Carath\'eodory theorem to $f|_B$ we obtain the existence and positivity of the limit. Denote by $\bar{\eta}\frac{\partial f}{\partial n_\xi}(\xi)$ such limit.
			
			Now let $\varphi\colon\D\to D$ be a complex geodesic with $\varphi(0)=p$ and endpoint $\xi$, and define $\nu:=\varphi'_N(1)n_\xi$. Consider the normal segment $\sigma\colon[t_0,1)\to D$ given by $\sigma(t)=\xi+(t-1)\nu$, then by Proposition \ref{prop:JFCGeneral}, Lemma \ref{normalrestricted} and the fundamental theorem of calculus  we have
			\begin{align*}
				\lambda_{\xi,p}&=\lim_{t\to1^-}\exp(k_\D(0,t)-k_\D(f(\varphi(t)),0))=\lim_{t\to1^-}\exp(k_\D(0,t)-k_\D(f(\sigma(t)),0))\\&=\lim_{t\to1^-}\frac{1+t}{1-t}\cdot\frac{1-|f(\sigma(t))|}{1+|f(\sigma(t))|} =\lim_{t\to1^-}\frac{1-|f(\sigma(t))|}{1-t} =\lim_{t\to1^-}\frac{1}{1-t}\int_t^1\frac{{\rm Re}\left(\overline{f(\sigma(s)})\frac{\partial f}{\partial \nu}(\sigma(s))\right)}{|f(\sigma(s))|}ds\\& =\lim_{t\to1^-}\frac{\varphi_N'(1)}{1-t}\int_t^1\frac{{\rm Re}\left(\overline{f(\sigma(s)})\frac{\partial f}{\partial n_\xi}(\sigma(s))\right)}{|f(\sigma(s))|}ds=\varphi_N'(1)\bar{\eta}\frac{\partial f}{\partial n_\xi}(\xi)=\frac{1}{|\Omega_\xi(p)|}\bar{\eta}\frac{\partial f}{\partial n_\xi}(\xi).
			\end{align*}
		\end{proof}

		\begin{proof}[Proof of (ii) in  Theorem \ref{JWC1}]  
			Set $m:=m_\xi(v)$ and $s:=1/m$. First of all we prove the result in dimension 2 for $f\colon \mathbb{E}_m\to\D$ where $\mathbb{E}_m$ is the egg domain, $\xi=(1,0)\in\partial\mathbb{E}_m$ is a regular contact point and $v=(0,1)$. 
			
			We start writing  $f$ as follows:
			$$f(z_0,z_1)=f(z_0,0)+\frac{\partial f}{\partial z_1}(z_0,0)z_1+\sum_{k=2}^{\infty}\frac{1}{k!}\frac{\partial^kf}{\partial z_1^k}(z_0,0)z_1^k.$$
			Define the holomorphic function $g\colon \D\to \C$  as
			$$g(z_0)=\frac{1}{2}\frac{\partial f}{\partial z_1}(z_0,0)(1-z_0)^{s-1},$$ where we are using  the principal value of the $m$-th root. 
			We need to show $g(t)\to0$ if $t\to1^-$.
			Define the holomorphic function $h\colon \mathbb{E}_m\to \C$  as $$h(z_0,z_1)=f(z_0,0)+\frac{z_1}{2}\frac{\partial f}{\partial z_1}(z_0,0)=f(z_0,0)+z_1(1-z_0)^{1-s}g(z_0).$$
			We want to prove that $h(\mathbb{E}_m)\subseteq\D$. Fix $z_0\in\D$ and set $r:=\sqrt[m]{1-|z_0|^2}$. By the Schwarz-Pick lemma applied to $f(z_0,\cdot)\colon r\D\to \D$, we have
			$$\frac{\left|\frac{\partial f}{\partial z_1}(z_0,z_1)\right|}{1-|f(z_0,z_1)|^2}\leq\frac{r}{r^2-|z_1|^2},$$
			which implies that at $z_1=0$ we have $$\left|\frac{\partial f}{\partial z_1}(z_0,0)\right|\leq\frac{1-|f(z_0,0)|^2}{r}.$$
			Finally, for all $z_1\in r\D$,
			\begin{align*}|h(z_0,z_1)|&=\left|f(z_0,0)+\frac{z_1}{2}\frac{\partial f}{\partial z_1}(z_0,0)\right|\leq |f(z_0,0)|+\frac{|z_1|}{2}\left|\frac{\partial f}{\partial z_1}(z_0,0)\right|\\&\leq |f(z_0,0)|+\frac{|z_1|}{2r}(1-|f(z_0,0)|^2)\leq |f(z_0,0)|+\frac{1}{2}(1-|f(z_0,0)|^2)<1,\end{align*}
			so $h(\mathbb{E}_m)\subseteq\D$.
			
			Fix $a>0$, set $A:=\frac{a^2}{a^2+1}$  and define the curve
			$$\zeta_a\colon (A,1)\to \D,\quad \zeta_a(t)=t+ia(1-t).$$
			Notice that for $t\in (A,1)$ we have
			$|\zeta_a(t)|^2<t$,
			so we can find $\eta_a(t)\in\C$ is such that 
			$$1-t<|\eta_a(t)|^m< 1-|\zeta_a(t)|^2$$
			and 
			$$\eta_a(t)(1-\zeta_a(t))^{1-s}g(\zeta_a(t))\in\R_{\geq0}.$$
			Define the function 
			$$\sigma_a\colon (A,1)\to \C^2,\quad \sigma_a(t)=(\zeta_a(t),\eta_a(t)).$$ 
			Notice that $\sigma_a((A,1))\subset \mathbb{E}_m$.
			Now we  estimate 
			$\limsup_{t\to1^-}|g(\zeta_a(t))|.$
			Consider the holomorphic function $\ell\colon \D\to \D$ defined as $\ell(\zeta)=f(\zeta,0)$. The point $1\in\partial\D$ is a regular contact point for $\ell$, hence  by the one-dimensional Julia--Wolff--Carath\'eodory theorem (see e.g. \cite[Theorem 1.2.7]{Abatebook}) there exists $\alpha>0$  such that
			$$\frac{1-\ell(\zeta_a(t))}{1-\zeta_a(t)}=\alpha+o(1),$$
			and thus
			$\ell(\zeta_a(t))=1-(\alpha+o(1))(1-ia)(1-t)$. Since $h(\mathbb{E}_m)\subseteq\D$ we have
			\begin{align*}1&\geq \Re h(\sigma_a(t))=1-(\alpha+o(1))(1-t)+|\eta_a(t)||1-\zeta_a(t)|^{1-s}|g(\zeta_a(t))|\\&>1-(\alpha+o(1))(1-t)+|1-ia|^{1-s}|g(\zeta_a(t))|(1-t).\end{align*}
			So
			$$\limsup_{t\to1^-}|g(\zeta_a(t))|\leq \frac{\alpha}{|1-ia|^{1-s}}.$$ Similarly, we obtain
			$\limsup_{t\to1^-}|g(\bar\zeta_a(t))|\leq \frac{\alpha}{|1-ia|^{1-s}},$ and since  $g\colon\D\to\C$ is $K$-bounded at $1$ 
			it follows that
			$$\limsup_{t\to1^-}|g(t)|\leq \frac{\alpha}{|1-ia|^{1-s}}.$$ Letting $a\to+\infty$ we have 
			$\lim_{t\to 1^-}g(t)=0$.
			
			For the general case, by Proposition \ref{ratiodeltas+} we can reduce to the case  $d=2$ by  cutting $D$ with the complex plane $\xi+\text{span}_\C\{n_\xi,v\}$.
			Notice that the line type remains the same. Now $D$ has the following defining function near the origin $$r(z):=\Re z_0+ H(z_1)+R(z)$$
			where $H\colon\C\to\C$ is a non-negative $m$-homogeneous ($H(tz_1)=t^{m}H(z_1)$) convex polynomial that is not zero and 
			$$R(z)=o(|z_0|+|z_1|^{m}).$$
			Now we can find $\varepsilon>0$ small enough such that the egg domain
			$$E=\{z\in\C^2:|z_0+\varepsilon|^2+|z_1|^{m}<\varepsilon^2\}$$
			is contained in $D$. We conclude applying the first part of the proof to $f|_E$.
		\end{proof}
		This concludes the proof of Theorem \ref{JWC1}. Combining it with Lemma \ref{dilgeoinv}, we immediately obtain the following corollary.
		
		\begin{corollary}\label{rhoJWC}
			Let $D\subset\C^d$ be a $\C$-proper convex domain and let $\xi\in\partial D$ be a point of locally finite type.
			Let $\varphi\colon\D\to D$ be a complex geodesic with endpoint $\xi$ and let  $\tilde{\rho}\colon D\to \D$ be a  left inverse of $\varphi$. Then if $v\in \C^d$,
			$$K'\textrm{-}\lim_{z\to\xi}d\tilde{\rho}_z(v)=\langle v,n_\xi\rangle\alpha_\xi(\tilde \rho)=\langle v,n_\xi\rangle|\Omega^{D}_\xi(\varphi(0))|.$$
		\end{corollary}
		
		\subsection{Third step: proof of Theorem  \ref{genJWC}}
		We are left with proving points (i),(ii),(iii) of Theorem \ref{genJWC}. 
		\begin{proof}[Proof of (i) and (ii) in Theorem \ref{genJWC}]
			Let $\varphi\colon \D\to D$ be a complex geodesic with $\varphi(0)=p$ and endpoint $\xi$, and let $\psi\colon\D\to D'$ be a complex geodesic with $\psi(0)=p'$ and endpoint $\eta$. Let $\tilde{\rho}\colon D'\to \D$ be a  left inverse of $\psi$. 
			First of all notice that $\xi$ is a regular contact point for the function $\tilde{\rho}\circ f$, with $K$-limit $1\in \partial \D$.
			Let $0\leq i\leq d-1$. For all $t\in [0,1)$ we have
			$$\frac{\partial (\tilde{\rho}\circ f)}{\partial v_j}(\varphi(t))=d\tilde{\rho}_{f(\varphi(t))}( df_{\varphi(t)}(v_j))
			=\sum_{i=0}^{q-1} \langle df_{\varphi(t)}(v_j),u_i   \rangle d\tilde{\rho}_{f(\varphi(t))} (u_i).
			$$
			
			By  Theorem \ref{JWC1} we have
			$$\frac{\partial (\tilde{\rho}\circ f)}{\partial v_0}(\varphi(t))\stackrel{t\to 1^-}\to\alpha_\xi(\tilde{\rho}\circ f)=\alpha_\eta(\tilde{\rho})\alpha_\xi(f)=\alpha_\xi(f)|\Omega_\eta^{D'}(p')|,$$
			while for all $1\leq j\leq d-1$,
			$$\frac{\frac{\partial(\tilde\rho\circ f)}{\partial v_j}(\varphi(t))}{\delta_D(\varphi(t))^{1-\frac{1}{m_j}}}\stackrel{t\to 1^-}\to 0.$$
			Fix $1\leq i\leq q-1$. Notice that $f(\varphi(t))$ is a  restricted curve.
			We have 
			$$\frac{\langle df_{\varphi(t)}(v_j),u_i   \rangle d\tilde{\rho}_{f(\varphi(t))} (u_i)}{{\delta_D(\varphi(t))^{1-\frac{1}{m_j}}}}=\frac{\langle df_{\varphi(t)}(v_j),u_i   \rangle}{\delta_D(\varphi(t))^{\frac{1}{n_i}-\frac{1}{m_j}}}\frac{d\tilde{\rho}_{f(\varphi(t))} (u_i)}{\delta_{D'}(f(\varphi(t)))^{1-\frac{1}{n_i}}}\frac{\delta_{D'}(f(\varphi(t)))^{1-\frac{1}{n_i}}}{\delta_D(\varphi(t))^{1-\frac{1}{n_i}}}\stackrel{t\to 1^-}\longrightarrow 0,$$
			thanks to  \eqref{proofKbounded}, Theorem \ref{JWC1}, and Proposition \ref{ratiodeltas}.
			
			Moreover, $ d\tilde{\rho}_{f(\varphi(t))} (u_0)\stackrel{t\to 1^-}\to \alpha_\eta(p')= |\Omega_\eta^{D'}(p')|>0.$
			It follows that 
			$$\langle df_{\varphi(t)}(v_0),u_0  \rangle \stackrel{t\to 1^-}\to\alpha_\xi(f),$$ and for all $1\leq j\leq d-1$,
			$$\frac{\langle df_{\varphi(t)}(v_j),u_0   \rangle}{\delta_D(\varphi(t))^{1-\frac{1}{m_j}}} \stackrel{t\to 1^-}\to 0,$$
			and the Lindel\"of principle (Theorem \ref{localLindelof}) together with Lemma \ref{lemmadistanza} yield the result.

		\end{proof}

		\begin{proof}[Proof of (iii) in Theorem \ref{genJWC}] 
			Define $$\varphi_t(\zeta)=\xi-\varepsilon(1-\zeta)(1-t)n_\xi.$$
			Notice that there exists $\varepsilon>0$ such that $\varphi_t(\D)\subset D'$ for all $t\in(0,1)$. We want to show that if we set $\sigma(t):=\varphi_t(0)=\xi-\varepsilon(1-t)n_\xi$, we have
			$$\lim_{t\to 1^-}(1-t)^{1-1/n_i}\langle df_{\sigma(t)}(n_\xi),u_i\rangle=0.$$
			Fix $r\in(0,1)$. Notice that for all $\theta$ the curve $t\mapsto\varphi_t(re^{i\theta})$ is contained in the cone 
			with vertex $\xi$, axis $ n_\xi$ and amplitude $2\arcsin(r)<\pi$.
			By \cite[Proposition 2.4]{Mercer}
			\begin{equation}\label{ultimosforzo}
				\left|\log\frac{\delta_D(\varphi_t(re^{i\theta}))}{\delta_D(\varphi_t(0))}\right|\leq k_D(\varphi_t(re^{i\theta}),\varphi_t(0))\leq k_\D(0,r).
			\end{equation}
			Using the Cauchy formula we have 
			\begin{align*}(1-t)^{1-1/n_i}\langle df_{\sigma(t)}(n_\xi),u_i\rangle&=\frac{1}{\varepsilon(1-t)^{1/n_i}}\langle df_{\varphi_t(0)}(\varphi_t'(0)),u_i\rangle\\&=\frac{1}{2\varepsilon\pi}\int_{-\pi}^{\pi}\frac{\langle f(\varphi_t(re^{i\theta}))-\eta,u_i\rangle}{\delta_D(\varphi_t(re^{i\theta}))^{1/n_i}}\cdot\left(\frac{\delta_D(\varphi_t(re^{i\theta}))}{(1-t)}\right)^{1/n_i}\cdot\frac{1}{ re^{i\theta}}d\theta.
			\end{align*}
			Notice that the first factor in the integral is bounded by (i) in Proposition \ref{fmenopifn}. The second factor  is bounded by (\refeq{ultimosforzo}) since  $\delta_D(\varphi_t(0))=\varepsilon(1-t)$. Finally, by (ii) in Proposition \ref{fmenopifn} the first factor converges pointwise to $0$ as $t\to 1^-$, so we can conclude by the dominated convergence theorem.
		\end{proof}
		This ends the proof of Theorem \ref{genJWC}.
 
		\subsection{Examples and consequences}
		\begin{example}\label{exJWC}
			We now give some examples  that show that our results are sharp for holomorphic maps between egg domains. Let $m\geq2$ be an even integer and let $\mathbb{E}_m$ be the egg domain (see (\refeq{eggdomain})).
			Now consider the holomorphic function
			$\vartheta \colon\D\to\D$ given by
			$$\vartheta(\zeta)=\exp\left(-\frac{\pi}{2}-i\log(1-\zeta)\right).$$
			Notice that as $\zeta\to1$ the function $\vartheta(\zeta)$ spirals around the origin without limit.
			Moreover
			$$\vartheta'(z)=\frac{i}{1-\zeta}\vartheta(\zeta).$$
			The function $\varphi\colon \D\to \mathbb{E}_m$ given by $\varphi(\zeta)=(\zeta,0)$ is a complex geodesic with endpoint $\xi:=(1,0)$ and the function $\tilde \rho\colon \mathbb{E}_m\to\D$ given by $\tilde \rho(z_0,z_1)=z_0$ is a left inverse of $\varphi$, so by Lemma \ref{lemmadistanza} the function $\delta_{\mathbb{E}_m}(z)/(1-z_0)$ and its reciprocal are $K$-bounded at $\xi$.
			Notice also the quantity $\frac{z_1^{m}}{1-z_0}$ is $K$-bounded at $\xi$ by Proposition \ref{fmenopifn}. 
			
			Let $f\colon \mathbb{E}_{m}\to\D$ be given by
			$$f(z)=z_0+\frac{1}{2}z_1^{m}\vartheta(z_0).$$
			Notice that $\xi$ is a regular contact point for $f$.
			Then
			$$\frac{\partial f}{\partial z_0}(z)=1+\frac{i}{2}\frac{z_1^{m}}{1-z_0}\vartheta(z_0)$$
			and
			$$(1-z_0)^{\frac{1}{m}-1}\frac{\partial f}{\partial z_1}(z)=\frac{m}{2}\left(\frac{z_1^{m}}{1-z_0}\right)^\frac{m-1}{m}\vartheta(z_0).$$
			The two functions are $K$-bounded at $\xi$ with sharp exponents and  have restricted $K$-limit at $\xi$, but they do not have $K$-limit at $\xi$. Indeed they do not have limit along the curve $\gamma_\lambda(t)=(t,\lambda\sqrt[m]{1-t^2})$, with $\lambda\in\D^*$, which $K$-converges to $\xi$ but does not $K'$-converge to $\xi$.
			
			Now we discuss two examples of holomorphic mappings $f\colon \mathbb{E}_{m_1}\to \mathbb{E}_{m_2}$ depending on the  order relation between $m_1$ and $m_2$.
			\begin{itemize}
				\item Let $2\leq m_1\leq m_2$. Let $f\colon \mathbb{E}_{m_1}\to \mathbb{E}_{m_2}$ be given by
				$$f(z)=\left(z_0,\frac{1}{2^{1/m_1}}\frac{z_1}{(1-z_0)^\frac{m_2-m_1}{m_1m_2}}\vartheta(z_0)\right).$$
				Then
				$$(1-z_0)^{1-\frac{1}{m_2}}\frac{\partial f_2}{\partial z_0}(z)=\frac{1}{2^{1/m_1}}\left(\frac{m_2-m_1}{m_1m_2}+i\right)\left(\frac{z_1^{m_1}}{1-z_0}\right)^\frac{1}{m_1}\vartheta(z_0),$$
				$$(1-z_0)^{\frac{1}{m_1}-\frac{1}{m_2}}\frac{\partial f_2}{\partial z_1}(z)=\frac{1}{2^{1/m_1}}\vartheta(z_0).$$
				\item Let $2\leq m_2\leq m_1$ and fix $r>0$. Let $f\colon \mathbb{E}_{m_1}\to \mathbb{E}_{m_2}$ be given by
				$$f(z)=\left(\frac{z_0+r}{1+r},\frac{1}{2^{1/m_1}}\frac{r^{1/m_2}}{(1+r)^{2/m_2}}z_1(1-z_0)^\frac{m_1-m_2}{m_1m_2}\vartheta(z_0)\right).$$
				Then
				$$(1-z_0)^{1-\frac{1}{m_2}}\frac{\partial f_2}{\partial z_0}(z)=\frac{1}{2^{1/m_1}}\frac{r^{1/m_2}}{(1+r)^{2/m_2}}\left(\frac{m_2-m_1}{m_1m_2}+i\right)\left(\frac{z_1^{m_1}}{1-z_0}\right)^\frac{1}{m_1}\vartheta(z_0),$$
				$$(1-z_0)^{\frac{1}{m_1}-\frac{1}{m_2}}\frac{\partial f_2}{\partial z_1}(z)=\frac{1}{2^{1/m_1}}\frac{r^{1/m_2}}{(1+r)^{2/m_2}}\vartheta(z_0).$$
			\end{itemize}	
			In both cases, $\xi$ is a regular contact point. Notice that for all the functions we have sharp exponents for the $K$-boundedness. Moreover, $(1-z_0)^{1-\frac{1}{m_2}}\frac{\partial f_2}{\partial z_0}(z)$ has restricted $K$-limit at $\xi$ but  does not  have $K$-limit at $\xi$.
			Finally, $(1-z_0)^{\frac{1}{m_1}-\frac{1}{m_2}}\frac{\partial f_2}{\partial z_1}(z)$ does not  even have radial limit at $\xi$.
			
			Finally, the same example shows that formula \eqref{formulafalsa}  is false.
			Set $\eta:=(1,0)\in\partial \mathbb{E}_{m_2}$ and consider $u\in\C^2\backslash\{0\}$ such that $\langle u,n_\eta\rangle\neq0$ and $u\notin T_\eta^\C\partial \mathbb{E}_{m_2}$. Notice that $m_\eta(u)=1$ and $M_\eta(v)=m_2$. Then
			$\langle df_z(v), u\rangle$ is 
			$O_K\left((1-z_0)^{\frac{1}{m_2}-\frac{1}{m_\xi(v)}}\right)$ but not 
			$O_K\left((1-z_0)^{1-\frac{1}{m_\xi(v)}}\right)$.
		\end{example}

		We conclude with some consequences. The first is an asymptotic estimate, depending only on the multitypes at $\xi$ and $\eta$, of the complex Jacobian  of $f$ when the dimension of $D$ is equal to the dimension of $D'$. This generalizes the result \cite[Corollary, p.178]{Rudin}, which shows that if $D$ and $D'$ are balls, then the complex Jacobian  of $f$ is $K$-bounded.  \begin{corollary}
			Let $D,D'\subset\C^d$ be  $\C$-proper convex domains, and let $\xi\in\partial D,\eta\in \partial D'$ be points of locally finite type with multitype respectively $(m_j)_{j=0}^{d-1}$ and $(n_j)_{j=0}^{d-1}$.  Let $f\colon D\to D'$ be a holomorphic map, and 
			assume that $\xi$ is a regular contact point with   $K\textrm{-}\lim_{z\to\xi}f(z)=\eta$.
			Then $$\det df_z=O_K\left(\delta_D(z)^{\sum_{j=0}^{d-1}\frac{1}{n_j}-\frac{1}{m_j}}    \right).$$
		\end{corollary}
		\begin{proof}
			Let $(v_j)_{j=0}^{d-1}$ (resp. $(u_j)_{j=0}^{d-1}$) be an orthonormal  multitype basis at $\xi$ (resp. at $\eta$). 
			The result follows  computing the determinant of the Jacobian matrix of $f$ using Theorem \ref{genJWC}.
		\end{proof}
		\begin{remark}
			If $H$ denotes the harmonic mean, then clearly
			$$\sum_{j=0}^{d-1}\frac{1}{n_j}-\frac{1}{m_j}=\frac{d}{H(n_0,\dots, n_{d-1})}-\frac{d}{H(m_0,\dots, m_{d-1})}.$$
			In particular, if the harmonic mean of the multitype at $\eta$ is strictly smaller than the harmonic mean of the multitype at $\xi$, then 
			$$K\textrm{-}\lim_{z\to \xi} \det df_z=0$$

		\end{remark}
		
		Theorem \ref{genJWC} has an interesting consequence for self-maps.
		\begin{corollary}
			Let $D\subset\C^d$ be a $\C$-proper convex domain and $\xi\in\partial D$ a point of locally finite type  and let $f\colon D\to D$ be a holomorphic self-map.
			If $\xi$ is a regular contact point with $K\textrm{-}\lim_{z\to\xi}f(z)=\xi$,
			then for all $p\in D$
			$$K'\textrm{-}\lim_{z\to\xi}\langle df_z(n_\xi),n_\xi\rangle=\lambda_{\xi,p,p}(f).$$
		\end{corollary}

		We now prove a version of the Julia Lemma for  the pluricomplex Poisson kernel. 
		\begin{proposition}
			Let $D\subset\C^d$ and $D'\subset\C^{q}$ be  $\C$-proper convex domains, and  let $\xi\in \partial D,\eta\in \partial D' $ be  points of locally finite type.
			Let $f\colon D\to D'$ be a holomorphic map. \begin{itemize}
				\item[(i)] If the function $$z\mapsto \frac{\Omega^D_\xi(z)}{\Omega^{D'}_\eta(f(z))}$$ is bounded from above, then 
				$\xi$ is a regular contact point for $f$ with  $K\textrm{-}\lim_{z\to\xi}f(z)=\eta$.
				\item[(ii)] If  $\xi$ is a regular contact point for $f$ with  $K\textrm{-}\lim_{z\to\xi}f(z)=\eta$, then
				$$\sup_{z\in D}\frac{\Omega^D_\xi(z)}{\Omega^{D'}_\eta(f(z))}=\alpha_\xi=K'\textrm{-}\lim_{z\to\xi}\langle df_z(n_\xi),n_\eta\rangle.$$ 
			\end{itemize}
		\end{proposition}
		\begin{proof}
Point (i) follows from the implication $(3)\Rightarrow(4)$ in the Julia Lemma (Theorem \ref{JuliaLemma}) and from \eqref{orosferepoisson}.
			We prove (ii). Fix $p\in D,p'\in D'$.
			From  \eqref{conilsup} and  \eqref{orosferepoisson} we obtain
			$$\sup_{z\in D}\left(\log|\Omega^{D'}_\eta(p')|-\log|\Omega^{D'}_\eta(f(z))|- \log|\Omega^{D}_\xi(p)|+\log|\Omega^{D}_\xi(z)|\right)=\log\lambda_{\xi,p,p'},$$
			and the result immediately follows.
		\end{proof}
		\begin{remark}
			Point (ii) of the previous proposition  yields an interesting estimate from below of $K'\textrm{-}\lim_{z\to\xi}\langle df_z(n_\xi),n_\eta\rangle$.
			As an example, if $f\colon\mathbb{E}_m\to\mathbb{E}_m$ if a holomorphic self-map such that $\xi=(1,0)$ is a regular contact point with
			$K\textrm{-}\lim_{z\to\xi}f(z)=\xi$,  then by Example \ref{poisegg} if $f(0)=(z_0,z_1)$ is the image of the origin, we have
			$$K'\textrm{-}\lim_{z\to\xi}\langle df_z(e_0),e_0\rangle \geq\frac{|1-z_0|^2}{1-|z_0|^2-|z_1|^m}.$$
		\end{remark}

		\begin{corollary}
			Let $D,D'\subset \C^d$ be  $\C$-proper convex domains, and let $\xi\in\partial D,\eta\in \partial D' $ be  points of locally finite type. Let $f\colon D\to D'$ be a biholomorphism and assume that there exists a sequence $(z_n)$ in $D$ which is $K$-converging to $\xi$ such that $f(z_n)\to \eta$. Then for all $w\in D$
			$$\frac{\Omega_\eta(w)}{\Omega_\xi(f(w))}=\alpha_\xi=K'\textrm{-}\lim_{z\to\xi}\langle df_z(n_\xi),n_\eta\rangle.$$
		\end{corollary}
		\begin{proof}
			It is enough to notice that, since $f$ is a biholomorphism,  
			$$\log\lambda_{\xi,w,f(w)}=\liminf_{z\to\xi} k_D(z,w)-k_{D'}(f(z),f(w))=0$$  for all $w\in D$.
		\end{proof}
		
		We conclude with two extrinsic characterizations of regular contact points (notice the analogy with \cite[Propositions 1.2.6, 1.2.8]{Abatebook}).
		
		\begin{corollary}\label{critcontactfun}
			Let $D\subset\C^d$ be a $\C$-proper convex domain and $\xi\in\partial D$ be a point of locally finite type. Let
			$f\colon D\to\D$ be a holomorphic map, then $\xi$ is a regular contact point if and only if
			\begin{equation}\label{critcontactfun1}
				\lim_{t\to0^+}|f(\xi-tn_\xi)|=1
			\end{equation}
			and
			\begin{equation}\label{critcontactfun2}
				\limsup_{t\to0^+}\left|\frac{\partial f}{\partial n_\xi}(\xi-tn_\xi)\right|<\infty.	
			\end{equation}
		\end{corollary}
		\proof 
		Assume that $\xi$ is a regular contact point for $f$, then (\refeq{critcontactfun1}) follows from the definition while (\refeq{critcontactfun2}) follows from  Theorem \ref{JWC1}.
		
		Conversely, assume (\refeq{critcontactfun1}) and (\refeq{critcontactfun2}).
		Let $t_0>0$ and $M> 0$ such that for all $0\leq t\leq t_0$ we have $\left|\frac{\partial f}{\partial n_\xi}(\xi-tn_\xi)\right|\leq M.$ 
		Up to taking a smaller $t_0$ we have  for all $0\leq t\leq t_0$
		\begin{align*}\frac{\delta_\D(f(\xi-tn_\xi))}{t}&=\frac{1-|f(\xi-tn_\xi)|}{t}\\&=\frac{1}{t}\int_0^{t}\frac{\Re(\overline{f(\xi-sn_\xi)}\frac{\partial f}{\partial n_\xi}(\xi-sn_\xi))}{|f(\xi-sn_\xi)|}ds\\&\leq\frac{1}{t}\int_0^{t}\left|\frac{\partial f}{\partial n_\xi}(\xi-sn_\xi)\right|ds \leq M,
		\end{align*}
		which implies by Proposition \ref{ratiodeltas+} that $\xi$ is a regular contact point.

	\begin{corollary}\label{critcontactmap}
		Let $D\subset\C^d$ and $D'\subset\C^{q}$ be  $\C$-proper convex domains and let $\xi\in\partial D,\eta\in \partial D'$ be  points of locally finite type. Let $f\colon D\to D'$ be a holomorphic map, then $\xi$ is a regular contact point  with  $K\textrm{-}\lim_{z\to\xi}f(z)=\eta$  if and only if
		\begin{equation}\label{critcontactmap1}
			\lim_{t\to0^+}f(\xi-tn_\xi)=\eta
		\end{equation}
		and
		\begin{equation}\label{critcontactmap2}
			\limsup_{t\to0^+}|\langle df_{\xi-tn_\xi}(n_\xi),n_\eta\rangle|<\infty.
		\end{equation}
	\end{corollary}
	\proof 
	Assume that $\xi$ is a regular contact point with  $K\textrm{-}\lim_{z\to\xi}f(z)=\eta$ . 
	Then \eqref{critcontactmap1} is obvious and \eqref{critcontactmap2} follows from  (i) in Theorem \ref{genJWC}.

	Conversely, assume (\refeq{critcontactmap1}) and (\refeq{critcontactmap2}). Assume $\eta=0$, $n_\eta=e_0$.
	Let $\pi\colon\C^{q}\to\C$ be the projection on the first component. By the convexity of $D'$ we have $\pi(D')\subseteq\H$. Notice that
	$$\frac{\partial(\pi\circ f)}{\partial n_\xi}(z)=\langle df_z(n_\xi),n_\eta\rangle$$
	so by Corollary \ref{critcontactfun} $\xi$ is a regular contact point for the function $\mathscr{C}^{-1}\circ\pi\circ f\colon D\to\D$ (and thus also for $\pi\circ f\colon D\to\H$). Now fix $p\in D$ and $p'\in D'$, then 	$$\liminf_{t\to0^+}k_D(p,\xi-tn_\xi)-k_{D'}(p',f(\xi-tn_\xi))\leq\liminf_{t\to0^+} k_D(p,\xi-tn_\xi)-k_{\H}(\pi(p'),\pi(f(\xi-tn_\xi)))<+\infty ,$$
	hence $\lambda_{\xi,p,p'}(f)<+\infty$. Finally,  the Julia Lemma (Proposition \ref{JuliaLemma}) yields the result.
	\endproof

	Our results also give us information at points of infinite type.
	\begin{corollary}
		Let $D\subset\C^d$ be  a $\C$-proper convex domain, and  let $\xi\in\partial D$ be a smooth point of infinite line type. Let $f\colon D\to \D$ be a holomorphic function such that
		$$\lim_{t\to0^+}|f(\xi-tn_\xi)|=1$$
		and
		$$\limsup_{t\to0^+}\left|\frac{\partial f}{\partial n_\xi}(\xi-tn_\xi)\right|<\infty.	$$
		Let  $v\in\C^d$ with $m_\xi(v)=+\infty$, then for all $s<1$ and for all sequences $(z_n)$ converging non-tangentially to $\xi$ we have
		$$\frac{\frac{\partial f}{\partial v}(z_n)}{\delta_D(z_n)^s}\stackrel{n\to+\infty} \longrightarrow 0.$$
	\end{corollary}
	\proof
	First of all assume that $\xi$ is the origin and $n_\xi=e_0$. Let $r\colon \C^d\to\R$ be a smooth convex defining function of $D$ (for instance, the signed distance function), and define 
	$$r_{L}(z):=r(z)+\sum_{j=1}^{d-1}|z_j|^L$$
	where $L$ is an even integer. Notice that the origin is a point of locally finite type for $D_L:=\{r_L<0\}$ with line type $L$. Moreover, since $r\leq r_L$ we have $D_L\subset D$.
	Now for all $v\in\C^d$, by definition we have
	$$m^{D_L}_\xi(v)=\min\{m^D_\xi(v),L\}.$$
	By Corollary \ref{critcontactfun} the origin is a regular contact point for the map $f|_{D_L}:D_L\to \D$. Then by Theorem \ref{JWC1} and recalling that a sequences which converges non-tangentially to $\xi$ is contained in a $K$-region of $D_L$ (Corollary \ref{rk->nt}), we have that for all sequence $(z_n)$ converging non-tangentially to $\xi$
	$$\frac{\partial f}{\partial v}(z_n) \delta_D(z_n)^{\frac{1}{\min\{m_\xi(v),L\}}-1}$$
	is bounded.
	We conclude letting $L\to+\infty$.
	\endproof
	
	\textbf{Acknowledgement.} We want to thank the anonymous referee for the careful reading of the paper and for many useful comments.


\begin{thebibliography}{88} 
		
		
		
		\bibitem{Abatebook} M. Abate, {\sl Iteration theory of holomorphic maps on taut manifolds}, Research and Lecture Notes in Mathematics. Complex Analysis and Geometry. Mediterranean Press, Rende (1989)
		
		\bibitem{Abatelindelof} M. Abate, {\sl The Lindel\"of principle and the angular derivative in strongly convex domains}, J. Anal. Math. {\bf 54} (1990), 189--228.
		
		\bibitem{AbateSymp} M. Abate, {\sl Angular derivatives in strongly pseudoconvex domains}, Proc. Symp. Pure Math. {\bf 52}, Part 2 (1991), 23--40.
		
		\bibitem{Abatepoly} M. Abate {\sl The Julia-Wolff-Carath\'eodory theorem in polydiscs}, J. Anal. Math. {\bf 74} (1998), 275--306.
		
		\bibitem{Absurv} M. Abate, {\sl Angular derivatives in several complex variables} In ``Real methods in complex and CR geometry", Eds. D. Zaitsev, G. Zampieri, Lect. Notes in Math. 1848, Springer, Berlin, 2004, pp. 1--47.
		
		\bibitem{Abateeglistesso} M. Abate, {\sl Common fixed points of commuting holomorphic maps}, Math. Ann. {\bf 283} (1989), 645--655.
		
		\bibitem{AbateRaissy} M. Abate, J. Raissy, {\sl A Julia-Wolff-Carath\'eodory theorem for infinitesimal generators in the unit ball}, Trans. Am. Math. Soc. {\bf 368}(8) (2016), 5415--5431.
		
		\bibitem{AbTau} M. Abate, R. Tauraso, {\sl The Lindel\"of principle and angular derivatives in convex domains of finite type}, J. Aust. Math. Soc. {\bf 73}(2) (2002), 221--250.
		
		\bibitem{AAG} A. Altavilla, L. Arosio, L. Guerini, {\sl Canonical models on strongly convex domains via the squeezing function}, J. Geom. Anal. {\bf 31} (2021), 4661--4702.	
		
		\bibitem{toappear} L. Arosio, F. Bracci, M. Fiacchi, {\sl The pluricomplex Poisson kernel for  convex domains of finite type}, in preparation.
		
		\bibitem{ADAF} L. Arosio, G.M. Dall'Ara, M. Fiacchi, {\sl Worm domains are not Gromov hyperbolic}, J. Geom. Anal. {\bf 33} (2023),  257.
		
		\bibitem{AFGG} L. Arosio, M. Fiacchi, S. Gontard, L. Guerini, {\sl The horofunction boundary of a Gromov hyperbolic space}, Math. Ann. {\bf 388} (2024), 1163--1204.	
		
		\bibitem{AFGK} L. Arosio, M. Fiacchi, L. Guerini, A. Karlsson, {\sl Backward dynamics of non-expanding maps in Gromov hyperbolic metric spaces}, Adv. Math. {\bf 439} (2024), 109484.
		
		
		\bibitem{BaBo} Z. M. Balogh, M. Bonk, {\sl Gromov hyperbolicity and the Kobayashi metric on strictly pseudoconvex domains}, Comment. Math. Helv. {\bf 75} (2000), 504--533.
		
		\bibitem{Barth} T. J. Barth, {\sl Convex domains and Kobayashi hyperbolicity}, Proc. Amer. Math. Soc. {\bf 79} (1980), 556--558. 
		
		\bibitem{BZ} G. Bharali, A. Zimmer, {\sl Goldilocks domains, a weak notion of visibility, and applications}, Advances in Mathematics {\bf 310} (2017), 377--425.
		
		\bibitem{BS}  H.P. Boas, E.J. Straube, {\sl On equality of line type and variety type of real hypersurfaces in $\C^n$}, J. Geom. Anal. {\bf 2} (1992), 95--98. 
		
		\bibitem{BH} M. Bridson, A. Haefliger, {\sl Metric spaces of nonpositive curvature}, Grundlehren der Mathematischen Wissenschaften [Fundamental Principles of Mathematical Sciences] {\bf 319}, Springer-Verlag, Berlin (1999).
		
		\bibitem{BCD} F. Bracci, M. D. Contreras, S. Diaz-Madrigal, {\sl Pluripotential theory, semigroups and boundary behavior of infinitesimal generators in strongly convex domains}, J. Eur. Math. Soc. {\bf 12}(1) (2010), 23--53.
		
		\bibitem{BGZ} F. Bracci, H. Gaussier, A. Zimmer, {\sl Homeomorphic extension of quasi-isometries for convex domains in $\C^d$ and iteration theory}, Math. Ann. {\bf 379} (2021), 691--718.
		
		\bibitem{BNT} F. Bracci, N. Nikolov, P. J. Thomas, {\sl Visibility of Kobayashi geodesics in convex domains and related properties}, Math. Z. {\bf 301}(2) (2022),  2011--2035.
		
		
		\bibitem{BP} F. Bracci, G. Patrizio, {\sl Monge-Ampère foliations with singularities at the boundary of strongly convex domains}, Math. Ann. {\bf 332}(3) (2005), 499--522.
		
		\bibitem{BPT}  F. Bracci, G. Patrizio and S. Trapani, {\sl The pluricomplex Poisson kernel for strongly convex domains}, Trans. Amer. Math. Soc.  {\bf 361}(2) (2009), 979--1005.
		
		\bibitem{Bracci-Saracco} F. Bracci, A. Saracco, {\sl Hyperbolicity in unbounded convex domains}, Forum. Math. {\bf 21}(5) (2009), 815--826.
		
		\bibitem{BST} F. Bracci, A. Saracco, S. Trapani, {\sl The pluricomplex Poisson kernel for strongly pseudoconvex domains}, Adv. Math. {\bf 380} (2021), 107577.
		
		\bibitem{BrSho} F. Bracci, D. Shoikhet, {\sl Boundary behavior of infinitesimal generators in the unit ball},  Trans. Amer. Math. Soc. {\bf 366}(2) (2014), 1119--1140.
		
		\bibitem{Cat} D. Catlin, {\sl Boundary Invariants of Pseudoconvex Domains}. Ann.  Math. {\bf 120}(3) (1984),  529--586.
		
		\bibitem{Catlin} D. Catlin, {\sl Subelliptic estimates for the $\bar\partial$-Neumann problem on pseudoconvex domains}, Ann. Math. {\bf 126}(1) (1987), 131--191.
		
		\bibitem{Carath}  C. Carath\'eodory, {\sl \"Uber die Winkelderivierten von beschr\"ankten analytischen Funktionen}, Sitzungsber, Preuss. Akad. Wiss. Berlin (1929), 39--54.
		
		
		\bibitem{cirka} E.M. \v{C}irca, {\sl The Lindel\"of and Fatou theorems in $\C^n$}, Math. USSR, Sb. {\bf 21} (1973), 619--641.
		
		\bibitem{CDP} M. Coornaert, T. Delzant, A. Papadopoulos {\sl G\'eom\'etrie et th\'eorie des groupes}, Lecture Notes in Mathematics {\bf 1441} (1990).
		
		\bibitem{DAngelopaper} J. P. D'Angelo, {\sl Real hypersurfaces, orders of contact, and applications}, Ann. Math. {\bf 115}(3) (1982), 615--637.
		
		\bibitem{DAngelo} J.P. D'Angelo, {\sl Several Complex Variables and the Geometry of Real Hypersurfaces}, (2019), 10.1201/9780203739785. 
		
		\bibitem{Diederich-Fornaess} K. Diederich, J. Fornaess, {\sl Pseudoconvex domains with real analytic boundary}, Ann. Math. {\bf 107}(3) (1978), 371--384.
		
		\bibitem{Fia} M. Fiacchi, {\sl Gromov hyperbolicity of pseudoconvex finite type domains in $\mathbb{C}^2$}, Math. Ann. {\bf 382} (2022), 37--68.
		
		\bibitem{Gauss} H. Gaussier, {\sl Characterization of convex domains with noncompact automorphism group}, Mich. Math. J. {\bf 44}(2) (1997), 375--388.
		
		\bibitem{GS}  H. Gaussier, H. Seshadri, {\sl On the Gromov hyperbolicity of convex domains in $\C^n$}, Comput. Methods Funct. Theory {\bf 18}(4) (2018), 617--641.
		
		\bibitem{GausZim} H. Gaussier, A. Zimmer, {\sl The space of convex domains in complex Euclidean space}, Journal of Geometric Analysis {\bf 30} (2020), 1312--1358.
		
		\bibitem{Grah} I. Graham, {\sl Sharp constants for the Koebe theorem and for estimates of intrinsic metrics on convex domains}, Several Complex Variables and Complex Geometry, Part 2, Santa Cruz, CA, Proc. Sympos. Pure Math., vol. 52, Amer. Math. Soc., Providence, RI, 1991, pp. 233--238.
		
		\bibitem{Harris}
		L. Harris, {\sl  Schwarz-Pick systems of pseudometrics for domains in normed linear spaces}, Advances in holomorphy, Notas de Matematica 65, North-Holland, Amsterdam, 1979, pp. 345--406.
		
		\bibitem{Huang} X. Huang {\sl A Boundary Rigidity Problem For Holomorphic Mappings on Some Weakly Pseudoconvex Domains}, Canad. J. Math. {\bf 47} (1995), 405--420.
		
		\bibitem{JarPfl} M. Jarnicki, P. Pflug, {\sl Invariant Distances and Metrics in Complex Analysis}, Berlin, Boston: De Gruyter, 2013.
		
		\bibitem{JPZ} M. Jarnicki, P. Pflug, R. Zeinstra, {\sl Geodesics for convex complex ellipsoids}, Ann. Scuola Norm. Sup. Pisa, Serie 4,  {\bf 20}(4) (1993) 535--543.
		
		\bibitem{Julia} G. Julia, {\sl Extension nouvelle d'un lemme de Schwarz}, Acta Math. {\bf 42}  (1920), 349--355.
		
		\bibitem{kobayascione} S. Kobayashi, {\sl Hyperbolic Complex Spaces}, Grundlehren der Mathematischen Wissenschaften [Fundamental Principles of Mathematical Sciences] {\bf 318}, Springer-Verlag, Berlin (1998).
		
		\bibitem{Kohn} J. Kohn, {\sl Subellipticity of the $\bar \partial$-Neumann problem on pseudo-convex domains: sufficient conditions}, Acta Math., {\bf 142}(1-2)  (1979), 79--122.
		
		\bibitem{Kor} A. Kor\'anyi, {\sl Harmonic functions on hermitian hyperbolic spaces}. Trans. Am. Math. Soc. {\bf 135} (1969), 507--516.
		
		
		\bibitem{KorSte}  A. Kor\'anyi, E.M. Stein, {\sl Fatou’s theorem for generalized halfplanes}, Ann. Sc. Norm. Super. Pisa {\bf 22} (1968), 107--112.

		\bibitem{KNO} Ł. Kosiński, N. Nikolov, A. Y. Ökten, {\sl Precise estimates of invariant distances on strongly pseudoconvex domains}, Adv. Math. {\bf 478} (2025).
		
		\bibitem{krantz} S.G. Krantz, {\sl  Invariant metrics and the boundary behavior of holomorphic functions
			on domains in $\C^n$}, J. Geom. Anal. {\bf 1} (1991), 71--97.
		
		\bibitem{Lee} L. Lee, {\sl Asymptotic behavior of the Kobayashi metric on convex domains}, Pacific Journal of Mathematics, {\bf238}(1) (2009).
		
		\bibitem{mcneal}
		J.D. McNeal, {\sl Convex domains of finite type}, J. Funct. Anal. {\bf 108}(2) (1992),  361--373.
		
		\bibitem{Mercer} P. R. Mercer, {\sl  Complex geodesics and iterates of holomorphic maps on convex domains in $\C^n$}, Trans.  Amer. Math. Soc. {\bf 338}(1) (1993), 201--211. 
		
		\bibitem{NOT} N. Nikolov, A. Y. \"Okten, P. J. Thomas, {\sl Local and global notions of visibility with respect to Kobayashi distance, a comparison}, Ann. Polon. Math. {\bf 132}(2) (2024), 169--185.
		
		
		\bibitem{NPZ} N. Nikolov, P. Pflug, W. Zwonek, {\sl Estimates for invariant metrics on $\C$-convex domains}, {\bf 363}(12) (2011): 6245–56.
		
		\bibitem{Raissy} J. Raissy, {\sl The Julia-Wolff-Carathéodory theorem and its generalizations},  Complex Analysis and Geometry, Springer Proceedings in Mathematics and Statistics {\bf 144} (2015),  281--293.
		
		\bibitem{Range} R. M. Range, \textit{Holomorphic Functions and Integral Representations in Several Complex Variables}, GTM, volume 108, Springer-Verlag, New York, 1986.
		
		\bibitem{RoyLocL} H. Royden, {\sl Remarks on the Kobayashi metric}, Several Complex Variables II. Lecture Notes in Math. {\bf 185}, Springer (1971), 125--137.
		
		\bibitem{Rudin}	W. Rudin, {\sl Function theory in the unit ball of $\C^n$}, Springer, Berlin, 1980.
		
		\bibitem{Stein} E. M.  Stein, {\sl Boundary Behavior of Holomorphic Functions of Several Complex Variables}, Princeton
		University Press, Princeton (1972).
		
		\bibitem{Yu} J. Yu, {\sl Multitypes of convex domains}, Indiana Univ. Math. J. {\bf 41}(3) (1992), 837--849.
		
		\bibitem{Wolff} J. Wolff, {\sl Sur une g\'en\'eralisation d’un th\'eorème de Schwarz}, C.R. Acad. Sci. Paris {\bf 183} (1926), 500--502.
		
		
		\bibitem{Zim} A. Zimmer, {\sl Gromov hyperbolicity and the Kobayashi metric on convex domains of finite type}, Math. Ann. {\bf 365}(3-4) (2016), 1425--1498.
		
		\bibitem{Zimsub} A. Zimmer, {\sl Subelliptic estimates from Gromov hyperbolicity},
		Adv. Math. {\bf 402} (2022), 108334.
		
	\end{thebibliography}
\end{document}